\documentclass[10pt]{scrartcl}
\usepackage[english]{babel}
\usepackage{tikz}
\usepackage[above, below,section]{placeins}

\usepackage{url}

\usepackage{xcolor}

\usepackage{a4wide}

\usepackage{amsmath,amssymb,mathrsfs,bbm,cases}
\usepackage{mathtools}
\usepackage{amsthm}

\usepackage[numbers,sort]{natbib}

\usepackage{graphicx}
\usepackage{subcaption}

\usepackage{enumerate}

\usepackage{hyperref}

\usepackage{appendix}

\allowdisplaybreaks

\theoremstyle{plain}
\newtheorem{thm}{Theorem}[section]
\newtheorem{prop}[thm]{Proposition}
\newtheorem{lem}[thm]{Lemma}
\newtheorem{cor}[thm]{Corollary}

\theoremstyle{definition}
\newtheorem{ex}[thm]{Example}

\theoremstyle{remark}
\newtheorem{rem}[thm]{Remark}

\newcommand{\R}{\mathbbm{R}}

\newcommand{\N}{\mathbbm{N}}

\newcommand{\F}{\mathscr{F}}

\renewcommand{\P}{\mathbf{P}}

\newcommand{\E}{\mathbf{E}}

\newcommand{\Cov}{\mathrm{Cov}}

\newcommand{\1}{\mathbf{1}}

\renewcommand{\O}{\mathcal O}

\newcommand{\e}{\mathrm{e}}

\newcommand{\pen}{\mathrm{pen}}

\newcommand{\eps}{\varepsilon}

\newcommand{\vt}{\vartheta}

\DeclareMathOperator*{\argmax}{arg\,max} 
\DeclareMathOperator*{\argmin}{arg\,min}

\title{Nonparametric calibration for stochastic reaction-diffusion equations based on discrete observations}
\author{Florian Hildebrandt and Mathias Trabs\footnote{M.T. acknowledges financial support by DFG via the Heisenberg grant TR 1349/4-1.}}
\date{Universit\"{a}t Hamburg and Karlsruhe Institute of Technology}

\begin{document}
\maketitle
\begin{abstract}
Nonparametric estimation for semilinear SPDEs, namely stochastic reaction-diffusion equations in one space dimension, is studied.  We consider observations of the solution field on a discrete grid in time and space with infill asymptotics in both coordinates. Firstly, we derive a nonparametric estimator for the reaction function of the underlying equation. The estimate is chosen from a finite-dimensional function space based on a least squares criterion. Oracle inequalities provide conditions for the estimator to achieve the usual nonparametric rate of convergence. Adaptivity is provided via model selection. Secondly, we show that  the asymptotic properties of realized quadratic variation based estimators for the diffusivity and volatility carry over from linear SPDEs. In particular, we obtain a rate-optimal joint estimator of the two parameters. The result relies on our precise analysis of the H\"older regularity of the solution process and its nonlinear component, which may be of its own interest. Both steps of the calibration can be carried out simultaneously without prior knowledge of the parameters.
\end{abstract}

\smallskip


\section{Introduction}
In view of a growing number of stochastic partial differential equation (SPDE) models used in the natural sciences as well as in mathematical finance, their data-based calibration has become an increasingly active field of research during the last years. 
 By studying stochastic reaction-diffusion equations, this article advances the statistical theory for SPDEs based on discrete observations in time and space to a semilinear framework and provides a first nonparametric estimator for the reaction function.
Specifically, we consider the mild solution $X=(X_t(x), x\in [0,1], t\geq 0)$ of the SPDE
\begin{equation} \label{eq:SPDE_intro}
\begin{cases}
dX_t(x) = \big(\vt \frac{\partial^2}{\partial x^2} X_t(x) + f(X_t(x)) \big)\,dt +\sigma dW_t(x),\\
 X_t(0)=X_t(1)=0,\\
X_0(x) = \xi(x) ,
\end{cases} 
\end{equation}
with  $dW$ denoting a white noise in space and time and a random initial condition $\xi\colon[0,1]\to \R$. The equation is parameterized by the volatility $\sigma>0$, the diffusivity
$\vt>0$  and a possibly nonlinear reaction function $f\colon\R \to \R$ on which we impose no
parametric assumptions.   Reaction-diffusion equations  are typically used to model a scenario where local production
of some quantity $X$ with the nonlinear reaction function $f$ competes with a linear diffusion effect while
undergoing internal fluctuations, see \cite{Haken13} for the physical background. Of particular interest is the case where
$f$ is a polynomial of odd degree with a negative leading coefficient. With $f \equiv 0$, the model also includes the classical linear stochastic heat equation
on an interval.

 A complete calibration of the model decomposes into the parametric estimation of $\sigma^2$ and $\vt$ and the nonparametric estimation of $f$.
We study the practically most natural situation where $X$ is observed at a discrete grid $\{(t_i,y_k)\}_{i=0,\ldots,N,\, k = 0,\ldots,M} \subset [0,T]\times [0,1]$  in time and space with $T > 0$ either fixed or $T \to \infty$. Our focus lies on a high frequency regime in time and space where both the number
$M$ of spatial observations and the number $N$ of temporal observations tend to infinity. \\


So far, the vast majority of literature on statistics for SPDEs deals with linear equations, see \cite{Cialenco18, lototsky2009} for reviews of various available approaches and observation schemes.
The statistical analysis of semilinear SPDEs is still limited. Within the \emph{spectral approach}, where one considers the observations $\langle X_t,e_k\rangle, t\in [0,T],k\leq K \to \infty$ with $(e_k)_{k \geq 1}$ being the eigenbasis of the differential operator in the underlying SPDE, \citet{Cialenco11} have considered diffusivity estimation for the stochastic two-dimensional Navier-Stokes equation. This theory has been generalized by \citet{Pasemann20} as well as \citet{PasemannFleming20} to more general equations. Diffusivity estimation based on the \emph{local measurements} approach due to \citet{altmeyerReiss2019} was studied in a semilinear framework in \citet{Altmeyer20, AltmeyerBretschneider20}. There, the observations are given by $\langle X_t,K_h \rangle, t\in [0,T],$ for a kernel function $K_h$ that localizes in space as $h\to 0$. The behavior of power variations (in space, at a fixed time instance) for semilinear SPDEs has been studied only very recently by \citet{Cialenco21}. 

 To estimate the parameters $\sigma^2$ and $\vt$ based on fully discrete observations, we extend the realized quadratic variation based methods developed in \citet{Hildebrandt19} and \citet{Bibinger18} to the semilinear framework \eqref{eq:SPDE_intro}. Similar methods have also been applied to various linear SPDE models, e.g., in \cite{torresEtAl2014, Cialenco17, shevchenkoEtAl2019, Khalil19, Chong18}. Specifically, assuming $f \in C^1(\R)$ in the model \eqref{eq:SPDE_intro}, we show that the asymptotic properties of the realized quadratic variations based on space, time and the double increments
derived for the case $f\equiv 0$  are robust with respect to a nonlinear perturbation of the equation. In particular, using double increments $D_{ik}:= X_{t_{i+1}}(y_{k+1})-X_{t_{i+1}}(y_k)-X_{t_{i}}(y_{k+1})+X_{t_{i}}(y_k)$, the parameters $(\sigma^2,\vt)$ can be estimated jointly at the rate $$\Big(\frac{\delta^3\vee \Delta^{3/2}}{T}\Big)^{1/2}\qquad \text{with}\qquad \delta:= y_{k+1}-y_k,\quad \Delta := t_{i+1}-t_i.$$ This rate is generally slower than the usual parametric rate $(MN)^{-1/2}$, unless a \emph{balanced sampling design} $\delta \eqsim \sqrt \Delta$ is present. Nevertheless, in view of an immediate extension of the lower bound from \cite{Hildebrandt19} to the case $T \to \infty$, this rate is seen to be optimal  up to a logarithmic factor. 

As for spectral and local observations, we analyze the diffusivity and volatility estimators by regarding the semilinear equation as a perturbation of the linear case, i.e., we decompose  $X_t= X_t^0 +N_t$ with $X_t^0$ being the solution to the corresponding linear system. The linear component $X_t^0$ contains all necessary information on $(\sigma^2,\vt)$. Exploiting that the regularity of the nonlinear component $N_t$ exceeds the regularity of $X_t^0$, statistical methods for the linear equation turn out to be applicable also in the semilinear model. Since our estimators are based on quadratic variations, we have to quantify the regularity of $N_t$ in terms of H\"{o}lder spaces. Besides the concrete application in statistics, our detailed account of the H\"{o}lder regularity in time and space also provides structural insights from a probabilistic point of view. 
\\


When we aim for the reaction function,  all relevant information is encoded in the nonlinear component $N_t$. The estimation of the nonlinearity in semilinear SPDEs  was conducted by \citet{Goldys02} who have studied a parametric problem, assuming a full observation $(X_t)_{t\leq T}$ as $ T\to \infty$. Somewhat related, \citet{PasemannFleming20} have studied the estimation of the diffusivity parameter when the nonlinearity is only known up to a finite-dimensional nuisance parameter by using a joint maximum likelihood approach for spectral observations.  Since it follows, e.g., from the absolute continuity result in \citet{KoskiLoges85} that even the simple linear function $f(x)=\vt_0 x$ for some $\vt_0<0$ cannot be identified in finite time (unless $\sigma \downarrow 0$), consistent estimation of the reaction function requires $T \to \infty$. Very recently \citet{gaudlitzReiss2022} have considered estimation of the nonlinearity for full observations  $(X_t)_{t\leq T}$ and some extensions to discrete observations in an alternative asymptotic regime where $T$ is fixed, but $\sigma\downarrow0$.  

Nonparametric estimation of the reaction function $f$ turns out to be comparable to nonparametric drift estimation for finite-dimensional stochastic ordinary differential equations (SODEs). The latter problem, has been addressed in the statistics literature for high-frequency observations in numerous works, see, e.g., \cite{Hoffmann99, Comte07}.  The starting point for drift estimation for SODEs is to formulate a regression model based on the discrete observations. Generalizing this approach to the infinite-dimensional SPDE framework, our key insight is the regression-type decomposition 
$$\frac{X_{t+\Delta}-S(\Delta)X_t}{\Delta} = f(X_t) + \text{``stochastic noise term"} + \text{``negligible remainder terms"}$$
where $(S(t))_{t\geq 0}$ is the strongly continuous semigroup on $L^2((0,1))$ generated by the operator $\vt \frac{\partial^2}{\partial x^2}$. Note that, in contrast to SODEs, the time increments in the response variables have to be corrected in terms of the semigroup due to its presence in the nonlinear component $N_t = \int_0^t S(t-s)f(X_s)\,ds$. Doing so, the contribution of the linear component to the right hand side of the regression model automatically becomes stochastically independent of the covariate $X_t$ and can, thus, be treated as stochastic noise. In a related context, semigroup corrected increments have also been studied by \citet{benthEtAl2022}. Clearly, computing $S(\Delta)X_t$ is not feasible in the discrete observation scheme,  as it depends on the whole spatial process $(X_t(x),\,x\in (0,1))$, and we replace it by an empirical counterpart. Additionally, the semigroup depends on the possibly unknown parameter $\vt$ which we address by employing a plug-in approach using any diffusivity estimator with a sufficiently fast convergence rate. Doing so, we obtain an approximation  $ \check S_t^\Delta \approx  S(\Delta)X_t$ which is only based on the discrete data.

To estimate $f$, we adapt the nonparametric least squares approach by \citet{Comte02} which was successfully applied to  ergodic one-dimensional diffusion processes in \citet{Comte07}. Hence, our nonparametric estimator is defined as the minimizer of 
$$\Gamma_{N,M}(g) := \frac{1}{MN} \sum_{i=0}^{N-1} \sum_{k=1}^{M-1} \Big( g(X_{t_i}(y_k))- \frac{X_{t_{i+1}}(y_k)-\check S_{t_i}^\Delta (y_k)}{\Delta} \Big)^2$$ 
over the functions $g$ from a suitable finite-dimensional approximation space. Working  in an ergodic regime for the process $(X_t)_{t\geq 0}$,
we derive $\mathcal{O}_p$-type oracle inequalities for the estimator when the risk is either the empirical 2-norm with evaluations at the data points or the usual $L^2$-norm on a compact set.  These oracle inequalities reflect the well-known bias-variance trade-off in nonparametric statistics. An optimal choice of the dimension of the approximation space yields the usual nonparametric rate $T^{-{\alpha}/({2\alpha+1})}$ where $\alpha$  quantifies the regularity of $f$. Employing a model selection approach similar to \cite{Comte07}, we obtain an adaptive estimator. Since our estimator is based on an approximation of the spatially continuous model, we require a fine observation mesh throughout the whole space domain. More precisely, the condition $M\Delta^2 \to \infty$ is necessary.\\

This article is organized as follows: In Section \ref{sec:model}, we introduce the SPDE model and the considered observation scheme. Further, we analyze the H\"{o}lder regularity in time and space of the solution process itself and of its nonlinear component. Section \ref{sec:nonpara} is devoted to nonparametric estimation of $f$. We introduce the approximation spaces, we define the estimator and we derive corresponding oracle inequalities. To that aim, we assume, firstly, that the diffusivity parameter $\vt$ is known and then employ a plug-in approach. In Section \ref{sec:parametric_nonlinear}, we verify that the asymptotic properties of realized quadratic variations based on space, time and double increments as well as the corresponding estimators mainly carry over from the linear setting.  All proofs are collected in Section \ref{sec:proofs}.\\  

We use the notations $\N:=\{1,2,\ldots\}$ and $\N_0 := \N \cup \{0\}$ as well as $\R_+:=[0,\infty)$. For $a,b \in \R$ we use the shorthand $a\wedge b:=\min (a,b)$ and $a\vee b := \max(a,b)$ as well as $[a]:= \max\{k\in \N_0:\,k\leq a\}$ for $a\in \R_+$. For two sequences $(a_n),(b_n)$, we write $a_n \lesssim b_n$ to indicate that there exists some $c>0$ such that $|a_n|\leq c\cdot |b_n|$ for all $n\in \N$ and we write {$a_n \eqsim b_n$ if $a_n\lesssim b_n\lesssim a_n$}. Throughout $a_n,b_n \to \infty$ is meant in the sense of $a_n\wedge b_n\to \infty$ for $n\to \infty$.  If  $a_n=a$ for some $a\in \R$ and  all $n\in\N$, we write $(a_n)\equiv a$. When we write statements like $M,N \to \infty$, we implicitly assume that $M$ and $N$ depend on a common index $n\in \N$ such that $N_n,M_n \to \infty$ for $n\to \infty$. Convergence in probability and convergence in distribution are denoted by $\overset {\P}{ \longrightarrow}$ and $\overset{d}{ \longrightarrow}$, respectively. The total variation distance of two probability measures $P,Q$ on some measure space $(\Omega,\mathcal F)$ is denoted by $\Vert P- Q\Vert_{\mathrm{TV}}:= \sup_{A\in \mathcal F} |P(A)-Q(A)|$. The stochastic Landau symbols are denoted by $\mathcal O_p$ and $o_p$.

\section{Basic assumptions and H\"{o}lder regularity of the solution process} \label{sec:model}
Throughout, we work on a probability space $(\Omega, \mathcal F, \P)$ equipped with a filtration $(\mathcal F_t)_{t\geq 0}$ satisfying the usual conditions.
	In order to introduce the reaction-diffusion equation from \eqref{eq:SPDE_intro} thoroughly, let us consider the semilinear SPDE
\begin{equation} \label{eq:SPDE1}
dX_t = \big(\vt \frac{\partial^2}{\partial x^2} X_t + F(X_t)\big)\,dt+\sigma dW_t,\quad X_0=\xi,
\end{equation}
where $W$ denotes a cylindrical Brownian motion on $L^2((0,1))$, $\xi$ is some $\mathcal F_0$-measurable initial value and where the nonlinearity $F$ acts on functions $u\colon[0,1]\to \R$ via $$F(u)=f\circ u \quad \text{for some} \quad f\in C^1(\R).$$ For simplicity, we will also refer to the functional by $f$, i.e.,~we write $f(u)=f\circ u$.


Throughout, the parameters $\sigma$ and $\vt$ are strictly positive constants. As usual, the Dirichlet boundary conditions in \eqref{eq:SPDE1} are implemented into the domain of the Laplace operator $A_{\vt} := \vt \frac{\partial^2}{\partial x^2} $, namely, we define
$\mathcal D (A_\vt) := H^2((0,1))\cap H_0^1((0,1))$ where $H^k((0,1))$  denotes the $L^2$-Sobolev spaces of order $k\in\N$ and with $H_0^1((0,1))$ being the closure of $C_c^\infty((0,1))$ in $H^1((0,1))$. We denote the strongly continuous semigroup on $L^2((0,1))$ generated by the operator $A_\vt$ by $(S(t))_{t\geq 0}$. Recall that a cylindrical Brownian motion $W$ is defined as a linear mapping $L^2((0,1)) \ni u \mapsto W_\cdot(u)$ such that 
 $t\mapsto  W_t(u)$ is a one-dimensional standard Brownian motion for all normalized $u\in L^2((0,1))$
and such that the covariance structure is given by
$\Cov\left( W_t(u) , W_s(v) \right)=(s\wedge t)\, \langle u,v \rangle_{L^2},$ for $u,v\in L^2((0,1)),\,s,t\geq 0$. $W$ can thus be understood as the anti-derivative in time of space-time white noise. \\

We will work under the standing assumption that there is a unique mild solution 
\begin{equation} \label{eq:SPDE1_solution}
X_t = S(t)\xi + \sigma \int_0^t S(t-s)\,dW_s + \int_0^t S(t-s)f(X_s)\,ds,\quad t\geq 0 ,
\end{equation}
such that $X:= (X_t)_{t\geq 0}$ is a Markov process with state space $E:=C_0([0,1]):=\{u \in C([0,1]):\, u(0)=u(1)=0\}$ and such that $X\in C(\R_+,E)$ holds almost surely for all $\xi \in E$. It follows from \cite[Example 7.8]{DaPrato14} that this assumption is fulfilled, e.g., when $f$ satisfies
\begin{equation} \label{eq:coerc} 
f(\lambda+\eta)\mathrm{sgn}(\lambda)\leq a(|\eta|)(1+|\lambda|),\qquad \lambda,\eta \in \R,
\end{equation}
for some increasing function $a\colon \R_+ \to \R_+$. This is particularly the case when $f$ is a polynomial of odd degree with a negative leading coefficient. The existence of continuous trajectories is a crucial requirement for dealing with fully discrete observations. Even in the linear setting $f\equiv 0$, function valued solutions to equation \eqref{eq:SPDE1} only exist in dimension one.

In our whole analysis, the decomposition of $X$ into its linear and its nonlinear component turns out to be useful. The decomposition is given by $X_t = S(t) \xi +X_t^0+N_t,\,t\geq 0,$ where 
$$X_t^0 := \sigma \int_0^t S(t-s)\,dW_s,\qquad N_t:=\int_0^t S(t-s)f(X_s)\,ds,\quad t\geq 0 ,$$
and $(X_t^0)_{t\ge0}$ is the mild solution to the associated linear SPDE ($f\equiv 0$, $\xi \equiv 0$).
Recall that the operator $A_\vt$ has a complete orthonormal system of eigenfunctions in $L^2((0,1))$. Its eigenpairs $(-\lambda_\ell,e_\ell)_{\ell \geq 1}$  are given by 
\begin{gather*} e_\ell (y) = \sqrt{2} \sin(\pi \ell y ),\quad
\lambda_\ell = \vartheta\pi^2 \ell^2, \qquad  y\in [0,1],\,\ell\in \N.
\end{gather*}
Employing this sine base, the cylindrical Brownian motion $W$ can be realized via $W_t=\sum_{\ell\geq 1} \beta_\ell(t)e_\ell$ in the sense that
$W_t(\cdot)= \sum_{\ell\geq 1} \beta_\ell(t)\langle \cdot,e_k\rangle_{L^2}$ for a sequence of independent standard Brownian motions $(\beta_\ell)_{\ell \geq 1}$. In particular, the linear component of $X$ admits the representation
\begin{equation} \label{eq:SPDE_solution}
{X_t^0 (x)}  {=\sum _{\ell\geq 1} u_\ell(t)e_\ell(x)}, \quad t \geq 0,\,x\in [0,1], 
\end{equation}
where $(u_\ell)_{\ell \ge1}$ are one dimensional independent processes satisfying the Ornstein-Uhlenbeck dynamics 
$
du_\ell(t)= -\lambda_\ell u_\ell(t)\,dt +\sigma\, d\beta_\ell(t)
$
or, equivalently,
$$u_\ell(t)= \sigma\int_0^t \e^{-\lambda_\ell (t-s)}\,d\beta_\ell(s)$$
in the sense of a finite-dimensional stochastic integral. From representation~\eqref{eq:SPDE_solution} it is evident that $(t,x) \mapsto X^0_t(x)$ is a two-parameter centered Gaussian field with covariance structure 
\begin{equation} \label{eq:cov}
\Cov\left(X_s^0(x),X_t^0(y)\right)= \sigma^2 \sum_{\ell\geq 1}\frac{\e^{-\lambda_\ell|t-s|}-\e^{-\lambda_\ell (t+s)}}{2\lambda_\ell}e_\ell(x)e_\ell(y), \quad s,t\geq 0,\,x,y\in [0,1]. 
\end{equation}

In our analysis, we will always assume that either $\xi=0$ or that $\xi$ follows the stationary distribution on $E$ associated with the Markov process $X$, provided that it exists.  The case $\xi\equiv 0$ is representative for other smooth initial conditions and considered here for the sake of simplicity. For the linear system $f\equiv 0$, the stationary distribution can be realized by letting $\{\beta_\ell,\langle \xi,e_\ell\rangle,\,\ell\in \mathbb N\}$ be independent with $\langle \xi,e_\ell\rangle\sim \mathcal N(0,{\sigma^2}/({2 \lambda_\ell}))$ such that each Fourier mode $t\mapsto \langle X_t ,e_\ell \rangle $ is stationary with covariance  function 
$
c(s,t)
:=\frac{\sigma^2}{2\lambda_\ell}\e^{-\lambda_\ell|t-s|}$, $s,t\geq 0$. For  semilinear equations there are generally no explicit expressions for the invariant distribution, though its existence can be guaranteed via abstract criteria in a large variety of cases, see, e.g., Proposition \ref{prop:f_coerc} below.\\

As long as the time horizon $T$ remains bounded, the realized-quadratic-variation-based estimators for $\sigma^2$ and $\vt$ can be generalized to the just introduced semilinear setting if the nonlinearity $f$ and its derivative are at most of polynomial growth:
\begin{itemize}
 \item[(F)] There exist constants $c>0$ and $d\in \N$ such that  $|f(x)|,|f'(x)| \leq c (1+|x|^d)$, $x\in \R$.
\end{itemize}
This basic assumption is sufficient to deduce the higher regularity of $(N_t)$ from properties of $X$.
 When dealing with the case $T\to \infty$, on the other hand, we need to impose a stricter assumption, that ensures that the error induced by the nonlinearity remains negligible uniformly in time, namely: 
\begin{itemize}
\item[(B)] The process $X$ from \eqref{eq:SPDE1_solution} with zero or, in case of existence,  stationary initial condition satisfies $\sup_{t\geq 0}\E(\Vert X_t \Vert_\infty^p)<\infty$ for any $p\geq 1$. 
\end{itemize}
For the nonparametric estimation of the nonlinearity $f$, our analysis will rely on a concentration inequality derived via the mixing property of a stationary process. Hence, we will assume:  
\begin{itemize}
\item[(M)] For the Markov process $X$ from \eqref{eq:SPDE1_solution}, there exists a stationary distribution $\pi$ on $E$ and the mild solution with $X_0=\xi\sim \pi$ satisfies $\E(\Vert X_t \Vert_\infty^p)= \E(\Vert X_0 \Vert_\infty^p)<\infty$ for any $p\geq 1$. Furthermore, $X$ is exponentially $\beta$-mixing, i.e., there exist constants $L,\tau>0$ such that 
\begin{equation} \label{eq:(M)betaX_def}
\beta_X(t):=\int_{E} \Vert P_t(u,\cdot) - \pi(\cdot)\Vert_{\mathrm{TV}}\,\pi(du)\leq L \e^{-\tau t},\qquad t\geq 0,
\end{equation}
where $(P_t)_{t\geq 0}$ is the transition semigroup on $E$ associated with the Markov process $X$.
\end{itemize}
Sufficient conditions for Assumptions (B) and (M) to be satisfied are given in the following Proposition which is a slight extension of  results derived in \citet{Goldys06}.
\begin{prop} \label{prop:f_coerc}
If (F) is satfisfied and there are constants $a,b,c,\beta \geq 0$ such that
\begin{equation} \label{eq:f_coerc}
\mathrm{sgn}(x)f(x+y)\leq -a |x| + b|y|^\beta +c
\end{equation}
holds for all $x,y \in \R$, then Assumptions $(B)$ and $(M)$ are satisfied.
\end{prop} 
A proof for the above proposition is given in Section \ref{subsec:aux_nonlinear}. Condition \eqref{eq:f_coerc} requires that $-f$ has at least linear growth at infinity and is, not surprisingly, stronger than the general existence condition \eqref{eq:coerc}. Still, it covers a large class of systems, including the case where $f$ is a polynomial of odd degree with a negative leading coefficient.

Finally, for the nonparametric estimation of $f$ on a compact set $\mathcal A\subset \R$, we will need that the $L^2(\mathcal A)$-norm is comparable to the empirical norm induced by the process $X$. This can be achieved by requiring the following equivalence condition.

\begin{itemize}
\item[(E)] For the Markov process $X$ from \eqref{eq:SPDE1_solution} there exists a stationary distribution $\pi$ on $E$ and, if $\xi \sim \pi$, the random variables $\xi(x)$ admit a Lebesgue density $\mu_x$ for each $x\in (0,1)$. Further, for any compact set $\mathcal A \subset \R$, there are constants $c_0,c_1>0$ and $\mathfrak b\in (0,\frac{1}{2})$ such that 
\begin{align*}
\mu_x(z)\leq  c_1 &\;\text{for all}\, z\in \mathcal A,\,x\in (0,1)\qquad\text{and}\qquad
\mu_x(z)\geq  c_0 \;\text{for all}\, z\in \mathcal A,\,x\in (\mathfrak b,1-\mathfrak b).
\end{align*}
\end{itemize}
The presence of the constant $\mathfrak b$ in the above lower bound is due to the Dirichlet boundary conditions.
Assumption (E) is clearly satisfied in the case where $f$ is a linear function $f(x)=\vt_0x$ for some $\vt_0<0$. Indeed, the corresponding stationary distribution matches the stationary distribution for the case $f\equiv 0$  with $\lambda_\ell$ replaced by $\tilde \lambda_\ell := \lambda_\ell -\vt_0$. Thus, the random variables $X_0(x),\,x\in (0,1),$ are Gaussian and (E) can be checked by examining their variances. Concerning more general nonlinearities $f$, there is a large amount of literature concerned with the existence and regularity of Lebesgue densities corresponding to the marginal distributions associated with various SPDE models, see, e.g., \cite{Mueller08,Marinelli13,Nualart09,Bally98}. However, to the authors' best knowledge, there are so far no estimates on the densities of the random variables $X_t(x)$ that hold uniformly in $x \in \mathcal X $ for some infinite set $\mathcal X\subset (0,1)$. {Deriving a sufficient condition on $f$ to ensure (E) goes beyond the scope of this work and is postponed to further research.} 
\subsection{Observation scheme} \label{subsec:observation_scheme}
Throughout, we suppose to have time- and space-discrete observations $\{X_{t_i} (y_k ),\, i=0,\dots, N,\,k=0,\dots, M\} $ of a path of the process $X$ from \eqref{eq:SPDE1_solution}
on a regular grid $\{(t_i,y_k)\}_{i,k}\subset [0,T] \times[0,1]$  with time horizon $T>0$ and $M,N\in \N$.
More precisely, we assume that 
\begin{align}
y_k=b+k\delta\quad\text{and}\quad t_i=i\Delta \qquad\text{where}\qquad\delta =\frac{1-2b}{M},\quad \Delta=\frac TN\label{eq:obs}
\end{align}
for some fixed $b\in[0,1/2)$ (which is not related to the boundary margin $\mathfrak b$ from (E)). 
Due to the boundedness of the space domain, we have high frequency observations in space whenever $M\to\infty$. High frequency observations in time are present when  $T/N \to 0$. This is trivially satisfied if $T$ is fixed and $N \to \infty$.

Note that the spatial locations $y_k$ are equidistant inside a (possibly proper) sub-interval $[b,1-b]\subset[0,1]$. For certain statistical procedures, we will exclude observations close to the boundary by requiring $b>0$. This is done to prevent undesired boundary effects, which lead to biased estimates.
\subsection{H\"{o}lder regularity of the solution process} \label{subsec:Hoelder}

Next, we discuss the H\"{o}lder regularity of the process $(X_t(y),\,t\geq 0, y \in [0,1])$ in time and space and, in particular, we show the higher order regularity of its nonlinear component $(N_t(y),\,t\geq 0, y \in [0,1])$. For $\alpha>0$, we consider the H{\"o}lder spaces $C^\alpha := C^{\alpha}([0,1])$ consisting of all $u\in C^{[\alpha]}$ such that
$$\Vert u \Vert_{C^\alpha} := \sum_{k=0}^{[\alpha]} \Vert u^{(k)}\Vert_\infty + \sup_{x,y\in [0,1]} \frac{|u(x)-u(y)|}{|x-y|^{\alpha-[\alpha]}}<\infty.$$
H\"{o}lder continuous functions with Dirichlet boundary conditions are denoted by $C^\alpha_0:= \{u\in C^\alpha,\,u(0)=u(1)=0\}.$ 

The linear component $(X_t^0(x),\,x \in [0,1], t\geq 0)$ of $X$ is a Gaussian process satisfying
\begin{align}
\E((X_t^0(x)-X_t^0(y))^2)
&\leq \sum_{\ell \geq 1} \frac{\sigma^2}{2\lambda_\ell} (e_\ell(x)-e_\ell(y))^2   \eqsim |x-y|, \label{eq:space_variation}\\
\E((X_t^0(x)-X_s^0(x))^2)
&\leq \sum_{\ell \geq 1} \frac{\sigma^2}{\lambda_\ell}(1-\e^{-\lambda_\ell |t-s|}) \eqsim \sqrt{|t-s|}, \label{eq:time_variation}
\end{align} 
see, e.g., \cite[Theorem 3.3]{Hildebrandt19} and \cite[Proposition 3.1]{Bibinger18}, respectively.
As a well known consequence, we have that, almost surely,  $X_t^0 \in E$ for any $t\geq 0$, $x\mapsto X_t^0(x)$ is $2\gamma$-H{\"o}lder continuous and $t\mapsto X_t^0(x)$ is locally ${\gamma}$-H{\"o}lder continuous for any $\gamma < 1/4$. The following proposition generalizes this fact to the semilinear setting and shows that, under Assumption (B), the corresponding H{\"o}lder norms are $L^p(\P)$-bounded as functions of time.
\begin{prop} \label{prop:X_reg} Let (F) be satisfied. For any $p\in [1,\infty)$ the following hold:
\begin{enumerate}[(i)]
\item 
For any $\gamma < 1/2$, we have
$X \in C(\R_+,C_0^\gamma)$ a.s.~and, if Assumption (B) is satisfied, then \linebreak
$\sup_{t\geq 0}\E(\Vert X_t \Vert_{C_0^{\gamma}}^p)<\infty $.
\item For any $\gamma<1/4$ and $T>0$, we have $(X_t)_{0\leq t\leq T} \in C^\gamma([0,T],E)$ a.s.~and, if Assumption (B) is satisfied, then
 there is a constant $C>0$ such that $\E(\Vert X_t-X_s \Vert_\infty ^p)\leq C|t-s|^{\gamma p}$ for all $s,t\geq 0$.
\end{enumerate}
Furthermore, the same results hold for $X$ replaced by $f_0(X)$ where $f_0(x):=f(x)-f(0)$.
\end{prop} 

A norm bound as in $(i)$ with $p=1$ is also derived by  \citet[Proposition 4.2]{Cerrai99}.  
To prove the above proposition, we analyze the linear and the nonlinear component of $X$ separately. The regularity of $(X^0_t)$ can be assessed by using properties \eqref{eq:space_variation} and \eqref{eq:time_variation} together with Sobolev embeddings and,  especially, the Garsia-Rodemich-Rumsey inequality, see Lemma \ref{lem:linear}. The regularity of $(N_t)$ is a consequence of the regularizing property of the semigroup $(S(t))_{t\geq 0}$ in view of the fact that, due to our basic assumptions, the process $f(X)$ is continuous as a function of time and space. \\

Having derived the H\"{o}lder regularity of the process  $X$ and, in particular, of $f(X)$, we can use the regularizing impact of $(S(t))_{t\geq 0}$ once more to deduce that the regularity of $N_t = \int_0^t S(t-s)f(X_s)\,ds,\, t\geq 0,$ exceeds the regularity of $X$. A related strategy has been pursued by \citet{Pasemann20} who studied the higher order regularity of the nonlinear component of $X$ in the Sobolev spaces $\mathcal D((-A_\vt)^\eps), \,\eps>0$ from \eqref{eq:Deps_def}, see also Section \ref{subsec:aux_nonlinear}. For our purpose, we can proceed similarly to \citet{Sinestrari85} who studied the H\"{o}lder regularity of mild solutions to deterministic systems.
We employ the decomposition $N_t= N_t^0 +M_t$ where
\begin{align} \label{eq:Nt0Mt_def}
N_t^0 := \int_0^t S(t-s)f_0(X_s)\,ds,\quad M_t := \int_0^t S(r)m\,dr
\end{align}
for $m \equiv f(0)$ and $f_0(x)= f(x)-f(0)$. Note that $u\mapsto f_0\circ u$ maps $E$ and, in particular, $C_0^{\alpha}$  into itself. 
\begin{prop} \label{prop:AN_reg}
Let (F) be satisfied. For any $T>0$ and $p\geq 1$  the following hold.
\begin{enumerate}[(i)]
\item
For any $\gamma<1/2$ and $t\geq0$, we have $N_t^0 \in C_0^{2+\gamma}$ and $\sup_{t\leq T} \Vert {A_\vt}N_t^0 \Vert_{C_0^{\gamma}}<\infty$ almost surely. If Assumption (B) is satisfied, then $ \sup_{t\geq 0}\E(\Vert {A_\vt}N_t^0\Vert_{C_0^{\gamma}}^p)<\infty .$
\item
For any $\gamma <1/4$, we have $(N^0_t)_{t\leq T} \in C^{1+\gamma}([0,T],E)$ and $\frac{d}{dt}N^0_t = f_0(X_t)+{A_\vt}N_t^0$ in $E$ almost surely. Under Assumption (B), there exists $C>0$ such that $\E(\Vert \frac{d}{dt}(N_t^0-N_s^0)\Vert_\infty^p ) \leq C(t-s)^{\gamma p}$ holds for all $s,t\geq 0$.
\end{enumerate}
Furthermore, the same results hold for $(N_t)$ and $f$ instead of $(N^0_t)$ and $f_0$, provided that we replace $E$ by $C([b,1-b])$ and $C_0^{\gamma}$ by $C^{\gamma}([b,1-b])$ for any $b\in (0,\frac{1}{2})$.
\end{prop}

\section{Nonparametric estimation of the nonlinearity} \label{sec:nonpara}
In order to estimate $f$ nonparametrically, we adapt an estimation procedure, introduced by \citet{Comte07} in the context of one-dimensional diffusions, to the SPDE setting. In a first step, we will assume that the parameters $(\sigma^2,\vt)$ (in fact, only $\vt$ is necessary) are known. Subsequently, a plug-in approach will be considered for the case of unknown $\vt$. The construction of a suitable estimator $\hat\vt$ is given in Section~\ref{sec:parametric_nonlinear}.

We assume that the mild solution $X$, defined by \eqref{eq:SPDE1_solution},
admits a stationary distribution, denoted by $\pi$, and, moreover, that the mixing assumption (M) is satisfied. Furthermore, it will be essential for the derivation of our oracle inequalities that we have $T\to \infty$ and $\Delta \to 0$. Let $\mathcal A=[-a,a]$, for some $a>0$, be the interval on which we want to estimate $f$.


\subsection{Approximation spaces}
In order to estimate $f$ on the set $\mathcal A$, we consider a sequence $(\mathcal V_m)_{m\in \N}$ of finite-dimensional sub-spaces of $L^2(\mathcal A)$ such that $D_m:=\mathrm{dim}(\mathcal V_m) \to \infty$ for $m\to \infty$. For a well chosen $m$, we will estimate $f$ by a function $\hat f_m \in \mathcal V_m$ that minimizes the empirical loss to be defined later.
As in  \cite{Baraud01}, our key assumption on the approximation spaces $(\mathcal V_m)$ is the following.
\begin{itemize}
\item[(N)] There is a constant $C>0$ such that for each $m\in \N$ there is an orthonormal basis $(\varphi_k)_{k\in \Lambda_m}$ of $\mathcal V_m$, $|\Lambda_m|= D_m$, with
$
\Vert \sum_{k \in \Lambda_m} \varphi_k^2 \Vert_\infty \leq C D_m.
$
\end{itemize}
The dependence of $\varphi_k$ on $m$ is not made explicit for ease of notation.
It is shown in \citet{Birge97} that Assumption (N) is equivalent to requiring $\Vert g\Vert_\infty^2 \leq C D_m \Vert g\Vert^2_{L^2(\mathcal A)}$ for all $g \in \mathcal V_m$ and $m\in \N$. 
Additionally, a minimal continuity property in the approximation spaces will be required:

\begin{itemize}
\item[(H)] For any $g\in \bigcup_{m\in \N}\mathcal V_m$, let $\bar g\colon\R \to \R$ be the extension of $g$ by zero on the set $\mathcal A^c$. Then, the function $\bar g$ is piecewise H\"{o}lder continuous, i.e., there are constants $\alpha>0$ and $-\infty =a_0<a_1<\ldots< a_L=\infty,\,L\in \N,$ such that $\bar g |_{(a_l,a_{l+1})}\in C^\alpha((a_l,a_{l+1}))$ for any $1\leq l\leq L-1$. 
\end{itemize}


Let us briefly recall some examples of approximation spaces with property (N)  that are considered in \cite{Baraud01}. In fact, all of those also meet our additional continuity requirement (H).
\begin{ex}~
\begin{itemize}
\item[{[T]}]  The trigonometric spaces 
$$\mathcal V_m = \mathrm{span}\Big(\Big\{\frac{1}{\sqrt{2a}},\,\frac{1}{\sqrt{a}}\sin\big(\frac{k\pi}{a}\cdot \big),\, \frac{1}{\sqrt{a}}\cos \big(\frac{k\pi}{a}\cdot \big),\,1\leq k\leq m \Big\} \Big)$$ have dimension $D_m=2m+1$ and property (N) follows directly from the fact that the trigonometric base functions are uniformly bounded.
\item[{[P]}] Piecewise polynomials on a dyadic grid: Let $(p_l)_{l\in \N_0}$ be the complete orthonormal system in $L^2([0,1])$ such that $p_l$ is the rescaled Legendre polynomial of degree $l$ for $l\in \N_0$. For $p\in \N$ and $j\in \{-2^{p},\ldots,2^p-1\}$, let $I_j^p := [ja2^{-p},(j+1)a2^{-p})$. Then, for $m=(p,r)$ with $r\in \{0,\ldots,r_{\max}\}$ and $r_{\max} \in \N_0$, we define 
$$\mathcal V_{(p,r)}:=\mathrm{span}\Big( \{ \varphi_{j,l}^p,\, l\leq r ,-2^p\leq j \leq 2^p-1\} \Big),\quad \varphi_{j,l}^p(x):=\sqrt{\frac{2^p}{a}}p_l\Big(\frac{2^px}{a}-j\Big)\1_{I_j^p}(x),x\in\mathcal A.$$
Clearly, $\mathrm{dim}(\mathcal V_{p,r}) = (r+1)2^{p+1} \leq (r_{max}+1)2^{p+1}$ and property (N) holds with a constant $C$ depending on $r_{\max}$.
\item[{[W]}] The dyadic wavelet generated spaces: For arbitrary $r \in \N$, there are a scaling and a wavelet function $\phi,\psi \in C^\alpha(\R)$, respectively, for some $\alpha>0$ with support in $[0,1]$ such that $\psi$ has $r$ vanishing moments and
$\Big\{\frac{1}{\sqrt a}\phi\Big(\frac{\cdot}{a}\Big),\, \frac{1}{\sqrt a}\phi \Big(\frac{\cdot}{a}+1\Big), \sqrt{\frac{2^p}{a}}\psi \Big(\frac{2^p\cdot}{a}-j\Big),\,-2^{p}\leq j < 2^p,\,p \in \N \Big\}$
is a complete orthonormal system in $L^2(\mathcal A)$, see \cite{Daubechies92}. Then, the subspace
$$\mathcal V_m= \mathrm{span}\Big(\Big\{\frac{1}{\sqrt a}\phi\Big(\frac{\cdot}{a}\Big),\, \frac{1}{\sqrt a}\phi \Big(\frac{\cdot}{a}+1\Big), \sqrt{\frac{2^p}{a}}\psi \Big(\frac{2^p\cdot}{a}-j\Big),\,-2^{p}\leq j < 2^p,\,p \leq m \Big\} \Big)$$
satisfies $\mathrm{dim}(\mathcal V_m) = 2^{m+2}$ and property (N) is fulfilled.
\end{itemize} 
\end{ex}


Following \citet{Baraud01}, we define matrices $V^m,B^m \in \R^{\Lambda_m\times \Lambda_m}$ by 
$$V^m_{k,k'}:= \Vert \varphi_k \varphi_{k'}\Vert_{L^2(\mathcal A)},\qquad B^m_{k,k'}:= \Vert  \varphi_k \varphi_{k'} \Vert_{\infty}$$
for a fixed orthonormal basis $(\varphi_k,\,k \in \Lambda_m)$ of $\mathcal V_m$ according to Assumption (N). These matrices are especially usefull to bound  $|\varphi_{k}(Z)\varphi_{k'}(Z)|\leq B^m_{k,k'}$ and $\E(|\varphi_{k}(Z)\varphi_{k'}(Z)|^2)\lesssim (V^m_{k,k'})^2$ for any $\mathcal A$-valued random variable $Z$ with bounded Lebesgue density.
Further, let
\begin{equation} \label{eq:Lm_def}
L_m:=\max(\rho^2({V^m}),\rho(B^m)),\qquad \rho(H):= \sup_{a \in \R^{\Lambda_m},\, \Vert a\Vert \leq 1}\sum_{k,k'} |a_k a_{k'} H_{k,k'}|,\,H\in \{{V^m},B^m\}.
\end{equation}
For our main oracle inequalities, we will require that $L_m$ is asymptotically negligible with respect to the time horizon $T$.
For the previous examples of approximation spaces, it is shown in \cite{Baraud01} that $L_m \lesssim D_m^2$ for [T] and $L_m \lesssim D_m$ for [P] and [W].	
 
\subsection{Construction and analysis of the estimator}
For a moment, let us consider observations that are discrete in time but continuous in space, i.e., the data is given by 
$$\{X_{t_i}(x),\,x\in[0,1],\,i=0,\ldots,N \}.$$
From \eqref{eq:SPDE1_solution}, it is evident that we can decompose
\begin{align*}
X_{t+\Delta} = S(\Delta)X_t + \sigma \int_t^{t+\Delta} S(t+\Delta-s)\,dW_s+\int_t^{t+\Delta} S(t+\Delta-s)f(X_s)\,ds .
\end{align*}
By rearranging, we can pass to  
\begin{align*}
\frac{X_{t+\Delta}-S(\Delta)X_t}{\Delta} =  f(X_t) &+ \frac{\sigma}{\Delta} \int_t^{t+\Delta} S(t+\Delta-s)\,dW_s \\
&+ \frac{1}{\Delta}\int_t^{t+\Delta} \Big( S(t+\Delta-s)f(X_s)-f(X_t)\Big)\,ds, 
\end{align*}
yielding the regression model
\begin{equation} \label{eq:reg_model_cont}
Y_i^{\mathrm{cont}} = f(X_{t_i}) + R_i^{\mathrm{cont}} + \eps_i^{\mathrm{cont}},\qquad 0\leq i \leq N-1,
\end{equation}
with
\begin{gather*}
Y_i^{\mathrm{cont}} := \frac{X_{t_{i+1}}-S(\Delta)X_{t_i}}{\Delta},\qquad
\eps_i^{\mathrm{cont}}:=\frac{\sigma}{\Delta} \int_{t_i}^{t_{i+1}} S(t_{i+1}-s)\,dW_s,\\
R_i^{\mathrm{cont}}: =\frac{1}{\Delta}\int_{t_i}^{t_{i+1}} \Big( S(t_{i+1}-s)f(X_s)-f(X_{t_i})\Big)\,ds .
\end{gather*}
The main term in the regression model is given by $f(X_{t_i})$, $\eps_i^{\mathrm{cont}}$ is the stochastic noise term and $R_i^{\mathrm{cont}}$ is a negligible bias. Note that the stochastic noise term is stochastically independent of the covariate $X_{t_i}$. 
The corresponding least squares estimator is defined by
$$\hat f_m^{\mathrm{cont}}:= \argmin_{g \in \mathcal V_m}  \frac{1}{N} \sum_{i=0}^{N-1}\Vert Y_i^{\mathrm{cont}} -g(X_{t_i})\Vert^2_{L^2}$$
with $\Vert \cdot \Vert_{L^2}:=\Vert \cdot \Vert_{L^2((0,1))}$.
While this estimator hinges on the parameter $\vt$ through the semigroup $S(\cdot)$, it is independent of $\sigma^2$.\\

Let us return to the fully discrete observation scheme described in Section \ref{subsec:observation_scheme}. In order to derive a discretized version of $\hat f_m^{\mathrm{cont}}$, we assume that discrete observations are recorded throughout the whole space domain $(0,1)$, i.e., we have $b=0$ in \eqref{eq:obs}. This allows us to approximate the Fourier modes $x_{k}(t):=\langle X_t,e_k \rangle_{L^2}$ by their empirical counterpart given by a Riemann sum approximation.
Recall that there is a discrete version of the othonormality property for the sine base, see, e.g., \cite[Section 5]{Hildebrandt19}. In particular, for $k \leq M-1$, we have the relation 
$$\langle X_t,e_k \rangle_M:=\frac{1}{M} \sum_{l=1}^{M-1}X_t(y_l)e_k(y_l)=\sum_{\ell \in \mathcal I_k^+} x_\ell(t)-\sum_{\ell \in \mathcal I_k^-} x_\ell(t)$$
where
$\mathcal I_k^+ := k+2M\cdot \N_0$ and $\mathcal I_k^- := 2M-k+2M\cdot\N_0$.
In order to approximate the expression $S(\Delta)X_{t_i} $ appearing in the definition of $\hat f _m ^{\mathrm{cont}}$, we define 
$\hat S(\Delta):= \hat S_M(\Delta)$ by $$\hat S(\Delta)  X_t:=\sum_{\ell = 1}^{ M-1}  \e^{-\lambda_\ell \Delta} \langle X_t,e_\ell \rangle_M\, e_\ell$$
which only hinges on $X_t$ through the discrete data $(X_t(y_k),\,k=1,\ldots,M-1)$.
Hence, the discrete version of the space-continuous regression model
\eqref{eq:reg_model_cont} is given by
\begin{equation} \label{eq:Reg_model_discrete}
 Y_i= \hat S(0)f(X_{t_i}) +  R_i +\eps_i
\end{equation}
with
\begin{gather*}
Y_i:= \frac{\hat S(0) X_{t_{i+1}}-\hat S(\Delta)X_{t_i}}{\Delta},\qquad \eps_i := \eps_i^{\mathrm{cont}}\qquad\text{and}\\
R_i:= R_i^{\mathrm{cont}}+f(X_{t_i})-\hat S(0) f(X_{t_i}) +\frac{ S(\Delta) X_{t_{i}}-\hat S(\Delta) X_{t_{i}}}{\Delta}+\frac{\hat S(0) X_{t_{i+1}}-X_{t_{i+1}}}{\Delta}.
\end{gather*}
This motivates our least squares estimator 
\begin{align}
\hat f_m &:= \argmin_{g\in \mathcal V_m} \frac{1}{N}\sum_{i=0}^{N-1}\big\Vert Y_i  -\hat S(0)g(X_{t_i})\big\Vert_{L^2}^2 \nonumber\\
&=\argmin_{g\in \mathcal V_m} \frac{1}{N}\sum_{i=0}^{N-1} \sum_{k=1}^{ M-1}\left(\frac{\langle X_{t_{i+1}}, e_k\rangle_M -\e^{-\lambda_k\Delta}\langle  X_{t_{i}}, e_k\rangle_M}{\Delta}-\langle g( X_{t_{i}}), e_k\rangle_M  \right)^2, \label{eq:hatf_cont}
\end{align}
which is purely based on the fully discrete observations.
Under Assumption (H), it is possible to derive a convenient and intuitive representation for our estimator $\hat f_m$ based on the following lemma.

\begin{lem} \label{lem:L2norm_approx}
Let  $H\colon[0,1]\to \R$ be H\"{o}lder continuous in a neighborhood of $y_k$ for each $1\leq k \leq M-1$ and set $h_k := \langle H,e_k \rangle_{L^2}$. Then, the series $H_k :=\sum_{l \in \mathcal I_k^+}h_l-\sum_{l \in \mathcal I_k^-}h_l$ converges and we have  
$\langle H ,e_k \rangle_M = H_k $ as well as
$$\frac{1}{M}\sum_{k=1}^{M-1} H^2 (y_k )= \Vert H^M\Vert_{L^2}^2=\sum_{l=1}^{M-1}H_l^2$$ 
 with $H^M := \hat S(0) H=\sum_{l=1}^{M-1}H_l e_l$.
\end{lem}

Under Assumptions (H) and (E), the random variables $X_{t_i}(y_k)$ hit a discontinuity of the extension $\bar g$ of some $g\in \bigcup_{m\in \N}\mathcal V_m$ with probability zero and, hence, the above lemma is applicable with $H:= \frac{X_{t_{i+1}}-\hat S(\Delta)X_{t_i}}{\Delta}  -g(X_{t_i})$. In particular, the estimator $\hat f_m$ can, almost surely, be expressed via 
\begin{align} \label{eq:hatfm_discrete}
\hat f_m &= \argmin_{g\in \mathcal V_m} \Gamma_{N,M}(g),\qquad \Gamma_{N,M}(g):= \frac{1}{N M}\sum_{i=0}^{N-1} \sum_{k=1}^{M-1}\Big( \frac{X_{t_{i+1}}(y_k)-S^\Delta_{t_i}(y_k)}{\Delta}  -g(X_{t_i}(y_k)) \Big)^2,
\end{align}
where $ S^\Delta_{t_i}:=\hat S(\Delta) X_{t_{i}}$.

The natural empirical norm associated with the discrete observations scheme is given by
$$\Vert g\Vert_{N,M}^2 := \frac{1}{NM}\sum_{i=0}^{N-1}\sum_{k=1}^{M-1} g^2(X_{t_i}(y_k))$$
and, in the sequel, we derive a bound on $\E(\Vert \hat f_m -f_{\mathcal A} \Vert_{N,M}^2)$ with $f_{\mathcal A} := f \1_{\mathcal A}$. As before, $\pi$ denotes the stationary distribution for $X$ and, for nonrandom $g\in L^2(\mathcal A)$, let
$$\Vert g \Vert_{\pi,M}^2 := \frac{1}{M} \sum_{k=1}^{M-1} \E\big( g^2(X_0(y_k))\big).$$
Due to Assumption (E), there are constants $c,C>0$ such that 
\begin{equation} \label{eq:(E)implication_cont}
c\Vert g\Vert_{L^2(\mathcal A)}^2 \leq \Vert g \Vert_{\pi,M}^2\leq C \Vert g\Vert_{L^2(\mathcal A)}^2 
\end{equation}
holds for all $g\in L^2(\mathcal A)$.
The oracle choice for an estimate of $f_{\mathcal A}$ from the space $\mathcal V_m$ is defined by 
\begin{equation*} 
f_m^* := \argmin_{g\in \mathcal V_m} \Vert f-g \Vert_{L^2(\mathcal A)}^2.
\end{equation*}

\begin{thm} \label{thm:nonpara_discrete}
Grant  Assumptions (F), (M), (E), (N) and (H). Assume that $M\Delta^2 \to \infty$ as well as $\frac{N\Delta}{\log^2 N} \to \infty$, $L_m=o(\frac{N\Delta}{\log^2 N})$ and $D_m\leq N$. Then, for any $\gamma<1/2$, we have
$$\E\big(\Vert \hat f_m - f_{\mathcal A}\Vert_{N,M}^2\big)\lesssim \Vert f- f_m^* \Vert_{L^2({\mathcal A})}^2 + \frac{D_m}{T} + \Delta^\gamma  +\frac{1}{M\Delta^2}.$$
\end{thm}

\begin{rem}
Under the same assumptions as in the above theorem, we can obtain the oracle inequality
$$ \frac{1}{N}\sum_{i=0}^{N-1}\E\big(\Vert \hat f_m^{\mathrm{cont}}(X_{t_i}) - f_{\mathcal A}(X_{t_i})\Vert_{L^2}^2\big)\lesssim \Vert f- f_m^* \Vert_{L^2({\mathcal A})}^2 + \frac{D_m}{T} + \Delta^\gamma $$
for the estimator $\hat f_m^{\mathrm{cont}}$ from \eqref{eq:hatf_cont} based on space-continuous observations.  In fact, this result can be obtained without the continuity Assumption (H).
\end{rem}
We encounter the usual bias-variance trade-off in nonparametric statistics: When $m$ is too small the estimator is not sufficiently versatile, leading to a large bias term $\Vert f- f_m^* \Vert_{L^2({\mathcal A})}^2$. On the other hand, when $m$ is too large, the estimated function will suffer from overfitting, resulting in a large variance term $D_m/T$. 
Assuming that $\Vert f_m^* -f\Vert_{L^2({\mathcal A})} \eqsim D_m^{-\alpha}$ for some $\alpha >0$, balancing the bias and the variance term leads to the optimal choice $D_m \eqsim T^{\frac{1}{1+2\alpha}}$. Under the additional assumption that 
\begin{equation} \label{eq:remainder_negligible}
T\Big( \Delta^\gamma  +\frac{1}{M\Delta^2}\Big) \to 0
\end{equation}
holds for some $\gamma <1/2$, the last two terms on the right hand side of the oracle inequality in Theorem \ref{thm:nonpara_discrete} are negligible and we obtain the usual (squared) nonparametric rate
$$\E\big(\Vert \hat f_m - f_{\mathcal A}\Vert_{N,M}^2\big)\lesssim T^{-\frac{2\alpha}{2\alpha+1}}. $$
Some caution is necessary in order to prevent a contradiction between $D_m\eqsim T^{\frac{1}{1+2\alpha}}$ and the condition $L_m=o\big(\frac{T}{\log^2N}\big)$ in the theorem, as already pointed out in \cite{Comte07}. When working with [P] or [W], we have $L_m \eqsim D_m$ and the conditions can be met at the same time. When working with the trigonometric spaces [T], we have $L_m \eqsim D_m^2$ and, thus, it is only possible to take $D_m \eqsim T^{\frac{1}{1+2\alpha}}$, provided that $\alpha>1/2$. For a function $f$ meeting our fundamental requirement $f \in C^1(\R)$, the $k$-th Fourier coefficients  are generally of the order $1/k$, a faster decay is only present in the exceptional case where the function $f$ is periodic on ${\mathcal A}$. Thus, we have $\Vert f_m^*-f \Vert_{L^2({\mathcal A})}\eqsim D_m^{-\alpha}$ with $\alpha = 1/2$ and it is still possible to achieve a convergence rate of $T^{-\frac{ \tilde \alpha}{2\tilde \alpha +1}}$ for any $\tilde \alpha< \alpha$. When using the approximation spaces [P], the decay $\Vert f-f_m^* \Vert_{L^2({\mathcal A})}\lesssim D_m^{-\alpha}$ can be ensured by assuming that $f_{\mathcal A}$ belongs to the Besov space $\mathcal B^\alpha_{2,\infty}({\mathcal A})$ with $\alpha < r_{\max} + 1$, see Theorem 7.3 in \cite[Chapter 7]{Devore93}. This gives the following:

\begin{cor}
Additionally to the assumptions of Theorem \ref{thm:nonpara_discrete} grant \eqref{eq:remainder_negligible} and $f_{\mathcal A}\in\mathcal B_{2,\infty}^\alpha({\mathcal A})$ for some $\alpha>0$. If one uses the approximation spaces $[P]$ with $r_{\max}>\alpha-1$ and $D_m\eqsim T^{\frac{1}{2\alpha+1}}$, we have
$$ \E\big(\Vert \hat f_m - f_{\mathcal A}\Vert_{N,M}^2\big)\lesssim T^{-\frac{2\alpha}{2\alpha+1}}.$$
\end{cor}
To ensure $\Vert f-f_m^* \Vert_{L^2({\mathcal A})}\lesssim D_m^{-\alpha}$  based on $f\in \mathcal B^\alpha_{2,\infty}({\mathcal A})$ when working with [W] or [T], respectively, one generally needs the additional assumption of periodicity or compact support in ${\mathcal A}$, respectively, see \cite{Birge97}.

In practice, the true value of the regularity parameter $\alpha$ is unknown, as it is a property of the unknown function $f$. This issue will be addressed via an adaptive procedure that chooses the approximation space $\mathcal V_m$ in a data driven way, see Theorem \ref{thm:adaptive}. 

The error term $\Delta^\gamma,\,\gamma<1/2$, in the oracle inequality bounds the remainder $R_i^{\mathrm{cont}}$ in the underlying regression model. In the corresponding result for SODEs \citet{Comte07} obtain instead the smaller bound $\Delta$.
The difference in the order of magnitude is due to the fact that for the SPDE model there only  is  temporal H\"{o}lder regularity up to exponent $1/4$, as opposed to exponent $1/2$ in the finite-dimensional setting. 

The last term on the right hand side of the  oracle inequality in Theorem \ref{thm:nonpara_discrete} is caused by the approximation error $R_i-R_i^{\mathrm{cont}}$ of the continuous model, which is not present in the SODE setting. More precisely, to approximate the semigroup, we estimate the Fourier coefficient processes of $X$ by their empirical counterparts. Due to the roughness  of the paths $x\mapsto X_t(x)$, the corresponding approximation quality  is rather poor. The resulting estimation error is of the order $\E(\Vert S(h)X_t-\hat S(h)X_t \Vert_{L^2}^2)=\O(1/M)$ for $h\ge0$.
A similar effect occurs in \citet{Uchida19} where a spectral approximation is used for parametric estimation for the linear equation. The approximation error of the order $\O(1/M)$ gets further amplified by dividing by the squared renormalization $\Delta^{2}$, leading to the condition $M \Delta^2 \to \infty$. Under this condition, the  observation frequency in space is much larger than in time which, in particular,  rules out a balanced sampling design $\delta \eqsim \sqrt \Delta$, see Section \ref{sec:parametric_nonlinear} below.  The additional error term $f(X_t)-\hat S(0)f(X_t)$ included in $R_i-R_i^{\mathrm{cont}}$  turns out to be negligible with respect to $\Delta^\gamma$ under the condition $M \Delta^2 \to \infty$.
 
As an illustration of the method Figure~\ref{fig:nonparametric} shows ten exemplary realizations of the estimator $\hat f_m$ with the trigonometric basis [T]  when $f$ is the polynomial $f(x):= - x^3+x/5$. The compact set on which $f$ is estimated is ${\mathcal A}= [-1,1]$. The general shape of the function $f$ is  captured accurately inside some interval containing the origin, roughly $[-0.5,0.5]$. It is evident from the histogram that areas further away from the origin do not contain as many data points which, naturally, affects the quality of the estimator there. Also, there is a boundary effect caused by the fact that the functions in $\mathcal V_m$ are necessarily periodic over $[-1,1]$.\\
  
\begin{figure}[t]
\centering
\input{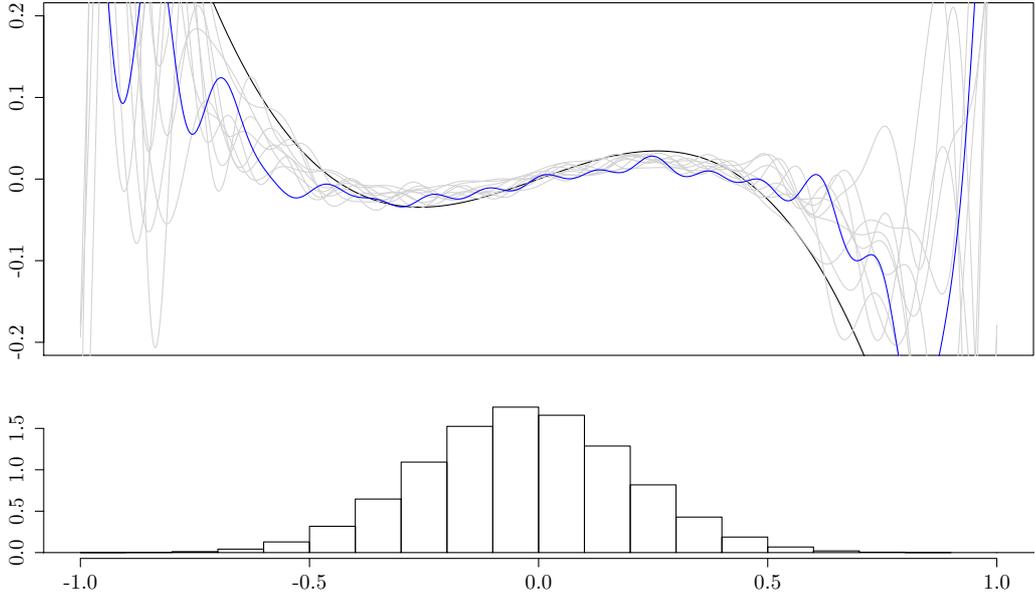}
\caption{Ten realizations of the estimator $\hat f_m$ from \eqref{eq:hatfm_discrete} with the trigonometric basis [T] on ${\mathcal A}=[-1,1]$ (blue or gray) along with the true underlying function $f(x) = -x^3+0.2\cdot x$ (black). The barplot shows a histogram of the corresponding discrete observations $\{X_{t_i}(y_k)\}_{i,k}$.  The sample sizes are given via $M=800,\,T=250,\,\Delta = 0.05$. The dimension of the approximation space was chosen to be $D_m=2m+1=33$, which corresponds to  $D_m \eqsim \sqrt T$. The discrete observations of $X$ are obtained by the exponential Euler method  with spectral cut-off at 1000. The parameter values are $\sigma=0.05$ and $\vt =0.01$.}
\label{fig:nonparametric}
\end{figure}

In order to prove Theorem \ref{thm:nonpara_discrete}, we adapt the proof strategy from \cite{Comte07} to our infinite-dimensional setting. The main steps of the proof are explained in the following: 
For an arbitrary function $g \in \bigcup_{m\in \N}\mathcal V_m$, we can use Lemma \ref{lem:L2norm_approx} and the regression model \eqref{eq:Reg_model_discrete} to write
\begin{align*}
\Gamma_{N,M}(g)-\Gamma_{N,M}(f) &= \Vert g-f \Vert_{N,M}^2 + \frac{2}{N}\sum_{i=0}^{N-1} \langle Y_i-\hat S(0)f(X_{t_i}), \hat S(0)f(X_{t_i})-\hat S(0)g(X_{t_i})\rangle_{L^2} \\
&= \Vert g-f \Vert_{N,M}^2 + \frac{2}{N}\sum_{i=0}^{N-1} \langle \eps_i+R_i, \hat S(0)f(X_{t_i})-\hat S(0)g(X_{t_i})\rangle_{L^2}.
\end{align*}
By definition of $\hat f_m$, we have $\Gamma_{N,M}(\hat f_m)-\Gamma_{N,M}(f)\leq \Gamma_{N,M}( f_m^*)-\Gamma_{N,M}(f)$ and using the above expansion on both sides of this inequality yields
\begin{align*}
\Vert \hat f_m -f \Vert^2_{N,M} \leq \Vert f_m^* -f \Vert^2_{N,M} + \frac{2}{N}\sum_{i=0}^{N-1} \langle \eps_i +R_i,\hat S(0)\hat f_m(X_{t_i})-\hat S(0)f_m^*(X_{t_i})\rangle_{L^2}.
\end{align*}
Since both $\hat f_m$ and $f_m^*$ are ${\mathcal A}$-supported, if we insert $f= f\1_{\mathcal A} +f\1_{{\mathcal A}^c}$ in the above equation, then the terms $\Vert f1_{{\mathcal A}^c}\Vert_{N,M}^2$ on both sides of the inequality cancel. We arrive at the fundamental oracle inequality
\begin{align}
\Vert \hat f_m - f_{\mathcal A}  \Vert^2_{N,M} \leq \Vert f_m^* -f_{\mathcal A} \Vert^2_{N,M} + & \frac{2}{N}\sum_{i=0}^{N-1} \langle \eps_i,\hat S(0)\hat f_m(X_{t_i})-\hat S(0)f_m^*(X_{t_i})\rangle_{L^2} \nonumber\\
+&\frac{2}{N}\sum_{i=0}^{N-1} \langle R_i,\hat S(0)\hat f_m(X_{t_i})-\hat S(0)f_m^*(X_{t_i})\rangle_{L^2}. \label{eq:oracle_prelim}
\end{align}
By treating each of the three terms appearing on the right hand side above individually, we can derive the following proposition.

\begin{prop} \label{prop:SpaceDiscrete}
Grant Assumptions (F), (M), (N) and (H) and assume $M\Delta^2 \to \infty$. For $\underline{c}>0$, define the event
\begin{align*}
\Omega_{N,M,m}&:=\Omega_{N,M,m,\underline{c}}:= \left\{ \Vert g \Vert_{N,M}^2 \geq \underline{c}\Vert g\Vert_{L^2({\mathcal A})}^2  \text{ for all } g \in \mathcal V_m\right\}.
\end{align*}
Then, for any $\gamma <1/2$, we have
$$\E \Big( \Vert \hat f_m - f_{\mathcal A} \Vert_{N,M}^2\1_{\Omega_{N,M,m}}\Big) \lesssim \Vert f_{\mathcal A}- f_m^* \Vert_{\pi,M}^2 + \frac{D_m}{T} + \Delta^\gamma  +\frac{1}{M \Delta^2}. $$
\end{prop}

In the proof we have to bound the stochastic noise term $\frac{1}{N}\sum_{i=0}^{N-1} \langle \eps_i,\hat S(0) g(X_{t_i}) \rangle _{L^2}$ uniformly over all $\Vert \cdot \Vert_{N,M}$-normalized $g\in \mathcal V_m$. This is difficult since both the object to be bounded and the norm are random objects. However, on the event $\Omega_{N,M,m}$, it is sufficient to bound it uniformly over all $\Vert \cdot \Vert_{L^2({\mathcal A})}$-normalized $g\in \mathcal V_m$ which is possible thanks to Assumption (N).

Under Assumption (E), we can further bound $\Vert f_{\mathcal A}- f_m^* \Vert_{\pi,M}^2\lesssim \Vert f- f_m^* \Vert_{L^2({\mathcal A})}^2$, hence,
Proposition \ref{prop:SpaceDiscrete} already provides the relevant terms appearing in the oracle inequality from Theorem \ref{thm:nonpara_discrete}. The second main step of the proof of the theorem is to verify that the event $\Omega_{N,M,m}^c$ has negligible probability. To that aim, let us consider the event 
$$ \Xi_{N,M,m} := \left\{\Big| \frac{\Vert g\Vert^2_{N,M}}{\Vert g\Vert^2_{\pi,M} }-1\Big|\leq \frac{1}{2}\, \forall g \in \mathcal V_m \right\}$$
which satisfies $\Xi_{N,M,m} \subset \Omega_{N,M,m,\frac{1}{2}}$. Since $\E(\Vert g \Vert_{N,M}^2 )= \Vert g \Vert_{\pi,M}^2$ and $$\Big| \frac{\Vert g\Vert^2_{N,M}}{\Vert g\Vert^2_{\pi,M} }-1\Big|\eqsim \Big|\frac{\Vert g\Vert^2_{N,M}-\Vert g\Vert^2_{\pi,M} }{\Vert g\Vert^2_{L^2({\mathcal A})} }\Big|$$ under Assumption (E), bounding the probability of $\Xi_{N,M,m}^c$ is equivalent to deriving a concentration inequality  for $\Vert g \Vert_{N,M}^2 $ uniformly over all $L^2({\mathcal A})$-normalized $g\in \mathcal V_m$. Therefore, we apply the Bennett inequality under strong mixing by \citet{Rio2017}.

\begin{lem} \label{lem:normequivalence_discrete}
Grant Assumptions (F), (M), (E), (N) and (H). Then, there are constants $K,K'>0$ such that   
$$ \P(\Xi_{N,M,m}^c)\leq K D_m^2\big(\e^{-K' N/(q_NL_m)}+L_m\e^{-\gamma q_N \Delta} \big)$$
holds for any $q_N\in \N$ satisfying $q_N=o(N/L_m)$. In particular, if $\frac{N\Delta}{\log^2 N} \to \infty$, $L_m = o(\frac{N\Delta}{\log^2 N})$ and $D_m\le N$, we have $\P(\Xi_{N,M,m}^c)\lesssim N^{-\gamma}$ for any $\gamma>0$. 
\end{lem}
The conclusion of the main theorem is a direct consequence of Proposition \ref{prop:SpaceDiscrete} and Lemma \ref{lem:normequivalence_discrete}, see Section~\ref{subsec:nonpara_proofs} for further details on the proof. \\


Next, we assess the quality of $\hat f_m$ in terms of the more intuitive distance measure $\Vert \hat f_m-f\Vert_{L^2({\mathcal A})}$, rather then  $\Vert \hat f_m-f_{\mathcal A}\Vert_{N,M}$.
Using the triangle inequality as well as the equivalence of the empirical and the $L^2({\mathcal A})$-norm on $\Xi_{N,M,m}$, we can bound
\begin{align*}
\Vert \hat f_m - f \Vert_{L^2({\mathcal A})}^2&\leq 2\Vert \hat f_m - f_m^* \Vert_{L^2({\mathcal A})}^2+2\Vert f_m^* - f \Vert_{L^2({\mathcal A})}^2\\
&\eqsim 2\Vert \hat f_m - f_m^* \Vert_{N,M}^2 \1_{\Xi_{N,M,m}}+2\Vert \hat f_m - f_m^* \Vert_{L^2({\mathcal A})}^2 \1_{\Xi_{N,M,m}^c}+2\Vert f_m^* - f \Vert_{L^2({\mathcal A})}^2.
\end{align*}
Thanks to Proposition \ref{prop:SpaceDiscrete} and Lemma \ref{lem:normequivalence_discrete}, it is straightforward, to derive an upper bound in probability.
Bounding $\E(\Vert \hat f_m - f \Vert_{L^2({\mathcal A})}^2)$, on the other hand, is more challenging since the behavior of $\Vert \hat f_m - f \Vert_{L^2({\mathcal A})}^2$ on the set $\Xi_{N,M,m}^c$ is a priori unclear. This issue can be circumvented by considering the truncated version 
$$\tilde f _{m} := (-N)\vee (\hat f_m\wedge N).$$
 
\begin{cor} \label{cor:nonpara_discrete} 
Grant  Assumptions (F), (M), (E), (N) and (H). Further, assume that $M\Delta^2 \to \infty$ as well as $\frac{N\Delta}{\log^2 N} \to \infty$, $L_m=o(\frac{N\Delta}{\log^2 N})$ and $D_m\leq N$. Then, for any $\gamma<1/2$, we have
\begin{align*}
  \Vert \hat f_m - f\Vert_{L^2({\mathcal A})}^2 &= \O_p \Big( \Vert f- f_m^* \Vert_{L^2({\mathcal A})}^2 + \frac{D_m}{T} + \Delta^\gamma +\frac{1}{M\Delta^2} \Big),\\
\E\big(\Vert \tilde f_m - f\Vert_{L^2({\mathcal A})}^2\big) &\lesssim \Vert f- f_m^* \Vert_{L^2({\mathcal A})}^2 + \frac{D_m}{T} + \Delta^\gamma +\frac{1}{M\Delta^2}.
\end{align*}
\end{cor}

In order to choose  an appropriate approximation space $\mathcal V_m$ in a purely data-driven way, let $\mathcal{M}_{N}:=\{1,\dots,\bar{m}\}$ be the indexes of a sequence
of approximation spaces $\mathcal V_{m}$ with dimensions $D_{m}$, subject to the nesting assumption $\mathcal V_m \subset \mathcal V_{\bar m}$ for $m \in \mathcal{M}_{N}$. We define the model
selection method 
\begin{equation}\label{eq:mHat}
\hat{m}:=\argmin_{m\in\mathcal{M}_{N}}\big\{\Gamma_{N,M}(\hat{f}_{m})+\pen(m)\big\}\qquad\text{with}\qquad\pen(m):=\kappa\frac{\sigma^{2}D_{m}}{T}
\end{equation}
for some appropriate constant $\kappa>0$. Note that the penalty
term is of the order of the stochastic error in the underlying regression
problem such that $\hat m$ automatically balances the deterministic approximation error and the stochastic error. The parameter $\sigma^{2}$ in the penalty can be replaced by some a priori or $\hat{\sigma}^{2}$-dependent upper bound
of the volatility. The adaptive estimator
for the reaction function is then $\hat{f}_{\hat{m}}$ fulfilling 
\[
(\hat m,\hat{f}_{\hat{m}})=\argmin_{m\in\mathcal{M}_{N},f_{m}\in \mathcal V_{m}}\big\{\Gamma_{N,M}(f_{m})+\pen(m)\big\}.
\]

\begin{thm} \label{thm:adaptive}
Grant Assumptions (F), (M), (E), (N) and (H) and let $\kappa$ be sufficiently
large (depending only on the constants in Assumption (E)). Further, assume that $M\Delta^{2}\to\infty$
as well $\frac{N\Delta}{\log^{2}N}\to\infty$, $L_{\bar{m}}=o(\frac{N\Delta}{\log^{2}N})$
and $D_{\bar{m}}\le N$. Then, for any $\gamma<1/2$ we have
\begin{align*}
\E(\|\hat{f}_{\hat{m}}-f_{\mathcal A}\|_{N,M}^{2}\big) & \lesssim\inf_{m\in\mathcal{M}_{N}}\Big\{\|f_{m}^*-f\|_{L^{2}({\mathcal A})}^{2}+\frac{D_{m}}{T}\Big\}+\frac{1}{M\Delta^{2}}+\Delta^{\gamma}.\\
\E(\|\tilde{f}_{\hat{m}}-f_{\mathcal A}\|_{L^2({\mathcal A})}^{2}\big) & \lesssim\inf_{m\in\mathcal{M}_{N}}\Big\{\|f_{m}^*-f\|_{L^{2}({\mathcal A})}^{2}+\frac{D_{m}}{T}\Big\}+\frac{1}{M\Delta^{2}}+\Delta^{\gamma}.
\end{align*}
\end{thm}

\subsection{Estimation of $f$ with unknown diffusivity and volatility}
While our nonparametric estimator for $f$ does not hinge on the volatility parameter $\sigma^2$, the diffusivity parameter $\vt$ enters the least square criterion in \eqref{eq:hatfm_discrete} via the expressions $$S^\Delta_{t_i}(y_k) = \sum_{\ell=1}^{M-1} \e^{-\pi^2 \vt \ell^2}\langle X_{t_i},e_\ell \rangle_M \,e_\ell(y_k).$$
In practice, $\vartheta$ is typically unknown and has to be replaced by an estimate $\hat \vartheta$. 
Based on that, we can  define an approximation $\check S (\Delta)$ of the discretized semigroup $\hat S(\Delta)$, namely
$$\check{S}(\Delta)u := \sum_{\ell =1}^{M-1} \e^{-\hat \lambda_\ell \Delta} \langle u,e_\ell \rangle_M e_\ell \quad\text{with}\quad\hat \lambda_\ell := \pi^2 \hat \vt \ell^2$$
for continuous functions $u\colon[0,1]\to \R$. The resulting nonparametric estimator for $f$ is then given by 
\begin{align*}
\check f_m &:= \argmin_{g\in \mathcal V_m}  \frac{1}{N M}\sum_{i=0}^{N-1} \sum_{k=1}^{M-1}\Big( \frac{X_{t_{i+1}}(y_k)-\check S^\Delta_{t_i}(y_k)}{\Delta}  -g(X_{t_i}(y_k)) \Big)^2
\end{align*}
where $ \check S^\Delta_{t_i}:=\check S(\Delta) X_{t_{i}}$. The counterpart to $\hat m$ from \eqref{eq:mHat} where $\hat S$ is replaced by $\check S$ will be denoted by $\check m$. In order to analyze the convergence rates of  $\check f_m$ and $\check f_{\check m}$, we incorporate the approximation of the semigroup into the regression model. Due to $\hat S(0)=\check S(0)$  and in view of \eqref{eq:Reg_model_discrete}, we obtain 
\begin{equation*} 
\frac{\hat S(0) X_{t_{i+1}}-\check S(\Delta)X_{t_i}}{\Delta} = \hat S(0)f(X_{t_i}) +  R_i' +\eps_i\quad\text{with}\quad R_i' :=  R_i + \frac{\hat S(\Delta)X_{t_i}-\check S(\Delta)X_{t_i}}{\Delta}.
\end{equation*}
Relying on this representation, we can show that the {estimation} of the discretized semigroup does not affect the convergence rate of the nonparametric estimator. 
\begin{thm} \label{thm:nonparametric_plugin}
Grant  Assumptions (F), (M), (E), (N) and (H). Further, assume that $M\Delta^2 \to \infty$ as well as $\frac{N\Delta}{\log^2 N} \to \infty$, $L_{\bar m}=o(\frac{N\Delta}{\log^2 N})$ and $D_{\bar m}\leq N$. Let $\hat \vartheta$ be an estimator for $\vartheta$ satisfying $(\hat\vartheta-\vartheta)^2={\mathcal O_p(\Delta^{3/2}/T)} $ . Then, for any $\gamma<1/2$, we have
\begin{align*}
\Vert \check f_m - f\Vert_{L^2({\mathcal A})}^2 &= \O_p \Big( \Vert f_m^*- f \Vert_{L^2({\mathcal A})}^2 + \frac{D_m}{T}  +\frac{1}{M\Delta^2}+ \Delta^\gamma \Big),\\
\Vert \check f_{\check m} - f\Vert_{L^2({\mathcal A})}^2 &= \O_p \Big(\inf_{m\in\mathcal{M}_{N}}\Big\{\|f_{m}^*-f\|_{L^{2}({\mathcal A})}^{2}+\frac{D_{m}}{T}\Big\}+\frac{1}{M\Delta^{2}}+\Delta^{\gamma} \Big). 
\end{align*}
The same bounds also hold for $\Vert \check f_m - f_{\mathcal A}\Vert_{N,M}^2$ and $\Vert \check f_{\check m} - f_{\mathcal A}\Vert_{N,M}^2$, respectively.
\end{thm}
In the next section we will construct an estimator $\hat\vt$ with a faster convergence rate than required in Theorem~\ref{thm:nonparametric_plugin}, see Remark~\ref{rem:plugin}.


\section{Diffusivity and volatility estimation}
\label{sec:parametric_nonlinear}
Exploiting the analysis of the H\"{o}lder regularity of the linear and the nonlinear component of the solution process, we can generalize the central limit theorems for space, double and time increments from \cite{Hildebrandt19} and \cite{Bibinger18}, respectively, to the semilinear framework. As a consequence, resulting method of moments estimators for the squared volatility $\sigma^2$ and the diffusivity $\vt$ apply in the semilinear framework and, under quite general assumptions, their asymptotic properties remain unchanged.
In the sequel we assume $b>0$ in the observation scheme \eqref{eq:obs}, so that Proposition \ref{prop:AN_reg} provides the regularity of the process $(N_t(x),\,x\in [b,1-b],\,t\geq 0)$ in space and time. In this section we also allow for a fixed time horizon $T$.\\

First, let us consider the realized quadratic variation based on time increments
	$$ V_{\mathrm{t}} :=\frac{1}{MN\sqrt \Delta}\sum_{i=0}^{N-1} \sum_{k=0}^{M-1} (X_{t_{i+1}}(y_k)-X_{t_{i}}(y_k))^2.$$
	
		\begin{thm}\label{thm:CLT_time_nonlinear}
	Grant Assumption (F) and suppose $TM= o(\Delta^{-\rho})$ for some $\rho<1/2$. 
	 If either $T$ is fixed and finite or Assumption~(B) is satisfied, then 
	 \begin{equation*}
\sqrt{MN}\left(V_{\mathrm{t}}-\frac{\sigma^2}{\sqrt{\pi \vt}}\right)\overset{d}{\longrightarrow}\mathcal N \left(0,  \frac{B\sigma^4}{\pi \vt} \right),\quad N,M \to \infty,
\end{equation*}
holds with $B:= 2+\sum_{J=1}^\infty \left(2\sqrt{J}  -\sqrt{J+1}-\sqrt{J-1}\right)^2 .$ 
\end{thm}	 
	  For the case $f\equiv 0$, the above central limit theorem is shown in \citet[Theorem 3.4]{Bibinger18} for fixed $T$ under the same assumption on the interplay of $M$ and $\Delta$. Their theorem can be directly generalized to $T\to \infty$ when assuming $M= o(\Delta^{-\rho})$ for some $\rho<1/2$ and  $TM= o(\Delta^{-1})$.
 A central limit theorem for time increments in the case $T \to \infty$ is also proved by \citet{Kaino20}.
Clearly, for $T\to \infty$ the assumption $TM= o(\Delta^{-\rho})$ for the nonlinear case is {considerably stricter}. In the proof of the above theorem, we show that $R_\mathrm{t}:= V_\mathrm{t}- \bar V_\mathrm{t}=o_p(1/\sqrt{MN})$, where $\bar V_\mathrm{t}$ is defined as $ V_\mathrm{t}$ with $f\equiv 0$. This proves the result in view of Slutsky's lemma. In fact, it follows from the temporal regularity properties of the processes $(X^0_t)$ and $(N_t)$, that $R_\mathrm{t}$ is of the order $\O_{p}(\Delta^{\alpha})$ for any $\alpha<3/4$. Hence, $\sqrt{MN}\Delta^{\alpha}=\sqrt{MT\Delta^{2\alpha-1}}$ is required to tend to 0.
	
Next, we consider the realized quadratic variation based on space increments
	$$ V_{\mathrm{sp}} :=\frac{1}{MN\delta}\sum_{i=0}^{N-1} \sum_{k=0}^{M-1} (X_{t_{i}}(y_{k+1})-X_{t_{i}}(y_k))^2.$$
Since the terms indexed by $i=0$ do not contribute to the sum if $X_0=0$, in this case, we sum over $i \in \{1,\ldots,N\}$ instead of $\{0,\ldots,N-1\}$. 
\begin{thm} \label{thm:CLT_space_nonlinear}
	 Grant Assumption (F) and let $N=o(M)$. If either $T$ is fixed and finite or Assumption (B) is satisfied, then we have 
	 \begin{equation*} 
\sqrt{MN}\left(V_\mathrm{sp}-\frac{\sigma^2}{2\vt}\right)\overset{d}{\longrightarrow} \mathcal{N}\left(0,\frac{\sigma^4}{2\vt^2}\right),\quad M,N \to \infty.
\end{equation*}
	\end{thm} 
	The above theorem is proved in \citet[Theorem 3.3]{Hildebrandt19} for the case $f\equiv 0$. Although our proof strategy  for the generalization to $f\neq  0$ is the same as for time increments, here, the result carries over from the linear setting without any extra conditions on $M,N$ and $T$. Indeed, defining $\bar V_\mathrm{sp}$ in the obvious way and using a summation by parts formula to rewrite $R_\mathrm{sp}:= V_\mathrm{sp}-\bar V_\mathrm{sp}$, we can profit from the fact that the second order spatial increments of $(N_t)$, namely $N_{t_i}(y_{k+1})-2N_{t_i}(y_{k})+N_{t_i}(y_{k-1})$, are of the order $\O_{p}(\delta^2)$, thanks to the spatial regularity of the process $(N_t)$.

Finally, we consider the realized quadratic variation based on double increments
$$ {\mathbb V}:=\frac{1}{MN\Phi_{\vt}(\delta,\Delta)} \sum_{i=0}^{N-1} \sum_{k=0}^{M-1}  D_{ik}^2 $$
with $D_{ik}:=X_{t_{i+1}}(y_{k+1})-X_{t_{i}}(y_{k+1})-X_{t_{i+1}}(y_{k})+X_{t_{i}}(y_{k})$ and the renormalization  
$$\Phi_\vt (\delta,\Delta):=2\sum_{\ell\geq1}\frac{1-\e^{-\pi^{2}\vt\ell^{2}\Delta}}{\pi^{2}\vt\ell^{2}}\big(1-\cos(\pi \ell \delta)\big)\eqsim \delta \wedge \sqrt \Delta.\label{eq:PhiTheta}$$
As discussed in \cite{Hildebrandt19} for the linear case, if a so called \emph{balanced sampling design} is present, i.e.~$\delta/\sqrt \Delta \equiv r$ for some $r>0$, we can also consider
$$ {\mathbb V}_r:=\frac{1}{MN\sqrt \Delta} \sum_{i=0}^{N-1} \sum_{k=0}^{M-1}  D_{ik}^2.$$
In the case $X_0=0$, $ {\mathbb V}$ and ${\mathbb V}_r$ are redefined just like $ V_{\mathrm{sp}}$.

\begin{thm} \label{thm:CLT_double_nonlinear}
Grant Assumption (F) and suppose $\Delta \to 0$ as well as $T=o(M^a)$ for some $a\in (0,1)$.
\begin{enumerate}[(i)]
\item If $\delta/\sqrt \Delta \to r\in \{0,\infty\}$ or $\delta/\sqrt \Delta \equiv r>0$, then
\begin{equation*} 
 \sqrt{MN}(\mathbb V-\sigma^2)\overset{d}{\longrightarrow}\mathcal N \big(0, C\big({r/\sqrt{\vt}}\big)\sigma^4 \big),\quad N,M \to \infty,
 \end{equation*}
holds for some bounded and strictly positive continuous function $ C(\cdot)$ on $[0,\infty]$. 
\item If $\delta/\sqrt \Delta \equiv r>0$, then we have
\begin{equation*} 
\sqrt{MN}\Big(\mathbb V_r-\psi_{\vt}(r)\sigma^2\Big)\overset{d}{\longrightarrow}\mathcal N \Big(0, C\big(r/\sqrt{\vt}\big)\psi_{\vt}^2(r)\sigma^4 \Big),\quad N,M \to \infty,
\end{equation*}
where $
\psi_{\vt}(r) 
:=\frac{2}{\sqrt{\pi\vt}}\Big(1-\e^{-\frac{r^2}{4\vt}}+\frac{r}{\sqrt{\vt}}\int_{\frac{r}{2\sqrt{\vt}}}^\infty\e^{-z^2}\,dz \Big).
$
\end{enumerate}
\end{thm}

For the case $f\equiv 0$, the above result is proved in Theorem 3.7 and Corollary 3.8 of \cite{Hildebrandt19}, to where we also refer for an explicit expression for the function $C(\cdot)$.
As for space increments, there are essentially no additional assumptions compared to the linear setting ($a=1$ is allowed there). The influence induced by the nonlinearity turns out to be negligible since the double increments computed from the process $(N_t)$ decay in both $\Delta$ and $\delta$ at the same time, as opposed to the double increments computed from $(X^0_t)$ which are of the order $(\delta \wedge \sqrt \Delta)^{1/2}$, see Lemma \ref{lem:Nik_pnorm}.\\

It is straightforward to derive asymptotically normal method of moments estimators for $\sigma^2$ or $\vt$ based on the above central limit theorems when one of the parameters is known, as discussed in, e.g., \cite{Hildebrandt19,Bibinger18, Cialenco17, Chong18}. Joint estimation of the parameters $(\sigma^2,\vt)$ remains possible in the semilinear framework as well by exploiting the central limit theorem for $ {\mathbb V}_r$. To that aim, one needs to revert to subsets of the data having a balanced sampling design $\tilde \delta/ \sqrt {\Tilde \Delta } \equiv r$ for two different values of $r$. Let us briefly recall the estimation procedure from \cite{Hildebrandt19}:

Choosing $v,w \in \N$ such that $v\eqsim \max(1,\delta^2/\Delta)$ and $w\eqsim \max(1,\sqrt\Delta/\delta)$, we have $r:= {\tilde \delta}/{\sqrt{\tilde \Delta}} \eqsim 1$ for $\tilde \Delta:=v\Delta$ and $\tilde \delta := w \delta$.
Using double increments on the coarser grid, namely 
$$ D_{v,w}(i,k):=X_{t_{i+v}}(y_{k+w})-X_{t_{i}}(y_{k+w})-X_{t_{i+v}}(y_{k})+X_{t_{i}}(y_{k}),$$ 
we set
$$V^\nu :=\frac{1}{(M-w+1)(N-\nu v+1)\sqrt{\nu v \Delta}}\sum_{k=0}^{M-w}\sum_{i=0}^{N-\nu v} D^2_{\nu v,w}(i,k),\qquad \nu=1,2.$$
In the case $X_0=0$ we employ the obvious redefinition of $ V^\nu$. The final estimator for $(\sigma^2,\vt)$ is
\begin{align} \label{eq:def_LSestimator.1}
(\hat \sigma^2, \hat \vt) := \argmin_{(\tilde \sigma^2,\tilde \vt) \in H}\Big(\big(V^1- 2 \tilde\sigma^2 \psi_{\tilde \vt}(r)\big)^2+\big(V^2- 2 \tilde\sigma^2 \psi_{\tilde \vt}\big(\frac{r}{\sqrt 2}\big)\big)^2\Big)
\end{align}
for some compact set $H \subset (0,\infty)^2$.
Denoting by $G_r$ the inverse function of $\vt \mapsto \psi_{\vt}(r)/\psi_{\vt}(r/\sqrt 2)$, whose existence is proved in \cite{Hildebrandt19}, we have the representation
$$\hat \vt = G_r(V^1/V^2),\qquad \hat\sigma^2= V^1/ \psi_{\hat \vt}(r),$$  
provided that $V_1/V_2$ lies in the range of $\vt \mapsto \psi_{\vt}(r)/\psi_{\vt}(r/\sqrt 2)$. Due to consistency of $(V^1,V^2)$, the latter is true with probability tending to one. In combination with the analysis in \cite{Hildebrandt19} and Theorem \ref{thm:CLT_double_nonlinear}, one immediately obtains the following result.
\begin{thm} \label{thm:optimal_rate}
Grant Assumption (F), assume $T\max(\sqrt \Delta ,\delta)\to 0$ and let $H$ be a compact subset of $(0,\infty)^2 $ such that $(\sigma^2,\vt)$ lies in its interior. If there exist values $v\eqsim \max(1,\delta^2/\Delta)$ and $w \eqsim \max(1,\sqrt \Delta/\delta)$ such that $w\delta/\sqrt{v\Delta}$ is constant, then  we have 
\begin{align*} 
(\hat \sigma^2 -\sigma^2)^2 +(\hat \vt -\vt)^2 = \O_p \Big( \frac{\delta^3 \vee \Delta^{3/2}}{T}\Big)
\end{align*}
 for $T,N,M \to \infty$ and $\Delta \to 0$. This convergence rate is optimal up to a logarithmic factor.
\end{thm}

The inverse of the squared rate, ${T/(\delta^3 \vee \Delta^{3/2})}$, is exactly the order of magnitude of the size of balanced sub-samples of the data. Further, the rate optimality of the estimator can be deduced just like in \cite{Hildebrandt19}, where the case $f\equiv 0$ on a fixed time horizon $T$ is treated: Allowing $T\to \infty$, the quantity $1/\sqrt{r_{\delta,\Delta,T}}$ with
\begin{equation*}
r_{\delta,\Delta,T}:=
\begin{dcases}
\frac{T}{\Delta^{3/2}},&\frac{\sqrt \Delta}{\delta} \gtrsim 1,\\
\frac{T}{\delta^3} \cdot \log \frac{\delta^2}{\sqrt \Delta}  ,& \frac{\sqrt \Delta}{\delta} \to 0
\end{dcases}
\end{equation*}
is a lower bound for joint estimation of $(\sigma^2,\vt)$. The above logarithmic factor is presumably due to technical issues.
Comparison with the upper bound from Theorem \ref{thm:optimal_rate} shows that the estimator \eqref{eq:def_LSestimator.1} is (almost) rate optimal.
\begin{rem}\label{rem:plugin}
The double increment estimator $\hat\vt$ can be used for the plug-in estimator $\check f_{\check m}$. Indeed the computation of $\hat \vt$ does not require prior knowledge of the volatility parameter $\sigma^2$ and Theorem~\ref{thm:optimal_rate} reveals the (squared) convergence rate $(\vt-\hat \vt)^2=\O_{p}((\Delta^{3/2}\vee\delta^3)/T)$. In the asymptotic regime $M\Delta^2 \to \infty$ from Theorem~\ref{thm:nonparametric_plugin} the (squared) convergence rate reads as $(\vt-\hat \vt)^2=\O_{p}(\Delta^{3/2}/T)$ which is exactly the rate that we needed for the plug-in approach.
\end{rem}

\section{Proofs}\label{sec:proofs}
We first prove the results on the H\"{o}lder regularity of the linear and nonlinear component of $X$. The subsequent Sections~\ref{subsec:nonpara_proofs} and \ref{subsec:para_proofs} contain the main proofs for the nonparametric estimator and the parameter estimators, respectively. Further proofs and auxiliary results are deferred to Section \ref{subsec:aux_nonlinear}.
\subsection{Proofs for the H\"{o}lder regularity of $X$} 
	We verify the results on the H\"{o}lder regularity of the processes $X$ and $(N_t)$ claimed in Propositions \ref{prop:X_reg}  and \ref{prop:AN_reg} of Section \ref{subsec:Hoelder}, respectively. 
To that aim,
recall that for $s\geq 0$ and $p\geq 1$  the Sobolev spaces $W^{s,p}:=W^{s,p}((0,1))$ are defined as the set of all $[s]$-times weakly differentiable functions $u\colon(0,1)\to \R$ such that 
\begin{align*}
\Vert u\Vert_{W^{s,p}} := \sum_{k=0}^{[s]} \Vert u^{(k)}\Vert_{L^p} + \left( \int_0^1 \int_0^1 \frac{|u^{([s])}(\xi)-u^{([s])}(\eta)|^p}{|\xi-\eta |^{1+(s-[s])p}}\,d\xi\,d\eta \right)^{1/p} < \infty.
\end{align*}
The Sobolev space $W^{s,p}$ embeds continuously into $C^\alpha$, if $\alpha< s-1/p$. Our first step is an analysis of the H\"{o}lder regularity of the linear component $(X^0_t)$. The norm bounds in statements $(i)$ and $(ii)$ of the following lemma are also stated in \cite{Cerrai99} as well as \cite[Sect. 5.5]{DaPrato14}.  The remaining results are derived using similar techniques. We provide a complete proof for the sake of completeness.

\begin{lem} \label{lem:linear}
For any $p\in [1,\infty)$, the following hold.
\begin{enumerate}[(i)]
\item $\sup_{t\geq 0}\E(\Vert X_t^0\Vert_\infty^p)<\infty $.
\item 
For any $\gamma < 1/2$, we have
$(X^0_t) \in C(\R_+,C_0^\gamma)$ a.s.~and
$\sup_{t\geq 0}\E(\Vert X_t^0 \Vert_{C_0^{\gamma}}^p)<\infty $.
\item For any $\gamma<1/4$ and $T>0$, we have $(X^0_t)_{0\leq t\leq T} \in C^\gamma([0,T],E)$ a.s.~and
 there exists a constant $C>0$ such that $\E(\Vert X_t^0-X_s^0 \Vert_\infty ^p)\leq C|t-s|^{\gamma p}$ for all $s,t\geq 0$.
\end{enumerate}
\end{lem}

\begin{proof}
$(iii)$
The property $(X^0_t)_{0\leq t\leq T} \in C^\gamma([0,T],E)$ is a consequence of Kolmogorov's criterion and $\E(\Vert X_t^0-X_s^0 \Vert_\infty ^p)\leq C|t-s|^{\gamma p}$ for all $p\ge1$. To verify the latter statement,
assume, without loss of generality that $s,t \in (a,a+1)$ for some $a\geq 0$ and define $\mathcal U := (a,a+1)\times(0,1)$.
Using \eqref{eq:space_variation} and \eqref{eq:time_variation}, we see that  
\begin{align*}
\E(|X^0_t(x)-X^0_s(y)|^2)&\lesssim \sqrt{|t-s|}+|x-y|
 \lesssim ((t-s)^2+(x-y)^2)^{1/4}
\end{align*}
holds uniformly in $x,y\in (0,1)$ and $s,t \geq 0$. The last step follows from the equivalence of norms on $\R^2$.
Now, since $(t,x)\mapsto X^0_t(x)$ is a continuous function, the Garsia-Rodemich-Rumsey inequality, see, e.g., \cite[Theorem B.1.5]{DaPrato96}, provides the following bound on its increments: for any $\alpha>0,\beta>4$, there exists a constant $c>0$ (independent of $a$) such that
\begin{align} \label{eq:mod_cont_bound}
|X^0_s(x)-X^0_t(y)|\leq c( (x-y)^2+(t-s)^2)^{\frac{\beta-4}{2 \alpha}} \left(\int_{\mathcal U \times \mathcal U} \frac{|X^0_u(\eta)-X^0_{u'}(\eta')|^\alpha}{(|\eta-\eta'|^2+|u-{u'}|^2)^{\beta/2}} \,d\eta \,d\eta'\, du \,d{u'}\right)^\frac{1}{\alpha}
\end{align}
for all $(x,s),(y,t)\in \mathcal U$. Note that for $x=y$, the right hand side of the above inequality is independent of $x$.
Now, choose $\alpha = 2m$ for some $m\in \N$ in such a way that $\alpha=2m>p$. Then, by applying Jensen's  inequality to the concave function $\R_+\ni h\mapsto h^{p/\alpha}$, we obtain
\begin{align*}
\E(\sup_x |X^0_s(x)-X^0_t(x)|^p)&\leq c^p (t-s)^{\frac{\beta-4}{2m}p}  \left(\int_{\mathcal U \times \mathcal U} \frac{\E(|X^0_u(\eta)-X^0_{u'}(\eta')|^{2m})}{(|\eta-\eta'|^2+|u-{u'}|^2)^{\beta/2}} \,d\eta \,d\eta'\, du \,d{u'}\right)^\frac{p}{\alpha}\\
&\lesssim c^p (t-s)^{\frac{\beta-4}{2m}p}  \left(\int_{\mathcal U \times \mathcal U} \frac{(|\eta-\eta'|^2+|u-{u'}|^2)^{m/4}}{(|\eta-\eta'|^2+|u-{u'}|^2)^{\beta/2}} \,d\eta \,d\eta'\, du \,d{u'}\right)^\frac{p}{\alpha}.
\end{align*}
The above integral is finite as long as  $\beta-\frac{m}{2}<2$.
Now, the result follows since for any given $\gamma<1/4$, we can pick $m\in \N$ and $\beta<2+\frac{m}{2}$ such that $\frac{\beta-4}{2m}\leq \gamma$.

Assertion $(i)$ can be proved similarly by taking $s=t$ and $y=1$ in \eqref{eq:mod_cont_bound} to obtain a bound for $\sup_x |X^0_t(x)|=\sup_x |X^0_t(x)-X^0_s(y)|$. Note that, in order to be able to chose $y=1$, we have to modify the set $\mathcal U$ by taking, e.g., $\mathcal U=(a,a+1)\times(-\eps,1+\eps)$ for some $\eps>0$, and extend $(X^0_t)$ by defining $X_t(z):=0$ for $z \notin[0,1]$ such that $(X^0_t)$ is a continuous function on $\mathcal U$.  

$(ii)$ Clearly, $A_\vt$ is a second order differential operator whose eigenvalues  satisfy the condition $\sum_{\ell \geq 1}\lambda_\ell^{-\rho}<\infty$ for any $\rho>1/2$.
Thus, by \cite[Theorem 5.25]{DaPrato14}, $(X^0_t) \in C(\R_+,W^{2\alpha,p})$ holds for any $\alpha>0$ and  $p>1$ such that $1/p+\alpha < 1/4$. Now, by choosing $\alpha$ close to $1/4$ and $p$ sufficiently large, $(X^0_t)\in C(\R_+, C_0^\gamma)$ follows from the Sobolev embedding $W^{2\alpha,p}\subset C^\gamma$.
Now,
with the bound \eqref{eq:space_variation} for the Gaussian process $(X^0_t)$, we get for any $h\in (0,1)$ that 
\begin{align*}
\E(\Vert X_t^0 \Vert_{W^{h,q}}^q )
 &\lesssim \E(\Vert X_t^0 \Vert_{\infty}^q)+ \int_0^1 \int_0^1 \frac{\E(|X_t^0(\eta)-X_t^0(\eta')|^q)}{|\eta-\eta'|^{1+h q }}\,d\eta\,d\eta' \\
&\lesssim \E(\Vert X_t^0 \Vert_{\infty}^q)+ \int_0^1 \int_0^1 \frac{|\eta-\eta'|^{q/2}}{|\eta-\eta'|^{1+h q }}\,d\eta\,d\eta'.
\end{align*}
In view of $(i)$, this shows that $\sup_{t\geq 0}\E(\Vert X_t^0 \Vert_{W^{h,q}}^q )<\infty$, as long as $h<1/2$.  
Further, by the Sobolev embedding theorem, we have 
$\Vert X_t^0\Vert_{C_0^\gamma} \lesssim \Vert X_t^0\Vert_{W^{h,q}}$, provided that $h-\frac{1}{q}> \gamma$. Thus, choosing $h \in (\gamma,\frac{1}{2})$ and $q> \max( (h-\gamma)^{-1},p)$, we get 
\begin{align*}
\E(\Vert X_t^0\Vert_{C_0^\gamma}^p) \lesssim \E(\Vert X_t^0\Vert_{W^{h,q}}^{q\frac{p}{q}})\leq \E(\Vert X_t^0\Vert_{W^{h,q}}^{q})^{\frac{p}{q}}
\end{align*}
by Jensen's inequality. The claim now follows by taking the supremum over $t\geq 0$.
\end{proof}

Before proving Proposition \ref{prop:X_reg}, we recall some facts from semigroup theory.  For details, in particular, on analytic semigroups generated by differential operators, we refer to, e.g.,~\cite{Lunardi12}. To deal with the situation where $f(0)\neq 0$ and, hence, $f(X_t)\notin E$, we need to regard  $(S(t))_{t\geq 0}$ as a semigroup acting on the space $\tilde E = C([0,1])$. To that aim, consider the part $A_{\tilde E}$ of ${A_\vt}=\vt\Delta$ in $\tilde E$, i.e., $A_{\tilde E}x := A_\vt x$ for $x \in \mathcal D(A_{\tilde E}):= \{x \in \tilde E \cap \mathcal D(A_\vt): A_\vt x \in \tilde E  \}$. Note that $A_{\tilde E}$ generates a semigroup $(S_{\tilde E}(t))_{t\geq 0}$ on $\tilde E$ which is not strongly continuous.  Indeed, we have  $\overline{\mathcal D(A_{\tilde E})}^{\tilde E}= E$ and $\lim_{t\to 0}S_{\tilde E}(t)x= x$ in $\tilde E$ holds if and only if $x\in E$. Nevertheless, $(S_{\tilde E}(t))_{t\geq 0}$ defines a so called \emph{analytic} semigroup on $\tilde E$ which retains many properties of $C_0$-semigroups. In particular, for any $x \in \tilde E$, it holds that $\int_0^t S_{\tilde E}(r)x\,dr \in \mathcal D(A_{\tilde E})$ and we have the representation
\begin{align} \label{eq:Sx-x}
S_{\tilde E}(t)x-x = A_{\tilde E}\int_0^t S_{\tilde E}(r)x\,dr .
\end{align} 
Hence, if $r\mapsto \Vert A_{\tilde E} S_{\tilde E}(r)x\Vert_{\tilde E}$ is integrable over $[0,t]$, then  $S_{\tilde E}(t)x-x = \int_0^t A_{\tilde E} S_{\tilde E}(r)x\,dr $. 
Since the definitions of the semigroups $S$ and $S_{\tilde E}$ and their generators agree on the intersection of their domains, respectively, we will refer to both by $(S(t))_{t\geq 0}$ and $A_\vt$ from now on.  The following inequalities, which are particular cases of results derived in \cite{Sinestrari85}, are our main tool to study the regularity of $(N_t)$. 
\begin{lem}  \label{lem:intermediate_norms}
We fix an element $\lambda_0 \in (0, \lambda_1).$ For any $\alpha,\beta,\in (0,2)\setminus \{1\}$  and $n\in \N_0$, there exists a constant $C>0$ such that for all $t>0$:
\begin{enumerate}[(i)]
\item 
$\Vert A_\vt^n S(t) x\Vert_\infty \leq C\e^{-\lambda_0 t}t^{-n}\Vert x \Vert_\infty$  for all $x \in \tilde E$,
\item
$\Vert S(t)x \Vert_{C_0^\alpha}\leq C \e^{-\lambda_0 t}t^{-\alpha/2} \Vert x \Vert_\infty $ for all $x\in \tilde E$,
\item
$\Vert A_\vt S(t)x \Vert_{\infty} \leq C t^{-(1-\alpha/2)} \Vert x\Vert_{C_0^\alpha}$ for all $x\in C^\alpha_0$,
\item
$\Vert A_\vt^n S(t)x \Vert_{C_0^\beta}\leq C \e^{-\lambda_0 t} t^{-(n+\frac{\beta-\alpha}{2})}\Vert x\Vert_{C_0^{\alpha}} $ for all $x\in C^\alpha_0$ if either $n\geq 1$ or $\alpha \leq \beta$.
\end{enumerate} 
\end{lem}
For a proof of $(i),(ii)$ and $(iv)$, we refer to \cite[Proposition 2.3.1]{Lunardi12}, $(iii)$ follows from \cite[Proposition 1.11]{Sinestrari85}. Further, in order to transfer the spatial to the temporal regularity, of particular importance for our study are the so called intermediate spaces, defined by
\begin{align*}
D_{A_\vt}(\alpha, \infty):=\left \{x\in \tilde E:\, \Vert x\Vert_{D_{A_\vt}(\alpha,\infty)}:=\Vert x \Vert_{\tilde E} + \sup_{t>0}\frac{\Vert S(t)x-x\Vert_{\tilde E}}{t^\alpha} <\infty \right\},\qquad \alpha\in (0,1),
\end{align*}
which are Banach spaces with the norm $\Vert \cdot \Vert_{D_{A_\vt}(\alpha,\infty)}$. These spaces can be defined for arbitrary analytic semigroups on a Banach space, see, e.g.,~\cite{Sinestrari85}.  
For our concrete choice of $A_\vt$ and $\tilde E$, they are given by the Dirichlet-H{\"o}lder spaces
$D_{A_\vt}(\alpha, \infty)= C_0^{2\alpha}([0,1])$ , $\alpha \neq \frac{1}{2},$
where the norms are equivalent, see \cite{Lunardi85}.

\begin{proof}[Proof of Proposition \ref{prop:X_reg}]
Due to Lemma \ref{lem:linear}, it remains to prove the statements for $(N_t)$ and, if $\xi$ follows the stationary distribution, for $(\xi_t)_{t\geq 0}$ with $\xi_t:=S(t)\xi$.

$(i)$ \emph{Step 1.} We show $\Vert N_t \Vert_{C_0^\gamma}<\infty$ a.s.~for all $t\geq 0$ and, under Assumption (B), $\sup_{t\geq 0}\E (\Vert N_t \Vert_{C_0^\gamma}^p)<\infty$: From Lemma \ref{lem:intermediate_norms} $(ii)$ we have that
\begin{align*}
\Vert N_t \Vert_{C_0^\gamma} \leq \int_0^t \Vert S(t-s)f(X_s) \Vert_{C_0^\gamma} \,ds \lesssim \int_0^t \e^{-\lambda_0(t-s)}(t-s)^{-\frac{\gamma}{2}} \Vert f(X_s) \Vert_{\infty} \,ds 
\end{align*}
and, consequently,
$
\Vert N_t \Vert_{C_0^\gamma}  \lesssim \sup_{s\leq t}\Vert f(X_s) \Vert_{\infty}\int_0^t \e^{-\lambda_0r}r^{-\frac{\gamma}{2}}  \,dr 
$
is almost surely finite by our basic assumptions.
Also, using Jensen's inequality and the fact that $r\mapsto a(r):=\e^{-\lambda_0r}r^{-\frac{\gamma}{2}}$ is integrable over $\R_+$, we get
\begin{align*}
\Vert N_t \Vert_{C_0^\gamma}^p \leq \int_0^t a(t-s) \Vert f(X_s) \Vert_\infty^p \,ds\, \cdot \left( \int_0^t a(r)\,dr \right)^{p-1} 
\lesssim \int_0^t a(t-s) \Vert f(X_s) \Vert_\infty^{p} \,ds.
\end{align*}
Thus,  Fubini's theorem and the polynomial growth condition on $f$ from (F) yield
\begin{align*}
\sup_{t\geq 0}\E (\Vert N_t \Vert_{C_0^\gamma}^p) \lesssim \sup_{s\geq0} \E(\Vert f(X_s) \Vert_\infty^{p})\lesssim 1+ \sup_{s\geq 0} \E(\Vert X_s \Vert_\infty^{dp})
\end{align*}
which is finite under Assumption (B). 

\emph{Step 2:} We show $(N_t) \in C(\R_+,C_0^\gamma)$: In order to verify $\Vert N_{t+h}-N_t\Vert_{C_0^\gamma} \to 0$ for $h\to 0$ almost surely, we use the decomposition 
\begin{align*}
N_{t+h}-N_t = (S(h)-I)N_t +\int_t^{t+h}S(t+h-r)f(X_r)\,dr.
\end{align*}
To treat the first term, choose $\alpha \in (\gamma,\frac{1}{2})$. Then, using \eqref{eq:Sx-x} and property $(iv)$ of Lemma \ref{lem:intermediate_norms}, we can bound
\begin{align*}
\Vert ( S(h)-I) N_t \Vert_{C_0^\gamma} \lesssim \int_0^h \Vert {A_\vt}S(r)N_t \Vert_{C_0^\gamma}\,dr \leq \Vert N_t \Vert_{C_0^\alpha} \int_0^h \e^{-\lambda_0r}r^{-(1+\frac{\gamma-\alpha}{2})} \,dr
\end{align*}
which tends to 0 for $h\to 0$.
For the second term, it follows from bound $(ii)$ in Lemma \ref{lem:intermediate_norms} that
\begin{align*}
\Big\Vert \int_t^{t+h} S(t+h-r)f(X_r)\,dr \Big\Vert_{C_0^\gamma} &\lesssim \sup_{r\leq T}\Vert f(X_r)\Vert_\infty \int_t^{t+h} \e^{-\lambda_0r}r^{-\frac{\gamma}{2}}\,dr
\end{align*}
which also tends to 0 almost surely for $h\to 0$. 

\emph{Step 3}: Steps 1 and 2 verify claim $(i)$ in the case $\xi=0$. To treat the case where $\xi$ follows the stationary distribution, we use the fact that $X$ has the same distribution as 
$\tilde X = (X_{1+t})_{t\geq 0}$. Again, we have the decomposition 
$$\tilde X_t = S(1+t)\xi + X^0_{1+t}+N_{1+t}$$ and $(i)$ has already been proved for the second and third term. 
For the first term, the result follows from $\Vert S(1+t) \xi\Vert_{C_0^\gamma}\lesssim \Vert \xi\Vert_\infty =\Vert X_0\Vert_{\infty}$ by inequality $(ii)$ in Lemma \ref{lem:intermediate_norms}.

\emph{Step 4.} We transfer the result $(i)$ from $X$ to $f_0(X)$: First of all, $f_0(X) \in C(\R_+,C_0^\gamma)$ almost surely holds due to the result for $X$ and the assumption $f_0 \in C^1(\R)$. Further, we have
\begin{align*}
\Vert f_0(X_t) \Vert_{C_0^\gamma} = \Vert f_0(X_t) \Vert_\infty+ \sup_{\xi\neq \eta} \frac{|f(X_t(\xi))-f(X_t(\eta))|}{|\xi-\eta|^\gamma} \leq  \Vert f_0(X_t) \Vert_\infty+ \Vert f'(X_t) \Vert_\infty \,\Vert X_t \Vert_{C_0^\gamma} 
\end{align*}
and, under Assumption (B),
\begin{align*}
\E(\Vert f_0(X_t) \Vert_{C_0^\gamma}^p) &\lesssim \E(\Vert f_0(X_t)\Vert_\infty^p)+ \E(\Vert f'(X_t)\Vert_\infty^{2p})+\E(\Vert X_t\Vert_{C_0^\gamma}^{2p})\\ &\lesssim 1+  \E(\Vert X_t\Vert_\infty^{2dp})+\E(\Vert X_t\Vert_{C_0^\gamma}^{2p}) <\infty
\end{align*}
uniformly in $t\geq 0$. 

$(ii)$ \emph{Step 1.} We show the claim for $(N_t)$: Using the same decomposition for the increments of $(N_t)$ as in the proof of $(i)$, we get
\begin{align*}
\Vert N_{t}-N_s \Vert _\infty \leq \Vert (S(t-s)-I)N_s\Vert _\infty + \int_s^{t}\Vert S(t-r)f(X_r) \Vert_\infty\,dr
\end{align*}
for $s<t$.
For the first term, by definition of the intermediate spaces, it holds that
\begin{align} \label{eq:N_time_reg1}
\Vert (S(t-s)-I)N_s \Vert_\infty  \lesssim \Vert N_s \Vert_{D_{A_\vt}(\gamma,\infty)}  \,(t-s)^\gamma \lesssim \Vert N_s \Vert_{C_0^{2\gamma}}  \,(t-s)^\gamma .
\end{align}
By  Lemma \ref{lem:intermediate_norms} $(i)$ and H{\"o}lder's inequality, we have
\begin{align}
\Big\Vert \int_s^{t} S(t-r)f(X_r)\,dr \Big\Vert_\infty^p &\leq \left( \int_s^{t} \Vert S(t-r)f(X_r)\Vert_\infty\,dr \right)^p \nonumber\\
& \lesssim \left( \int_s^{t} \e^{-\lambda_0(t-r)}\Vert f(X_r) \Vert_\infty \,dr \right)^p 
\leq (t-s)^{p-1} \int_s^{t} \frac{\Vert f(X_r) \Vert_\infty^p}{\e^{\,p\lambda_0(t-r)}} \,dr. \label{eq:N_time_reg2}
\end{align}
By combining \eqref{eq:N_time_reg1} and \eqref{eq:N_time_reg2}, we obtain $(N_t)_{0\leq t\leq T}\in C^\gamma ([0,T],E)$ almost surely and, under Assumption (B),
\begin{align*}
\E\left( \Vert N_{t}-N_s \Vert_\infty^p \right)\lesssim (t-s)^{\gamma p} \E( \Vert N_s \Vert_{C_0^{2\gamma}}^p) + (t-s)^p  (1+\sup_{h\geq 0} \E(\Vert  X_h\Vert_\infty ^{pd}),
\end{align*}
from which the result for $(N_t)$ follows due to $(i)$. 

\emph{Step 2.} The case where $\xi$ follows the stationary distribution can be treated as in $(i)$ since 
$$\Vert S(1+t)\xi-S(1+s)\xi\Vert_\infty \lesssim (t-s)^\gamma\Vert S(1)\xi\Vert_{C_0^{2\gamma}}\lesssim (t-s)^\gamma\Vert X_0\Vert_\infty.$$

\emph{Step 3.} We transfer the result $(ii)$ from $X$ to $f_0(X)$: First of all, the pathwise property is again a consequence of the assumption $f_0\in C^1(\R)$.
Next, without loss of generality, assume that $d$ from (F) is given by $d=2m$ for some $m\in \N$. Then, using the formula  $a^n- b^n=(a-b)\sum_{k=0}^{n-1}a^kb^{n-1-k}$ for $a,b\in \R$ and $n\in \N$, yields
\begin{align*}
| f(X_{t}(x))- f(X_s(x)) | & \leq \int_{X_{s}(x)}^{X_{t}(x)} |f'(h)|\,dh \lesssim \int_{X_{s}(x)}^{X_{t}(x)}(1+ h^{2m})\,dh \\
&\lesssim |X_{t}(x)-X_{s}(x)|\Big(1+ \sum_{k=0}^{2m} |X_{t}(x)^k X_{s}(x)^{2m-k}|\Big)=:|X_t(x)-X_s(x)|Z_{s,t}
\end{align*}
where we have assumed $X_t(x)\geq X_s(x)$ without loss of generality.
Consequently, since $(s,t)\mapsto \Vert Z_{s,t}\Vert_\infty$ is bounded in $L^p(\P)$ for  any $p\geq 1$ under Assumption (B), we obtain
\begin{equation*}
\E (\Vert f(X_{t})-f(X_s)) \Vert_\infty^p) \lesssim \E(\Vert X_{t}- X_s\Vert_\infty^{2p})^{1/2}\E(\Vert Z_{s,t} \Vert_\infty^{2p} )^{1/2}  
\lesssim (t-s)^{\gamma p}. \qedhere
\end{equation*}
\end{proof}
We turn to the excess H\"{o}lder regularity of the nonlinear component $(N_t)$ of $X$. Since $(N_t)$ is the pathwise solution of the equation $dN_t = {A_\vt}N_t + f(S(t)\xi +X_t^0+N_t),\, N_0 =0$, the almost sure properties are a consequence of the results of \cite{Sinestrari85} on the regularity of solutions to deterministic systems. In the following, we give a direct proof for them, both for the sake of completeness and since we require its steps in order to bound the respective norms in $L^p(\P)$.

\begin{proof}[Proof of Proposition \ref{prop:AN_reg}]
$(i)$ Due to Proposition \ref{prop:X_reg}, we have $f_0(X_t)\in D_{A_\vt}(\frac{\gamma}{2},\infty)= C_0^{\gamma}$ for any $\gamma<1/2$. Further, for any $\tilde \gamma \in( \gamma, \frac{1}{2})$, Lemma \ref{lem:intermediate_norms} $(iv)$ yields that
\begin{align*}
\Vert {A_\vt}N_t^0 \Vert_{C_0^{\gamma}} \lesssim \int_0^t \Vert {A_\vt}S(t-s)f_0(X_s)\Vert_{C_0^{\gamma}}\,ds
\leq \int_0^t h(t-s) \Vert f_0(X_s) \Vert_{C_0^{\tilde\gamma}}\, ds
\end{align*}
with $h(r):= \e^{-\lambda_0 r} r^{-1+{(\tilde \gamma-\gamma)/2}}$.
Since $h$ is integrable over $\R_+$ and $A_\vt = \vt \frac{\partial^2}{\partial x^2}$, the almost sure properties $N_t^0 \in C_0^{2+\gamma}$ and $\sup_{t\leq T} \Vert A_\vt N_t^0\Vert_{C_0^{\gamma}}<\infty$ immediately follow from $f_0(X)\in C(\R_+,C^{\tilde \gamma}_0)$, cf.~Proposition \ref{prop:X_reg}. Further, Jensen's inequality gives 
\begin{align*}
\Vert {A_\vt}N_t^0 \Vert_{C_0^{\gamma}}^p \lesssim \int_0^t h (t-s)\Vert f_0(X_s)\Vert_{C_0^{\tilde\gamma}}^p\,ds\, \left(\int_0^t h(r)\,dr \right)^{p-1}.
\end{align*}
Consequently, under Assumption (B), $\sup_{t\geq 0}\E(\Vert {A_\vt}N_t^0 \Vert_{C_0^{\gamma}}^p)\lesssim \sup_{t\geq 0} \E(\Vert f_0(X_t) \Vert_{C_0^{\tilde  \gamma}}^p)$ is finite, by Proposition \ref{prop:X_reg}.

$(ii)$ In order to prove $\frac{d}{dt}N_t^0 = {A_\vt}N^0_{t}+f_0(X_t)$ in $E$, note that the usual decomposition for the increments of $(N^0_t)$ and formula \eqref{eq:Sx-x} yield the representation
\begin{align*}
\Delta^{-1} (N_{t+\Delta}^0-N_t^0) - {A_\vt}N^0_{t}-f_0(X_t) =& \frac{1}{\Delta} \int_0^\Delta (S(r)-I){A_\vt}N_t^0\,dr \\
& +\frac{1}{\Delta} \int _t^{t+\Delta}\Big(S(t+\Delta-r)f_0(X_r)-f_0(X_t)\Big)\,dr.
\end{align*}  
We have $\Vert (S(r)-I){A_\vt}N_t^0\Vert_\infty \lesssim r^{\gamma}\Vert {A_\vt}N_t^0 \Vert_{C_0^{2\gamma}}$ and 
\begin{align*}
\Vert S(h)f_0(X_r)-f_0(X_t)\Vert_\infty \leq& \Vert S(h)(f_0(X_r)-f_0(X_t))\Vert_\infty +\Vert (S(h)-I)f_0(X_t)\Vert_\infty\\\lesssim & \Vert f_0(X_r)-f_0(X_t)\Vert_\infty + h^\gamma \Vert f_0(X_t)\Vert_{C_0^{2\gamma}}.
\end{align*}  
Thus, $(i)$ and Proposition \ref{prop:X_reg} yield $\Vert \Delta^{-1} (N_{t+\Delta}^0-N_t^0) - {A_\vt}N^0_{t}-f_0(X_t)\Vert_\infty \lesssim \Delta^\gamma \to 0$ uniformly on bounded time intervals, almost surely. 
The properties claimed for $\frac{d}{dt}N_t^0$ now follow from the properties of $f_0(X_t)$ provided by Proposition \ref{prop:X_reg} and 
\begin{align*}
\Vert {A_\vt}N_{t+\Delta}^0 -{A_\vt}N_{t}^0 \Vert_\infty &\leq \Vert (S(\Delta)-I){A_\vt}N_t^0 \Vert_\infty+ \int_t^{t+\Delta} \Vert {A_\vt} S(t+\Delta-r)f_0(X_r) \Vert_\infty\,dr\\
& \lesssim \Delta^\gamma \Vert {A_\vt}N_t^0\Vert_{C_0^{2\gamma}} +\int_t^{t+\Delta} (t+\Delta-r)^{-1+\gamma}\Vert f_0(X_r) \Vert_{C_0^{2\gamma}} \,dr
\end{align*}
where the bound on the integrand is taken from result $(iii)$ in Lemma \ref{lem:intermediate_norms}.

It remains to analyze the regularity of the process $(M_t)$. First of all, by \eqref{eq:Sx-x}, we have ${A_\vt}M_t = S(t)m-m$ and $m\in C^{\gamma}([b,1-b])$ for $\gamma<1/2$ is trivially fulfilled. Further, setting $m_t :=S(t)m$, we have $m_t (x) = \frac{2\sqrt 2}{\pi} \sum_{\ell\geq 0}  \frac{\e^{-\lambda_{2\ell+1} t}}{2\ell+1}e_{2\ell+1}(x) $. The mean value theorem yields
\begin{align*}
m_t(x)-m_t(y) = (x-y) {8} \sum_{\ell\geq 0}  {\e^{-\lambda_{2\ell+1} t}}\cos(\pi (2\ell+1)z)
\end{align*}
for some $z$ between $x$ and $y$. Thanks to the bound on trigonometric series from \cite[Lemma A.7]{Hildebrandt19}, the sum $\sum_{\ell\geq 0}  {\e^{-\lambda_{2\ell+1} t}}\cos(\pi (2\ell+1) z)$ is uniformly bounded in $t>0$ and $z\in [b,1-b]$ and we can conclude $\sup_{t\geq 0} \Vert {A_\vt}M_t \Vert_{C^{\gamma}([b,1-b])}<\infty$ for $\gamma<1/2$. The same argument shows that 
$$ \Vert S(t+\Delta)m-S(t)m\Vert_{C([b,1-b])} \lesssim \sup_{\ell \geq 0} \frac{1-\e^{-\lambda_{2\ell+1} \Delta}}{2\ell+1} \lesssim \sqrt \Delta.$$
Hence, 
$$\Vert \Delta^{-1}(M_{t+\Delta}-M_t)-S(t)m\Vert_{C([b,1-b])} \leq \frac{1}{\Delta} \int_{t}^{t+\Delta}\Vert(S(r)m-S(t)m)\Vert_{C([b,1-b])}\,dr\lesssim \sqrt \Delta$$ 
and, in particular, $\frac{d}{dt}M_t = S(t)m$ in $C([b,1-b])$ as well as $\Vert \frac{d}{dt}(M_{t+\Delta}-M_t)\Vert_{C([b,1-b])} \lesssim \sqrt \Delta \lesssim  \Delta^\gamma $ for $\gamma<1/2$. 
\end{proof}

\subsection{Proofs for the nonparametric estimator of $f$} \label{subsec:nonpara_proofs}
\begin{proof}[Proof of Proposition \ref{prop:SpaceDiscrete}]
By applying the Cauchy-Schwarz inequality, Young's inequality and Lemma \ref{lem:L2norm_approx} to \eqref{eq:oracle_prelim}, we can bound
\begin{align*}
 \Vert \hat f_m - f_{\mathcal A} \Vert_{N,M}^2  \leq& \Vert  f_m^* - f_{\mathcal A} \Vert_{N,M}^2 
 + \frac{2}{N} \sum_{i=0}^{N-1}  \big \Vert\hat S(0)\big( \hat f_m(X_{t_i}) -f_m^*(X_{t_i})\big)\big\Vert_{L^2} \big\Vert R_i \big\Vert_{L^2}\\
& + \Vert \hat f_m -f_m^* \Vert_{N,M} \sup_{g\in \mathcal V_m,\, \Vert g\Vert_{N,M}=1} \frac{2}{N} \sum_{i=0}^{N-1}  \big \langle \hat S(0) g(X_{t_i}),\eps_i \big\rangle_{L^2}\\
\leq & \Big(1+\frac{4}{\eta}\Big)\Vert  f_m^* - f_{\mathcal A} \Vert_{N,M}^2 +\frac{4}{\eta} \Vert \hat f_m - f_{\mathcal A} \Vert_{N,M}^2  +  \frac{\eta}{N} \sum_{i=0}^{N-1}  \big\Vert R_i \big\Vert_{L^2}^2\\
&+ \eta\left( \sup_{g\in \mathcal V_m,\, \Vert g\Vert_{N,M}=1} \frac{1}{N} \sum_{i=0}^{N-1} \big \langle  \hat S (0)g(X_{t_i}),\eps_i \big\rangle_{L^2} \right)^2
\end{align*}
for any $\eta >0$.
Taking $\eta = 8$ and rearranging gives 
\begin{align} 
 \Vert \hat f_m - f_{\mathcal A} \Vert_{N,M}^2\leq &  3\Vert  f_m^* - f_{\mathcal A} \Vert_{N,M}^2   +  \frac{16}{N} \sum_{i=0}^{N-1}  \big\Vert R_i \big\Vert_{L^2}^2 \nonumber\\
 & +\Big( \sup_{g\in \mathcal V_m,\, \Vert g\Vert_{N,M}=1} \frac{4}{N} \sum_{i=0}^{N-1} \big \langle  \hat S(0)g(X_{t_i}),\eps_i \big\rangle_{L^2} \Big)^2.  \label{eq:bound_OmegaNm}
\end{align}
The claim of the proposition follows by bounding the expectation on $\Omega_{N,M,m}$ of the three terms on the right hand side of the above inequality.

For the first term, we have
$\E(\Vert  f_m^* - f_{\mathcal A} \Vert_{N,M}^2 \1_{\Omega_{N,M,m}} ) \leq \E(\Vert  f_m^* - f_{\mathcal A} \Vert_{N,M}^2) = \Vert  f_m^* - f_{\mathcal A} \Vert_{\pi,M}^2 .$
To treat the second term, we show that
\begin{equation} \label{eq:Rtilde_bound}
\E(\Vert  R_i \Vert_{L^2}^2)\lesssim \frac{1}{M \Delta^2}  +\Delta^\gamma
\end{equation}
holds for any $\gamma<1/2$: First of all, with $f_0:= f-f(0)$ and $\1 := \1_{[0,1]}$,
 we have 
\begin{align*}
&\Big\Vert S(h )f(X_s)-f(X_{t})\Big\Vert_{L^2}^2 \\
&\qquad\lesssim \Big\Vert S(h )f_0(X_s)-f_0(X_s)\Big\Vert_{\infty}^2+f(0)^2\Big\Vert S(h )\1-\1\Big\Vert_{L^2}^2+\Big\Vert f(X_s)-f(X_{t})\Big\Vert_{\infty}^2 \\
&\qquad \lesssim  h^\gamma  \Vert f_0(X_s)\Vert^2_{D_A(\gamma/2,\infty)} + f(0)^2 \sum_{\ell \geq 1}(1-\e^{-\lambda_\ell \Delta})^2 \langle \1,e_\ell \rangle^2 +\Big\Vert f(X_s)-f(X_{t})\Big\Vert_{\infty}^2.
\end{align*}
Using Jensen's inequality, $D_{A_\vt}(\frac{\gamma}{2},\infty)=C^\gamma _0$ and $\langle \1,e_\ell\rangle^2\lesssim \ell^{-2}$, we get for any $\gamma <1/2$ that
\begin{align*}
\E \left( \big\Vert R_i^{\mathrm{cont}} \big\Vert_{L^2}^2\right)
\leq& \frac{1}{\Delta} \E \left( \int_{t_i}^{t_{i+1}} \Big\Vert S(t_{i+1} - s )f(X_s)-f(X_{t_i})\Big\Vert_{L^2}^2\,ds \right) \nonumber\\
\lesssim & \frac{1}{\Delta}  \int_{t_i}^{t_{i+1}} \Delta^\gamma \E \left(\Big\Vert f_0(X_s)\Big\Vert_{C_0^\gamma}^2\right)\,ds +f(0)^2 \sqrt \Delta \nonumber\\
&+ \frac{1}{\Delta}  \int_{t_i}^{t_{i+1}} \E \left(\Big\Vert f(X_s)-f(X_{t_i})\Big\Vert_{\infty}^2\right)\,ds 
\lesssim \Delta^\gamma 
\end{align*}
in view of Proposition  \ref{prop:X_reg}.
Further, by Lemma \ref{lem:L2norm_approx2}, we have 
$$\E(\Vert f(X_{t})-\hat S(0)f(X_t)\Vert_{L^2}^2)\lesssim \E(\Vert f(X_t)\Vert_{C^{2\alpha}}^2  +\Vert f(X_t)\Vert_\infty^2+\Vert f(X_t)\Vert_{D_\alpha}^2 )\delta^{\frac{8\alpha^2}{4\alpha+1}}$$ with the space $D_\alpha$ defined in \eqref{eq:Deps_def}. The expectation on the right hand side is finite as long as $\alpha<1/4$, due to Lemma \ref{lem:Sobolev_stationary} and Proposition \ref{prop:X_reg}. Thus, by picking $\alpha $ sufficiently close to $1/4$, we get $$\E(\Vert f(X_{t})-\hat S(0)f(X_t)\Vert_{L^2}^2) \lesssim \delta^{\gamma/2}= {M^{-\gamma/2}}=o(\Delta^\gamma)$$ under the condition $M\Delta^2 \to \infty$.  To bound $\Delta^{-2} \E( \Vert S(h) X_{t}-\hat S(h) X_{t} \Vert_{L^2}^2)$ for $h\in \{0,\Delta\}$, we use the usual decomposition $X_t = S(t)X_0 + X_t^0 +N_t$ where  we can fix a convenient value for $t>0$, due to stationarity. Since the decomposition is trivial for $t=0$, we pick $t:=t_1 = \Delta$. The linear component $X^0_t$ can easily be treated due to independence in view of Lemma \ref{lem:L2norm_approx}: 
\begin{align*}
\E(\Vert S(h)X_t^0-\hat S(h)X_t^0 \Vert_{L^2}^2 ) &= \E\Big( \sum_{k=1}^{M-1}\e^{-2\lambda_k h}\Big( \sum_{\ell \in \mathcal I_k^+\setminus \{k\}}u_\ell(t) -\sum_{\ell \in \mathcal I_k^-}u_\ell(t)  \Big)^2\Big) + \E \Big(\sum_{\ell \geq M}\e^{-2\lambda_\ell h }u_\ell^2(t)\Big)\\
&\leq 2 \sum_{\ell \geq M} \E( u_\ell^2(t)) \lesssim \frac{1}{M}
\end{align*}
and dividing by the squared renormalization $\Delta^2$ yields the claimed $\O(1/(M\Delta^2))$-bound. The other two terms in the decomposition are of lower order:
For $S(t)X_0$, we have 
\begin{align*}
\Vert S(h) S(t) X_0 -\hat S(h) S(t)X_0 \Vert_{L^2}^2=& \sum_{k=1}^{M-1} \e^{-2\lambda_k h} \Big(\langle S(t)X_0,e_k \rangle_{L^2}-\langle S(t)X_0,e_k \rangle_M \Big)^2 \\
&+ \sum_{k\geq M} \e^{-2\lambda_k (h+t)} \langle X_0,e_k \rangle_{L^2}^2.
\end{align*}
For the first sum, Lemma~\ref{lem:L2norm_approx} and the Cauchy-Schwarz inequality yields
\begin{align*}
\Big(\langle S(t)X_0,e_k \rangle_{L^2}-\langle S(t)X_0,e_k \rangle_M\Big)^2 =& \Big(\sum_{l\in \mathcal I_k^+\setminus \{k\}} \e^{-\lambda_l t} \langle X_0,e_l\rangle_{L^2} - \sum_{l\in \mathcal I_k^-} \e^{-\lambda_l t} \langle X_0,e_l\rangle_{L^2} \Big)^2\\
\leq&  \Vert X_0 \Vert_{L^2}^2 \sum_{l\in (\mathcal I_k^+\cup \mathcal I_k^-)\setminus \{k\}} \e^{-2\lambda_l t} 
\end{align*}
and, thus, with $t=\Delta$,
\begin{align*}
&\sum_{k=1}^{M-1} \e^{-2\lambda_k h} \Big(\langle S(t_1)X_0,e_k \rangle_{L^2}-\langle S(t_1)X_0,e_k \rangle_M \Big)^2 \\
&\qquad\leq \Vert X_0 \Vert_{L^2}^2 \sum_{l\geq M} \e^{-2\lambda_l \Delta} \leq \Vert X_0 \Vert_{L^2}^2 \frac{1}{\sqrt \Delta} \int^\infty_{M\sqrt \Delta} \e^{{-2\pi^2 \vt x^2}}\,dx \lesssim \Vert X_0 \Vert_{L^2}^2 \frac{1}{M^2 \Delta^{3/2}}. 
\end{align*}
The same bound holds for the second sum since
\begin{align*}
\sum_{k\geq M} \e^{-2\lambda_k (h+\Delta)} \langle X_0,e_k \rangle_{L^2}^2 \leq \Vert X_0 \Vert_{L^2}^2 \e^{-2\lambda_M \Delta } \lesssim \Vert X_0 \Vert_{L^2}^2 \frac{1}{M^2 \Delta} \lesssim \Vert X_0 \Vert_{L^2}^2 \frac{1}{M^2 \Delta^{3/2}}. 
\end{align*}
Therefore, assuming $M \Delta^2 \to \infty,$ we get
\begin{align*}
\Delta^{-2} \E(\Vert S(h) S(t_1) X_0 -\hat S(h) S(t_1)X_0 \Vert_{L^2}^2 ) \lesssim \E(\Vert X_0 \Vert_{L^2}^2) \frac{1}{M^2 \Delta^{7/2}} = o\Big( \frac{1}{M\Delta^2} \Big).
\end{align*}
For the nonlinear part, set $B_k := \sum_{\ell \in \mathcal I_k^+\setminus \{k\}}n_\ell(t) -\sum_{\ell \in \mathcal I_k^-}n_\ell(t)  $ with $n_\ell(t) := \langle N_t,e_\ell\rangle_{L^2} $. Then, by the Cauchy-Schwarz inequality,
$$B_k^2 \leq \Big( \sum_{\ell \in (\mathcal I_k^+ \cup \mathcal I_k^- )\setminus \{k\}} \lambda_\ell^{2\alpha}n_\ell^2 \Big) \Big( \sum_{\ell \in (\mathcal I_k^+ \cup \mathcal I_k^- )\setminus \{k\}} \lambda_\ell^{-2\alpha}\Big) 
\leq \Vert N_t \Vert_{D_\alpha}^2 \Big( \sum_{\ell \in (\mathcal I_k^+ \cup \mathcal I_k^- )\setminus \{k\}} \lambda_\ell^{-2\alpha}\Big).$$
Since, furthermore, $\sum_{\ell \geq M} n_\ell^2(t) \leq \lambda_M^{-2\alpha} \Vert N_t\Vert_{D_\alpha}^2$, we have
\begin{align*}
\E(\Vert S(h)N_t-\hat S(h) N_t \Vert_{L^2}^2 )&\leq \E\Big( \sum_{k=1}^{M-1}B_k^2 \Big)+ \E \Big(\sum_{\ell \geq M} n_\ell^2(t)\Big) \lesssim \E(\Vert N_t\Vert_{D_\alpha}^2) \Big( \sum_{k\geq M} \lambda_k^{-2\alpha} +\lambda_M^{-2\alpha}\Big)\\
&\lesssim \frac{1}{M^{4\alpha-1}}  \E(\Vert N_t\Vert_{D_\alpha}^2).
\end{align*}
Now, by Lemma~\ref{lem:auxAdapt}, we have $\E(\Vert N_t\Vert_{D_\alpha}^2)<\infty$  for $\alpha = 1/2$ and, thus, $\E(\Vert S(h)N_t-\hat S(h) N_t \Vert_{L^2}^2 ) \lesssim \frac{1}{M}$, which finishes the proof of \eqref{eq:Rtilde_bound}.

To treat the third term on the right hand side of \eqref{eq:bound_OmegaNm}, consider an orthonormal system $\{\varphi_k,\,k \in \Lambda_m\}$ of $\mathcal V_m$ with the property $\Vert \sum_{k\in \Lambda_m}\varphi_k^2\Vert_\infty \leq C D_m$ which exists due to Assumption (N). Since on $\Omega_{N,M,m}$, $\Vert g \Vert_{N,M} = 1$ implies $\Vert g \Vert_{L^2({\mathcal A})}^2\leq 1/\underline{c},$ we obtain
\begin{align*}
 &\sup_{g\in \mathcal V_m,\, \Vert g\Vert_{N,M}=1} \left(\frac{1}{N} \sum_{i=0}^{N-1}  \big \langle  \hat S(0)g(X_{t_i}), \eps_i \big\rangle_{L^2} \right)^2\1_{\Omega_{N,M,m}} \\ 
& \qquad \qquad\le \frac{1}{ \underline{c}}  \sup_{\alpha \in \R^{\Lambda_m},\, \Vert \alpha \Vert\leq 1} \left(  \sum_{k \in \Lambda_m}\alpha_k \frac{1}{N}\sum_{i=0}^{N-1}  \big \langle  \hat S(0)\varphi_k(X_{t_i}),\eps_i \big\rangle_{L^2}\right)^2\\
& \qquad \qquad\lesssim \sum_{k \in \Lambda_m} \left( \frac{1}{N} \sum_{i=0}^{N-1}  \big \langle  \hat S(0)\varphi_k (X_{t_i}),\eps_i \big\rangle_{L^2}\right)^2.
\end{align*}
To handle the expectation of the above bound, note that $\eps_i= \frac{\sigma}{\Delta}\int_{t_i}^{t_{i+1}}S(t_{i+1}-s)\,dW_s$ is independent of $\mathcal F_{t_i}$ and $\widehat{\varphi_k(X_{t_i})}:= \hat S(0)\varphi_k(X_{t_i})$ is $\mathcal F_{t_i}$-measurable, implying
\begin{align*}
\E \left(\big\langle \widehat{\varphi_k(X_{t_i})}, \eps_i \big\rangle_{L^2} | \mathcal F_{t_i}\right)
&=  \int_0^1 \big(\widehat{\varphi_k(X_{t_i})}\big)(x)\E \left( \eps_i(x)  | \mathcal F_{t_i}\right)dx =0.
\end{align*}
Hence, for $j< i$, we have
\begin{align*}
&\E \left( \big\langle\widehat{\varphi_k(X_{t_i})},\eps_i \big\rangle_{L^2} \big\langle\widehat{\varphi_k(X_{t_j})},\eps_j \big\rangle_{L^2} \right)=\E \left( \E \left( \big\langle \widehat{\varphi_k(X_{t_i})},\eps_i \big\rangle_{L^2} \Big|\mathcal F_{t_i} \right)\Big\langle \widehat{\varphi_k(X_{t_j})},\eps_j \Big\rangle_{L^2} \right)=0
\end{align*}
and, consequently,
\begin{align*}
&\E\Big(  \sup_{g\in \mathcal V_m,\, \Vert g\Vert_{N,M}=1} \Big(\frac{1}{N} \sum_{i=0}^{N-1}  \big \langle  \widehat{g(X_{t_i})},\eps_i \big\rangle_{L^2} \Big)^2\1_{\Omega_{N,M,m}}\Big)\leq \sum_{k \in \Lambda_m} \frac{1}{N^2} \sum_{i=0}^{N-1} \E\left(  \big \langle \widehat{ \varphi_k(X_{t_i})},\eps_i \big\rangle_{L^2} ^2\right).
\end{align*}
Further, Parseval's relation and independence of $\F_{t_i}$ and $\langle \eps_i,e_k\rangle_{L^2}$ yield
\begin{align*}
\E\left(    \big \langle  \widehat{ \varphi_k(X_{t_i})},\eps_i \big\rangle_{L^2} ^2 \right)
&= \E\Big(  \Big( \sum_{l\geq 1} \big\langle \widehat{\varphi_k(X_{t_i})},e_l \big\rangle_{L^2} \big\langle \eps_i  ,e_l \big\rangle_{L^2} \Big)^2 \Big)\\
&= \frac{\sigma^2}{\Delta^2}    \sum_{l\geq 1} \E(\langle \widehat{\varphi_k(X_{t_i})},e_l \rangle_{L^2}^2)  \E\left( \Big(\int_{t_i}^{t_{i+1}} \e^{-\lambda_l(t_{i+1}-s)}\,d\beta_k(s)\Big)^2\right)    \\
&= \frac{\sigma^2}{\Delta^2} \sum_{l\geq 1} \frac{1-\e^{-2\lambda_l\Delta}}{2\lambda_l} \E(\langle \widehat{\varphi_k(X_{t_i})},e_l \rangle_{L^2}^2)\\
& \leq  \frac{\sigma^2}{\Delta}\E \left( \Vert \widehat{\varphi _k (X_{t_i})}\Vert^2_{L^2} \right)
= \frac{\sigma^2}{M\Delta} \E \Big( \sum_{l=1}^{M-1} \varphi_k^2(X_{t_i}(y_l)) \Big) .
\end{align*}
Above, we have used independence of the (one-dimensional) stochastic integrals from $\mathcal F_{t_i}$ and pairwise independence of $\{\beta_k,\,k\geq 1\}$ in the third step as well as Lemma \ref{lem:L2norm_approx} in the last step.
In view of Assumption (N), we have shown
\begin{align*}
\E\Big(  \sup_{g\in \mathcal V_m,\, \Vert g\Vert_{N,M}=1} \Big(\frac{1}{N} \sum_{i=0}^{N-1}  \big \langle  \hat S(0)g(X_{t_i}),\eps_i \big\rangle_{L^2} \Big)^2 \1_{\Omega_{N,M,m}}\Big) 
&\leq  \frac{1}{T} \Big\Vert\sum _{k \in \Lambda_m} \varphi_k^2\Big\Vert_\infty
\lesssim \frac{D_m}{T}.\qedhere
\end{align*}
\end{proof}


\begin{proof}[Proof of Lemma \ref{lem:normequivalence_discrete}]
 With the constants $0<c<C<\infty$ from \eqref{eq:(E)implication_cont} in Assumption (E) we have
\begin{align*}
\P\left(\Xi_{N,M,m}^c\right) &= \P\left( \sup_{g\in \mathcal V_m\setminus \{0\}} \Big| \frac{\Vert g\Vert^2_{N,M}-\Vert g\Vert^2_{\pi,M}}{\Vert g\Vert^2_{\pi,M}}\Big| \geq \frac{1}{2}\right) 
\leq  \P\left( \sup_{g\in \mathcal V_m\setminus \{0\}} \Big| \frac{\Vert g\Vert^2_{N,M}-\Vert g\Vert^2_{\pi,M}}{c\Vert g\Vert^2_{L^2({\mathcal A})}}\Big| \geq \frac{1}{2}\right)\\
&= \P\left( \sup_{g\in \mathcal V_m,\Vert g\Vert_{L^2({\mathcal A})}=1} | v_{N,M}(g^2)| \geq  \frac{c}{2}\right),
\end{align*} 
where
$$v_{N,M}(g):= \frac{1}{NM}\sum_{i=0}^{N-1} \sum_{k=1}^{M-1} \left( g(X_{i\Delta}(y_k)) - \E \big( g(X_{i\Delta}(y_k))\big)\right).$$
Each $g \in \mathcal V_m$ with $\Vert g\Vert_{L^2({\mathcal A})}=1$ has a representation  $g= \sum_{l \in \Lambda_m}\alpha_l \varphi_l$ with $\sum_{l\in \Lambda_m} \alpha_l^2 =1$ and  
$v_{N,M}(g^2)= \sum_{l,l' \in \Lambda_m} \alpha_l \alpha_{l'} v_{N,M} (\varphi_l \varphi_{l'}).$
On the set $\{|v_{N,M}(\varphi_l \varphi_{l'})|\leq 5 V^m_{l l'}(C \kappa)^{1/2}+3 B^m_{l l'}\kappa,\, \forall l,l' \in \Lambda_m\}$ with $\kappa:= \frac{ c^2}{121 C L_m}$, we have $\sup_{g\in \mathcal V_m,\Vert g\Vert_{L^2({\mathcal A})}=1} | v_{N,M}(g^2)|  \leq \frac{c}{2}$ because 
\begin{align*}
\sum_{l,l'\in \Lambda_m} |\alpha_l \alpha_{l'}| |v_{N,M} (\varphi_l \varphi_{l'})| \leq 5 (C \kappa)^{1/2} \rho(V^m) +3\kappa \rho(B^m)\leq \frac{5 c}{11}  +\frac{3c^2}{121C} \leq \frac{c}{2}
\end{align*}
where the last bound is due to $c\leq C$. Consequently,
\begin{align*}
\P\left(\Xi_{N,M,m}^c\right)&\leq \P\Big(\exists l,l' \in \Lambda_m:|v_{N,M}(\varphi_l \varphi_{l'})|\geq 5 V^m_{l l'}(C \kappa)^{1/2}+3 B^m_{l l'}\kappa  \Big) \\
&\leq \sum_{l,l'\in \Lambda_m} \P^*\Big( |v_{N,M}(\varphi_l \varphi_{l'})|\geq 5 V^m_{l l'}(C \kappa)^{1/2}+3 B^m_{l l'}\kappa  \Big).
\end{align*}

To bound the probabilities in the previous line, we apply the Bennett inequality under strong mixing from \cite[Theorem 6.1]{Rio2017}. The latter implies for $\alpha$-mixing real-valued random variables $Z_1,\ldots,Z_n$ with $|Z_i| \leq B$ and $\E(Z_i^2)\leq \nu^2$ for some constants $B,\nu>0$ and with strong mixing coefficients $(\alpha_n)_{n\ge0}$ that for all $1<q\le N$ and all $\lambda\ge  Bq$ 
\begin{equation*}
  \P\Big(\Big|\sum_{i=1}^N(Z_i-\E(Z_i))\Big|\ge \frac72\lambda\Big)\le 4\exp\Big(-\frac{\lambda^2}{2(Nqv^2+\lambda q B/3)}\Big)+\frac{4BN}{\lambda}\alpha_{q+1}.
\end{equation*}
With $\lambda=N(B\kappa/3+\sqrt{(B\kappa/3)^2+2\nu^2 \kappa})$ for $\kappa\ge 3q/N$ we conclude 
\begin{equation}\label{eq:Rio}
  \P\Bigg(\frac{1}{N}\Big|\sum_{i=1}^N(Z_i-\E(Z_i))\Big|\ge \frac72\Big(\sqrt{2\nu^2 \kappa}+\frac{2B}{3}\kappa\Big)\Bigg)\le 4\e^{-\kappa N/q}+\frac{12}{\kappa}\alpha_{q+1}.
\end{equation}
By construction, $Z_{i}^{(l,l')} := \frac{1}{M}\sum_{k=1}^{M-1}(\varphi_{l}\varphi_{l'})(X_{i\Delta}(y_k))$ satisfy $|Z_{i}^{(l,l')}|\le B_{l,l'}^m$ and with Jensen's inequality:
\begin{align*}
\E\Big( (Z_{i}^{(l,l')})^2\Big) 
= \Vert \varphi_l \varphi_{l'} \Vert_{\pi,M}^2 \leq C \Vert \varphi_l \varphi_{l'} \Vert_{L^2({\mathcal A})}^2=C (V^m_{l ,l'})^2.
\end{align*}
Since $Z_{i}^{(l,l')}$ are exponentially $\beta$-mixing under Assumption~(M) and the $\alpha$-mixing coeffients $\alpha_q$ are bounded by $\beta_X(q\Delta)$, \eqref{eq:Rio} yields for some $K,K'>0$ with $\kappa \eqsim 1/L_m \gtrsim q/N$
\[
\P\left(\Xi_{N,M,m}^c\right)\le 4D_m^2\big(\e^{-\kappa N/q}+\kappa^{-1}\beta_X(q\Delta) \big)
\le K D_m^2\big(\e^{-K' N/(qL_m)}+L_m\beta_X(q\Delta) \big)
\]
if $q=o(N/L_m)$. Note that $\beta_X(q\Delta)\le L\e^{-\tau q\Delta}$ under (M).

Assuming $\frac{N\Delta}{\log^2 N} \to \infty$ and $L_m = o(\frac{N\Delta}{\log^2 N})$, we can choose $q=q_N$ such that $q_N/ (\frac{p \log N}{\Delta}) \to 1$ for any fixed $p>0$ and $L_mq_N/N=o(1/\log N)$. Due to $D_m\le N$, we get for sufficiently large $N$:
\[
  \P\left(\Xi_{N,M,m}^c\right)\lesssim N^2\big(\e^{-K' N/(q_NL_m)}+L_m\e^{-\tau q_N\Delta} \big)
  \lesssim N^{-p}+\frac{N^3\Delta}{\log^2 N}N^{-\tau p/2}.
\]
Since $\Delta\to 0$ and $p$ was arbitrary, we obtain the upper bound $N^{-\gamma}$ for any $\gamma>0$. 
\end{proof}

Based on the previous results, we are now ready to verify the conclusion of the main theorem.
\begin{proof}[Proof of Theorem \ref{thm:nonpara_discrete}]
Consider $\Omega_{N,M,m}=\Omega_{N,M,m,\frac{c}{2}}$ as defined in Proposition \ref{prop:SpaceDiscrete} with $c>0$ from the implication \eqref{eq:(E)implication_cont} of Assumption (E). Then, on $\Xi_{N,M,m}$, we have $\Vert g\Vert_{N,M}^2\geq \frac{1}{2}\Vert g\Vert_{\pi,M}^2\geq \frac{c}{2}\Vert g\Vert_{L^2({\mathcal A})}^2 $ for all $g\in \mathcal V_m$, implying $\Xi_{N,M,m}\subset \Omega_{N,M,m}$. 
Thus,
\begin{align*}
\E\big(\Vert \hat f_m - f_{\mathcal A}\Vert_{N,M}^2\big) &= \E\big(\Vert \hat f_m - f_{\mathcal A}\Vert_{N,M}^2\1_{\Omega_{N,M,m}}\big) + \E\big(\Vert \hat f_m - f_{\mathcal A}\Vert_{N,M}^2 \1_{\Omega_{N,M,m}^c}\big)\\
&\lesssim \Vert f_{\mathcal A}- f_m^* \Vert_{L^2({\mathcal A})}^2 + \frac{D_m}{T} + \Delta^\gamma  +  \frac{1}{M\Delta^2} +\E\big(\Vert \hat f_m - f_{\mathcal A}\Vert_{N,M}^2 \1_{\Xi_{N,M,m}^c}\big)
\end{align*}
by Proposition \ref{prop:SpaceDiscrete} and Assumption (E).
In the following, we conclude the theorem by showing that  
$$\E(\Vert \hat f_m -f_{\mathcal A} \Vert_{N,M}^2 \1_{\Xi_{N,M,m}^c})=o(\Delta^\gamma).$$ We consider the Hilbert space $H^N:=(L^2(0,1))^N$ equipped with the inner product $\langle u, v\rangle_{H^N}:= \frac{1}{N}\sum_{i=1}^N \langle u_i,v_i\rangle_{L^2}$ for $u,v\in H^N$.
Note that $\Vert g\Vert_{N,M}^2 = \Vert \bar g \Vert^2_{H^N}$ with $\bar g:= (\hat S(0)g(X_0),\ldots, \hat S(0)g(X_{t_{N-1}}))$. Clearly, the vector  \mbox{$(\hat S(0)\hat f_m(X_0),\ldots,  \hat S(0)\hat f_m(X_{(N-1)\Delta}))$} is the orthogonal projection in $H^N$ of $\bar Y := (Y_0,\ldots ,Y_{N-1})$ onto the subspace $\{(\hat S(0)g(X_0),\ldots,\hat S(0) g(X_{t_N-1})),\,g\in \mathcal V_m\}$. Denoting the corresponding projection operator by $\Pi_m$, we thus have $\hat f_m=\Pi_m \bar Y$ and 
\begin{align*}
\Vert \hat f_m -f_{\mathcal A} \Vert_{N,M}^2 & \leq \Vert \hat f_m -f \Vert_{N,M} ^2
= \Vert (I-\Pi_m)  \bar f \Vert_{H^N}^2 +\Vert \Pi_m( \bar Y - \bar f )\Vert_{H^N}^2\leq \Vert  \bar f \Vert_{H^N}^2 +\Vert  \bar Y - \bar f \Vert_{H^N}^2 
\end{align*}
since the operator norm of the projections is given by one.
Now,
\begin{align*}
\E(\Vert \bar f  \Vert_{H^N}^2 \1_{\Xi_{N,M,m}^c})=  \E(\Vert f\Vert_{N,M}^2 \1_{\Xi_{N,M,m}^c}) \leq \E(\Vert f(X_{0})\Vert_{\infty}^4)^{1/2} \P(\Xi_{N,M,m}^c)^{1/2} \lesssim \P(\Xi_{N,M,m}^c)^{1/2}
\end{align*}
and, due to $ Y_i = \hat S(0) f(X_{t_i} )+ R_i +\eps_i$, we have
\begin{align*}
\E(\Vert  \bar Y - \bar f \Vert_{H^N}^2  \1_{\Xi_{N,M,m}^c}) &= \frac{1}{N}\sum_{i=0}^{N-1} \E(\Vert R_i + \eps_i\Vert^2_{L^2} \1_{\Xi_{N,M,m}^c})\\
& \lesssim \max_i({\E(\Vert R_i\Vert^4_{L^2})}^{1/2} +{\E(\Vert \eps_i\Vert^4_{L^2} )}^{1/2})\P(\Xi_{N,M,m}^c)^{1/2}.
\end{align*}
It can be shown just like in the proof of Proposition \ref{prop:SpaceDiscrete} that
$\E(\Vert R_i\Vert^4_{L^2}) =\O(1)$
and an explicit calculation yields
\begin{align*}
\E (\Vert \eps_i \Vert^4_{L^2} ) 
&= \frac{\sigma^4}{\Delta^4}  \sum_{\ell,\ell'\geq 1} \E\left(\Big( \int_{t_i}^{t_{i+1}} \e^{-\lambda_\ell(t_{i+1}-s)}\,d\beta_\ell(s) \Big)^2\Big( \int_{t_i}^{t_{i+1}} \e^{-\lambda_{\ell'}(t_{i+1}-s)}\,d\beta_{\ell'}(s) \Big)^2  \right)\\
&\lesssim \frac{1}{\Delta^4}  \Big(\sum_{\ell\geq 1}  \frac{1-\e^{-2\lambda_\ell \Delta}}{2\lambda_\ell} \Big)^2 =\O( \Delta^{-3}).
\end{align*}
Gathering bounds and applying Lemma~\ref{lem:normequivalence_discrete}, we obtain
\begin{align*} 
\E\big(\Vert \hat f_m - f_{\mathcal A}\Vert_{N,M}^2 \1_{\Xi_{N,M,m}^c}\big) &\lesssim \Delta^{-3/2} \P(\Xi_{N,M,m}^c)^{1/2}\lesssim T^{-3/2}N^{-\gamma}=o(\Delta^\gamma).\qedhere
\end{align*}
\end{proof}

\begin{proof}[Proof of Corollary \ref{cor:nonpara_discrete}]
For the bound in probability it suffices to bound $\|\hat f_m-f\|^2_{L^2({\mathcal A})}\1_{\Xi_{N,M,m}}$ since $\P(\Xi_{N,M,m})\to1$ by Lemma~\ref{lem:normequivalence_discrete}. Using $\Vert g\Vert_{N,M}^2\geq \frac{1}{2}\Vert g\Vert_{\pi,M}^2\geq \frac{c}{2}\Vert g\Vert_{L^2({\mathcal A})}^2 $  on $\Xi_{N,M,m}$ for all $g\in \mathcal V_m$ and $\hat f_m -f_m^* \in \mathcal V_m$, we have
\begin{align*}
  \Vert \hat f_m - f \Vert_{L^2({\mathcal A})}^2\1_{\Xi_{N,M,m}}
  &\lesssim \Vert \hat f_m - f_m^* \Vert_{N,M}^2\1_{\Xi_{N,M,m}}+\Vert f_m^* - f \Vert_{L^2({\mathcal A})}^2\\
  &\lesssim \Vert \hat f_m - f_{\mathcal A} \Vert_{N,M}^2\1_{\Xi_{N,M,m}}+\Vert f_{\mathcal A} - f_m^* \Vert_{N,M}^2+\Vert f_m^* - f \Vert_{L^2({\mathcal A})}^2.
\end{align*}
Together with Assumption~(E), Proposition~\ref{prop:SpaceDiscrete} and $\E(\Vert f_{\mathcal A} - f_m^* \Vert_{N,M}^2)=\Vert f_{\mathcal A} - f_m^* \Vert_{\pi,M}^2\lesssim\Vert f_{\mathcal A} - f_m^* \Vert_{L^2({\mathcal A})}^2$, we deduce 
\begin{equation}\label{eq:L2bound}
  \E(\Vert \hat f_m - f \Vert_{L^2({\mathcal A})}^2\1_{\Xi_{N,M,m}})\lesssim \Vert f- f_m^* \Vert_{L^2({\mathcal A})}^2 + \frac{D_m}{T} + \Delta^\gamma +\frac{1}{M\Delta^2} 
\end{equation}
which implies the claimed $\mathcal O_p$-bound for $\Vert \hat f_m - f_{\mathcal A} \Vert_{L^2({\mathcal A})}^2$.

For the truncated estimator $\tilde f_m$, we have 
\begin{align*}
\Vert\tilde f_m -f \Vert_{L^2({\mathcal A})}^2 
\leq \Vert\tilde f_m -f \Vert_{L^2({\mathcal A})}^2\1_{\Xi_{N,M,m}} + 2(\Vert f\Vert_{L^\infty({\mathcal A})}^2 +N^2) \1_{\Xi_{N,M,m}^c} 
\end{align*}
and, thus, as soon as $N \geq\Vert f\Vert_{L^\infty({\mathcal A})} $, we can further bound 
$$\Vert\tilde f_m -f \Vert_{L^2({\mathcal A})}^2\leq \Vert\hat f_m -f \Vert_{L^2({\mathcal A})}^2\1_{\Xi_{N,M,m}} + 4N^2 \1_{\Xi_{N,M,m}^c}.$$ 
The expectation of the first term is bounded by \eqref{eq:L2bound}. The expectation of the second term is $4N^2 \P(\Xi_{N,M,m}^c)$, which decreases faster than any negative power of $N$, thanks to Lemma \ref{lem:normequivalence_discrete}. Therefore,  $N^2 \P(\Xi_{N,M,m}^c)\lesssim N^{-\gamma}\lesssim \Delta^\gamma$, which finishes the proof.
\end{proof}

\begin{proof}[Proof of Theorem \ref{thm:adaptive}]
As seen in the proof of Theorem~\ref{thm:nonpara_discrete}, it suffices to bound the estimation error on the event $\Xi_{N,M,\bar m}$. Since $\Gamma_{N,M}(\hat{f}_{\hat{m}})+\pen(\hat{m})\le\Gamma_{N,M}(f_{m})+\pen(m)$
for any $m\in\mathcal{M}_{N}$ and $f_{m}\in \mathcal V_{m}$, we can modify
the fundamental inequalities \eqref{eq:oracle_prelim} and \eqref{eq:bound_OmegaNm}, respectively,
to 
\begin{align*}
\|\hat{f}_{\hat{m}}-f_{\mathcal A}\|_{N,M}^{2}+2\pen(\hat{m}) & \le3\|f_{m}-f_{{\mathcal A}}\|_{N,M}^{2}+2\pen(m)+\frac{16}{N}\sum_{i=0}^{N-1}\|R_{i}\|_{L^{2}}^{2}\\
 & \qquad+16\Big(\sup_{g\in{ \mathcal V_{\hat m, m}}:\|g\|_{N,M}=1}\frac{1}{N}\sum_{i=0}^{N-1}\langle\widehat{g(X_{t_{i}})},\varepsilon_{i}\rangle_{L^{2}}\Big)^{2}
\end{align*}
with $\mathcal V_{m',m}:= \mathrm{span}(\mathcal V_{m'}\cup \{f_m\})$. In view of \eqref{eq:Rtilde_bound} and Assumption (E), we obtain
\begin{align}
 & \E\big(\|\hat{f}_{\hat{m}}-f_{\mathcal A}\|_{N,M}^{2}\1_{\Xi_{N,M,\bar{m}}}\big)\nonumber \\
 & \quad\le3\|f_{m}-f_{{\mathcal A}}\|_{\pi,M}^{2}+2\pen(m)+\mathcal{O}\Big(\frac{1}{M\Delta^{2}}+\Delta^{\gamma}\Big)\nonumber \\
 & \qquad\qquad+16\E\Big(\Big(\sup_{g\in \mathcal V_{\hat m,m}:\|g\|_{N,M}=1}\Big(\frac{1}{N}\sum_{i=1}^{N-1}\langle\widehat{g(X_{t_{i}})},\varepsilon_{i}\rangle_{L^{2}}\Big)^{2}-\frac{1}{8}\pen(\hat{m})\Big)\1_{\Xi_{N,M,\bar{m}}}\Big)\nonumber \\
 & \quad\lesssim \|f_{m}-f\|_{L^2({\mathcal A})}^2+\pen(m)+\mathcal{O}\Big(\frac{1}{M\Delta^{2}}+\Delta^{\gamma}\Big) \nonumber\\
 & \qquad\qquad+\sum_{m'\in\mathcal{M}_{N}}\E\Big(\Big(\Gamma(m',m)^{2}-\frac{1}{8}\pen(m')\Big)_{+}\1_{\Xi_{N,M,\bar{m}}}\Big)\label{eq:oracleInequAdapt}
\end{align}
with 
\begin{equation}\label{eq:gamma}
\Gamma(m',m):=\sup_{g\in {\mathcal V_{m',m}}:\|g\|_{N,M}=1}\frac{1}{N}\sum_{i=0}^{N-1}\langle\widehat{g(X_{t_{i}})},\varepsilon_{i}\rangle_{L^{2}}.
\end{equation}
Using a martingale concentration of $\frac{1}{N}\sum_{i=1}^{N-1}\langle\widehat{g(X_{t_{i}})},\varepsilon_{i}\rangle_{L^{2}}$ and a classical chaining argument, Lemma~\ref{lem:auxAdapt} shows 
\[
\E\Big(\Big(\Gamma(m',m)^{2}-\frac{1}{8}\pen(m')\Big)_{+}\1_{\Xi_{N,M,\bar{m}}}\Big)\le\frac{2\sigma^{2}}{T}\e^{-D_{m'}}.
\]
Therefore, 
\[
\E\big(\|\hat{f}_{\hat{m}}-f\|_{N,M}^{2}\1_{\Xi_{N,M,\bar{m}}}\big)\lesssim 3\|f_{m}-f\|_{L^2({\mathcal A})}^{2}+2\pen(m)+32\frac{\sigma^{2}}{T}\sum_{m'\in\mathcal{M}_{N}}\e^{-D_{m'}}+\mathcal{O}\Big(\frac{1}{M\Delta^{2}}+\Delta^{\gamma}\Big).
\]
For the considered approximation spaces, we have $\sum_{m'\in\mathcal{M}_{N}}\e^{-D_{m'}}=\mathcal O(1)$ such that this term is negligible. The $L^2({\mathcal A})$-bound for $\tilde f_{\hat m}$ follows exactly as in Corollary~\ref{cor:nonpara_discrete}.
\end{proof}

\begin{proof}[Proof of Theorem \ref{thm:nonparametric_plugin}] We verify the bound for $\Vert \check f_m - f_{\mathcal A}\Vert_{N,M}^2$:
Define
$\Psi_{N,M}^h := \Big\{(\hat \vt -\vt)^2\leq h \frac{\Delta^{3/2}}T\Big\}.$

\emph{Step 1:} We show that  
$\E(\1_{\Psi_{N,M}^h}\Delta^{-2}{\Vert \hat S(\Delta)X_{t_i}-\check S(\Delta)X_{t_i}\Vert_{L^2}^2)}\lesssim T^{-1}:$
For fixed $\underline \vt\in (0,\vt)$, we have $\hat \vt \geq \underline \vt$ on the event $\Psi_{N,M}^h$ as soon as $T$ is sufficiently large. Thus, we can estimate
\begin{align*}
&\E(\1_{\Psi_{N,M}^h}\Delta^{-2}\Vert \hat S(\Delta)X_{t_i}-\check S(\Delta)X_{t_i}\Vert_{L^2}^2)= \E\Big(\1_{\Psi_{N,M}^h}\sum_{k=1}^{M-1} \frac{(\e^{-\lambda_k \Delta}-\e^{-\hat\lambda_k \Delta})^2}{\Delta^2} \langle X_{t_i},e_k\rangle_M^2\Big) \\
&\qquad\qquad\lesssim \E\Big(\1_{\Psi_{N,M}^h}(\vt - \hat \vt)^2 \sum_{k=1}^{M-1} \lambda_{k}^2 \e^{- \underline\vt\pi^2 k^2 \Delta} \langle X_{t_i},e_k\rangle_M^2\Big)
\le h \frac{\Delta^{3/2}}{T}\sum_{k=1}^{M-1} \lambda_{k}^2 \e^{- \underline\vt\pi^2 k^2 \Delta} \E\big(\langle X_{t_i},e_k\rangle_M^2\big).
\end{align*}
The sum is of the order $\Delta^{-3/2}$ due to a Riemann sum argument if we show that $\E(\langle X_{t_i},e_k\rangle_M^2)\lesssim \lambda_k^{-1}$.
To that aim, we apply the decomposition $X_t = S(t)\xi+X_t^0 +N_t$. As in previous results, $S(t)\xi$ is negligible since we can choose $t$ arbitrarily large due to stationarity. For the linear part, we have
\begin{align*}
\E(\langle X^0_{t_i},e_k\rangle_M^2) &\lesssim \sum_{\ell \in \mathcal{I}_k^+\cup\mathcal{I}_k^-} \frac{1}{\lambda_\ell} \lesssim \sum_{\ell \geq 0} \frac{1}{(k+2\ell M)^2} \leq \frac{1}{k^2} \sum_{\ell \geq 0} \frac{1}{(1+2\ell )^2} \lesssim \frac{1}{\lambda_{k}}.
\end{align*}
For the nonlinear part, define 
$n_\ell(t) := \langle N_t,e_\ell\rangle_{L^2} $. Using the Cauchy-Schwarz inequality and the spaces $D_\alpha$ from \eqref{eq:Deps_def}, we have
\begin{align*}
\langle N_{t},e_k \rangle^2_M &\leq \Big( \sum_{\ell \in (\mathcal I_k^+ \cup \mathcal I_k^- )} \lambda_\ell^{2\alpha}n_\ell^2 \Big) \Big( \sum_{\ell \in (\mathcal I_k^+ \cup \mathcal I_k^- )} \lambda_\ell^{-2\alpha}\Big) 
\leq \Vert N_t \Vert_{D_\alpha}^2 \Big( \sum_{\ell \in (\mathcal I_k^+ \cup \mathcal I_k^- )} \lambda_\ell^{-2\alpha}\Big)\\
&\leq \frac{1}{\lambda_k} \Vert N_t \Vert_{D_\alpha}^2 \Big( \sum_{\ell \geq 1} \lambda_\ell^{-(2\alpha-1)}\Big) \lesssim \frac{1}{\lambda_k} \Vert N_t \Vert_{D_\alpha}^2 ,
\end{align*}
provided that $\alpha>3/4$. By picking $\alpha \in (\frac{3}{4},1)$, we get $\E(\langle N_{t},e_k \rangle^2_M) \lesssim \lambda_k^{-1}$ in view of Lemma~\ref{lem:Sobolev_stationary}.

\emph{Step 2}:  
By Markov's inequality, we can estimate 
\begin{align*}
\P \Big( \Vert \check f_m -f_{\mathcal A} \Vert_{N,M}^2 \geq a\Big)
 &\leq
 a^{-1}\E \Big(\1_{ \Psi_{N,M}^h \cap \Xi_{N,M,m}} \Vert \check f_m -f_{\mathcal A} \Vert_{N,M}^2 \Big) +\P((\Psi_{N,M}^h)^c)+\P(\Xi_{N,M,m}^c)
\end{align*}
for any $a>0$.
Now, using Step 1, we can show
$$\E \Big(\1_{ \Psi_{N,M}^h \cap \Xi_{N,M,m}} \Vert \check f_m -f_{\mathcal A} \Vert_{N,M}^2 \Big) \lesssim  \Vert f- f_m^* \Vert_{L^2({\mathcal A})}^2 + \frac{D_m}{T} + \Delta^\gamma  +\frac{1}{M \Delta^2} + \frac{h}{T}$$
just like in the proof of Proposition \ref{prop:SpaceDiscrete}.
Further, $\P(\Xi_{N,M,m}^c)$ converges to 0 under the assumptions of this theorem and, due to Theorem~\ref{thm:optimal_rate}, $\P((\Psi_{N,M}^h)^c)$ can be made arbitrarily small by choosing $h$ sufficiently large. 
Since for any fixed $h>0$ we have $h/T\lesssim T^{-1}\le D_m/T$, we have shown that, for arbitrary $\eps>0$, we can pick $K>0$ such that
\begin{equation*} \limsup_{M,N \to \infty}\P \Big( \Vert \check f_m -f_{\mathcal A} \Vert_{N,M}^2 \geq K\big(\Vert f_{\mathcal A}- f_m^* \Vert_{L^2({\mathcal A})}^2 + \frac{D_m}{T} + \Delta^\gamma  +\frac{1}{M \Delta^2} \big)\Big)<\eps. 
\end{equation*}
 From here, the bound for $\Vert \check f_m - f_{\mathcal A}\Vert_{L^2({\mathcal A})}^2$ follows as in the proof of Corollary \ref{cor:nonpara_discrete}. The proof of Theorem~\ref{thm:adaptive} for $\hat f_{\hat m}$ is applicable with the same modification such that the result for $\check f_{\check m}$ follows.
\end{proof}

\subsection{Proofs for the estimators of $\sigma^2$ and $\vt$}
\label{subsec:para_proofs}
In the following, we prove the central limit Theorems  \ref{thm:CLT_time_nonlinear}, \ref{thm:CLT_space_nonlinear} and  \ref{thm:CLT_double_nonlinear} for the realized quadratic variations in the semilinear framework.  Central limit theorems for the derived method of moments estimators for $\sigma^2$ and $\vt$ follow directly in view of the delta method. 

Note that the central limit theorems for space and double increments have been derived in \cite{Hildebrandt19} for the linear case $f\equiv 0$, assuming a stationary initial condition. A perturbation argument together with Slutsky's lemma shows that these central limit theorems remain valid in the case $X_0=0$   when using the slight modification explained in Section \ref{sec:parametric_nonlinear}. We omit a detailed verification for the sake of brevity.
	\begin{proof}[Proof of Theorem \ref{thm:CLT_time_nonlinear}]
	It is sufficient to consider the case of zero initial condition: if $\xi$ follows the stationary distribution, then $X$ has the same distribution as $\tilde X$ with $\tilde X_t := X_{\tau+t}=S(t)S(\tau) \xi +X_{t+\tau}^0+N_{t+\tau}$ for any $\tau>0$. Now, $(S(t)S(\tau)\xi)_{t\geq 0}$ becomes negligible when choosing $\tau$ sufficiently large and the properties of $(N_{t+\tau})_t$ and $(X^0_{t+\tau})_t$ used in the sequel are not affected by the the initial condition and the value of $\tau >0$.
		
	We have the decomposition
	\begin{align*}
	 V_{\mathrm t}&= \frac{1}{MN\sqrt \Delta}\sum_{i=0}^{N-1} \sum_{k=0}^{M-1} (X^0_{t_{i+1}}(y_k)-X^0_{t_{i}}(y_k))^2
	+\frac{1}{MN\sqrt \Delta}\sum_{i=0}^{N-1} \sum_{k=0}^{M-1} (N_{t_{i+1}}(y_k)-N_{t_{i}}(y_k))^2\\
	&\qquad\qquad\qquad+ \frac{2}{MN\sqrt \Delta}\sum_{i=0}^{N-1} \sum_{k=0}^{M-1} (X^0_{t_{i+1}}(y_k)-X^0_{t_{i}}(y_k))(N_{t_{i+1}}(y_k)-N_{t_{i}}(y_k))
	 =: \bar V_{\mathrm t}+R_1+R_2.
	\end{align*}
	Since $ \bar V_{\mathrm t}$ satisfies the claimed central limit theorem, due to Slutsky's lemma, it suffices to prove that $R_1$ and $R_2$ are of the order $o_p(1/\sqrt{MN})$.
	
	 If $T$ is finite, it follows from Lemma \ref{lem:linear} and Proposition \ref{prop:AN_reg} that for all $\gamma <1/4$ and $\P$-almost all realizations $\omega\in \Omega$, there exists a constant $C=C(\omega,T)$ such that $|X^0_{t_{i+1}}(y_k)-X^0_{t_{i}}(y_k)|\leq C \Delta^\gamma$ and $|N_{t_{i+1}}(y_k)-N_{t_{i}}(y_k)|\leq C \Delta$ for all $i\leq N$, $k\leq M$ and $N,M\in \N$. Consequently, $R_1$ and $R_2$ are of the order $o_p(\Delta^{\frac{1}{2}+\gamma})$ and the statement follows due to the condition $M= o( \Delta^{-\rho})$ for some $\rho < 1/2$.
	 
	If $T\to \infty$ and Assumption (B) is satisfied, Lemma \ref{lem:linear} and Proposition \ref{prop:AN_reg} yield $\E(|R_1|)\lesssim \Delta ^{3/2}$ and, by applying the Cauchy-Schwarz inequality to the cross terms, we get  $\E(|R_2|)\lesssim \Delta ^{\frac{1}{2}+\gamma}$ for any $\gamma <1/4$. The claim follows since  $\sqrt{MN} \Delta^{\frac{1}{2}+\gamma}= \sqrt{TM \Delta^{2\gamma}} $ converges to 0 for any $\gamma \in (\frac{\rho}{2},\frac{1}{4})$ and the fact that convergence in $L^1(\P)$ implies convergence in probability.
	\end{proof}

	\begin{proof}[Proof of Theorem \ref{thm:CLT_space_nonlinear}]
	We only consider the case of a finite time horizon, the case $T\to\infty$ can be treated similarly by taking expectations. Further, it suffices to consider the case $\xi=0$, see also the proof of Theorem \ref{thm:CLT_time_nonlinear}.
	We have
	\begin{align*}
	 V_{\mathrm{sp}}&= \frac{1}{MN\delta}\sum_{i=1}^{N} \sum_{k=0}^{M-1} (X^0_{t_{i}}(y_{k+1})-X^0_{t_{i}}(y_{k}))^2
	+\frac{1}{MN\delta}\sum_{i=1}^{N} \sum_{k=0}^{M-1} (N_{t_{i}}(y_{k+1})-N_{t_{i}}(y_k))^2\\
	&\qquad\qquad\qquad+ \frac{2}{MN\delta}\sum_{i=1}^{N} \sum_{k=0}^{M-1} (X^0_{t_{i}}(y_{k+1})-X^0_{t_{i}}(y_k))(N_{t_{i}}(y_{k+1})-N_{t_{i}}(y_k))
	 =: \bar V_{\mathrm{sp}}+R_1+R_2
	\end{align*}
	and the claim follows if $R_1$ and $R_2$ are of the order $o_p(1/\sqrt {MN})$. 
	
To bound the term $R_2$, we use the summation by parts formula
\begin{equation} \label{eq:sum_by_parts}
\sum_{k=0}^{M-1}a_k(b_{k+1}-b_k)=-\sum_{k=0}^{M-2}(a_{k+1}-a_k)b_{k+1}+a_{M-1}b_M-a_0b_0.
\end{equation}
Setting $a_k:= N_{t_i}(y_{k+1})-N_{t_i}(y_k)$ and $b_k:=X_{t_i}^0(y_k)$, we get
\begin{align*}
R_2
&=\frac{2}{MN \delta}\sum_{i=1}^{N}\sum_{k=0}^{M-2} X_{t_i}^0(y_{k+1}) (N_{t_i}(y_{k+2})-2N_{t_i}(y_{k+1})+N_{t_i}(y_k))\\
&\qquad+\frac{1}{MN\delta}\sum_{i=1}^{N}\left( (N_{t_i}(y_N)-N_{t_i}(y_{N-1}))X^0_{t_i}(y_N)+(N_{t_i}(y_1)-N_{t_i}(y_{0}))X^0_{t_i}(y_0)\right). 
\end{align*} 
By Lemma \ref{lem:linear} and Proposition \ref{prop:AN_reg}, we have $(X^0_t)\in C(\R_+,C([b,1-b]))$ and $\sup_{t\leq T} \Vert {A_\vt}N_t \Vert_{C([b,1-b])}\linebreak<\infty $ almost surely. Thus, there exists a random variable $C=C(\omega,T)$ with $|N_{t_i}(y_{k+2})-2N_{t_i}(y_{k+1}) -N_{t_i}(y_{k})| \leq C\delta^2$, $|N_{t_i}(y_{k+1})-N_{t_i}(y_k)| \leq C\delta$ and $|X^0_{t_i}|\lesssim C$ for all $i\leq N,\,k\leq M-1$ and $M,N\in \N$ almost surely. It follows that $|R_1|\leq C^2 \delta$ and $|R_2|\lesssim C^2 \delta$ hold almost surely and, therefore, the claim follows from the fact that $\sqrt{MN}\delta \eqsim \sqrt{{N}/{M}}$ tends to 0, by assumption.
\end{proof}

To prove the result for double increments, define $\mathbf N_{ik} := N_{t_{i+1}}(y_{k+1})-N_{t_{i+1}}(y_k)-N_{t_{i}}(y_{k+1})+N_{t_{i}}(y_k)$.
The main ingredient of the proof of Theorem \ref{thm:CLT_double_nonlinear} is the following lemma.
\begin{lem} \label{lem:Nik_pnorm} 
Assume that the constant $b$ from the observation scheme defined in Section \ref{subsec:observation_scheme} is strictly positive and let $p\geq 1$. 
\begin{enumerate}[(i)]
\item 
Let $\alpha \in (0,1)$ and $\beta \in (0,1]$ be such that $\alpha+\beta < \frac{3}{2}$.  If $T$ is finite, then there exists a random variable $C=C(\omega,T)>0$ such that
$
|\mathbf N_{ik}|\leq C  \delta^\alpha \Delta ^\frac{1+\beta}{2}
$
holds for all $i\leq N,k\leq M$ and $N,M\in \N$ almost surely. If Assumption (B) is satisfied, then there exists a constant $C>0$ such that
$
\E(|\mathbf N_{ik}|^p)\leq C \Big( \delta^\alpha \Delta ^\frac{1+\beta}{2}\Big)^p
$
holds for all $i\leq N,k\leq M$, $N,M\in \N$ uniformly in $T>0$.
\item Let $\gamma<2$ and $\eps<1/4$. If $T$ is finite, then there exists a random variable $C=C(\omega,T)>0$ such that
$
|\mathbf N_{i(k+1)}-\mathbf N_{ik}|\leq C  \delta^\gamma \Delta^\eps
$
holds for all $i\leq N,k\leq M$ and $N,M\in \N$ almost surely. If Assumption (B) is satisfied, then there exists a constant $C>0$ such that
$
\E(|\mathbf N_{i(k+1)}-\mathbf N_{ik}|^p)\leq C \big( \delta^\gamma \Delta^\eps \big)^p
$
holds for all $i\leq N,k\leq M$, $N,M\in \N$ uniformly in $T>0$.
\end{enumerate}
\end{lem}

\begin{proof}
We write $\mathbf N_{ik} = \mathbf N_{ik}^0+ \mathbf M_{ik}$ where $\mathbf N_{ik}^0$ and $\mathbf M_{ik}$ are the double increments computed from the processes $(N^0_t)$ and $(M_t) $  defined by \eqref{eq:Nt0Mt_def}, respectively. In the following, these double increments are estimated separately.

$(i)$ For $\alpha \in (0,1)$, we have 
\begin{align*}
| \mathbf N_{ik}^0| &\leq \delta^\alpha \Vert N^0_{t_{i+1}}-N^0_{t_{i}} \Vert_{C_0^\alpha}
\leq \delta^\alpha \left( \Vert (S(\Delta)-I)N^0_{t_i}\Vert_{C_0^\alpha} + \Big\Vert \int_{t_i}^{t_{i+1}}S(t_{i+1}-s) f_0(X_s)\,ds \Big\Vert_{C_0^\alpha}\right) .
\end{align*}
Further, using formula \eqref{eq:Sx-x} and Lemma \ref{lem:intermediate_norms} $(iv)$ in combination with $\alpha+\beta-1\leq \alpha$ yields 
\begin{align*}
\Vert (S(\Delta)-I)N^0_{t_i}\Vert_{C_0^\alpha} &= \Big\Vert \int_0 ^\Delta {A_\vt} S(r ) N^0_{t_i}\,dr\Big\Vert_{C_0^\alpha}
\leq \int_0^\Delta \Vert S(r)\Vert_{ L(C_0^{\alpha+\beta-1},C_0^\alpha)} \Vert {A_\vt}N_{t_i}^0 \Vert_{C_0^{\alpha+\beta-1}} \,dr\\
&\lesssim \int_0^\Delta {r^{-\frac{1-\beta}{2}}} \Vert {A_\vt}N_{t_i}^0 \Vert_{C_0^{\alpha+\beta-1}} \,dr \lesssim \Delta^{\frac{1+\beta}{2}} \Vert {A_\vt}N_t^0 \Vert_{C_0^{\alpha+\beta-1}}. 
\end{align*}
Similarly, by Lemma \ref{lem:intermediate_norms} $(iii)$ and H{\"o}lder's inequality,
\begin{align*}
  \Big\Vert \int_{t_i}^{t_{i+1}}S(t_{i+1}-s) f_0(X_s)\,ds \Big\Vert_{C_0^\alpha} & \lesssim \int_{t_i}^{t_{i+1}} {{(t_{i+1}-s)}^{-\frac{1-\beta}{2}}} \Vert f_0(X_s)\Vert_{C_0^{\alpha+\beta-1}} \,ds\\
 & \lesssim \left(\int_{t_i}^{t_{i+1}} \Vert f_0(X_s) \Vert^p_{C_0^{\alpha+\beta-1}} \,ds \right)^{\frac{1}{p}} \Delta^{1-\frac{1}{p}-\frac{ 1-\beta}{2}}.
\end{align*}
Thus, noting $\alpha+\beta-1<\frac{1}{2}$,  Propositions \ref{prop:X_reg} and \ref{prop:AN_reg} yield the claim for the case of a finite time horizon and, under Assumption (B),
\begin{align*}
\E(|\mathbf N^0_{ik}|^p) &\lesssim \delta^{p\alpha} \left(\Delta^{p\frac{1+\beta}{2}} \E(\Vert {A_\vt}N^0_t \Vert_{C_0^{\alpha+\beta-1}}^p)+\Delta^{p-1-p\frac{1-\beta}{2}} \int_{t_i}^{t_{i+1}}\E(\Vert f_0(X_s)\Vert_{C_0^{\alpha+\beta -1 }}^p)\,ds  \right)
\lesssim \delta^{p\alpha} \Delta^{p\frac{1+\beta}{2}}.
\end{align*}

To verify that $\mathbf M_{ik}$ is of the claimed order,  recall that in the proof of Proposition \ref{prop:AN_reg} it is shown that $\frac{d}{dt}M_t = S(t)m=: m_t$ and that $|m_t(x)-m_t(y)|\lesssim |x-y|$ holds uniformly in $t>0$ and $x,y\in [b,1-b]$. Thus, we have
$
\mathbf M_{ik} = \int_{t_i}^{t_{i+1}}  (m_s(y_{k+1})-m_s(y_k))\,ds
$
and, consequently, $|\mathbf M_{ik} |\lesssim \Delta \delta\lesssim \delta^\alpha \Delta^{\frac{1+\beta}{2}}$. 

$(ii)$ For $\gamma \in (1,2)$, we have
$
|\mathbf N^0_{i(k+1)}-\mathbf N^0_{ik}| \leq \delta^{\gamma} \Vert N^0_{t_{i+1}}-N^0_{t_i}\Vert_{C_0^{\gamma}}.
$
Using the decomposition 
$$N^0_{t_{i+1}}-N^0_{t_i}= \int_0^{t_i} S(t_i-s)(f_0(X_{s+\Delta})-f_0(X_{s}))\,ds+ \int_0^\Delta S(t_{i+1}-s)f_0(X_s)\,ds,$$
we get from Lemma \ref{lem:intermediate_norms} $(i)$ that  
\begin{align*}
\Big\Vert\int_0^{t_i} S(t_i-s)(f_0(X_{s+\Delta})-f_0(X_{s}))\,ds \Big\Vert_{C_0^{\gamma}} &= \Big\Vert \int_0^{t_i} S(r)(f_0(X_{t_{i+1}-r})-f_0(X_{t_{i}-r}))\,dr \Big \Vert_{C_0^{\gamma}}\\
&\lesssim \int_0^{t_i} \e^{-\lambda_0 r}r^{-\frac{\gamma}{2}} \Vert f_0(X_{t_{i+1}-r})-f_0(X_{t_{i}-r}) \Vert_\infty\,dr.
\end{align*}
Further, for $h<1/2$, Lemma \ref{lem:intermediate_norms} $(iii)$ gives
\begin{align*}
\Big\Vert \int_0^\Delta S(t_{i+1}-s)f_0(X_s)\,ds \Big\Vert_{C_0^{\gamma}} \lesssim \int_{0}^\Delta (t_{i+1}-r)^{-\frac{\gamma-h}{2}} \Vert f_0(X_{r})\Vert_{C_0^h} \,dr.
\end{align*}
Now, the result in case of a fixed $T$ follows from the path regularity of $f_0(X)$. Further, under Assumption (B), we can  use Jensen's and H{\"o}lder's inequality to estimate
\begin{align*}
\E (|\mathbf N^0_{i(k+1)}-\mathbf N^0_{ik}|^p) &\lesssim \delta^\gamma\sup_{t\geq 0} \E(\Vert f_0(X_{t+\Delta})-f_0(X_t) \Vert_\infty^p)+\delta^\gamma\Delta^{p-1-\frac{\gamma-h}{2}p} \int_0^\Delta \E(\Vert f_0(X_{t_{i+1}-r}) \Vert_{C_0^h}^p)\,dr\\
&\lesssim \delta^\gamma\Delta^{p\eps}+\delta^\gamma\Delta^{p(1-\frac{\gamma-h}{2})}.
\end{align*}
The result follows, since one can pick $h\in (0,\frac{1}{2})$ such that $ 1-\frac{\gamma-h}{2}\geq \eps$.

To estimate $|\mathbf M_{ik}|$, recall that in the proof of Proposition \ref{prop:AN_reg} it is shown that $\frac{\partial^2}{\partial x^2}M_t = \frac{1}{\vt}{A_\vt} M_t= \frac{1}{\vt}(S(t)-I)m$ and that $\Vert \frac{\partial^2}{\partial x^2}M_t-\frac{\partial^2}{\partial x^2}M_s \Vert_{C([b,1-b])}= \frac{1}{\vt}\Vert S(t)m-S(s)m \Vert_{C([b,1-b])}\lesssim \sqrt{|t-s|}$.
Further, recall that by Taylor's formula, we have the expansion $h(x+\delta)= h(x) + \delta h'(x) + \int_x^{x+\delta}(x+\delta -z)h''(z)\,dz$ for any $h\in C^2(\R)$. Hence, we can write
\begin{align*} 
\delta^{-2}(h(x+\delta)-2h(x)+h(x-\delta))= \int K_\delta(z-x)h''(z)\,dz
\end{align*}
with $K_\delta(z):=\delta^{-1}K(\delta^{-1}z)$ and the triangular kernel $K(z):=(1-|z|)\1_{\{-1\leq z\leq 1\}}$.
Application to the double increments yields 
\begin{align*}
\mathbf M _{i(k+1)}-\mathbf M_{ik}= \delta^2 \int_{x-\delta}^{x+\delta} K_\delta(z-x) \frac{\partial^2}{\partial z^2} (M_{t_{i+1}}(z)-M_{t_{i}}(z)) \,dz  
\end{align*} 
and, consequently, $|\mathbf M_{i(k+1)}-\mathbf M_{ik}| \lesssim \delta^2 \sqrt \Delta\lesssim \Delta^\gamma \Delta^\eps$. 
\end{proof}

\begin{proof}[Proof of Theorem \ref{thm:CLT_double_nonlinear}]
As for time and space increments, we can assume $\xi=0$ and the claim follows if we verify $|R_i|=o_p(1/{\sqrt{MN}}),\,i\in \{1,2\}$, with
$$R_1 := \frac{1}{MN \Phi_\vt(\delta,\Delta)} \sum_{i=1}^{N} \sum_{k=0}^{M-1} \mathbf N_{ik}^2\qquad\text{and} \qquad R_2 :=\frac{1}{MN \Phi_\vt(\delta,\Delta)} \sum_{i=1}^{N} \sum_{k=0}^{M-1} \mathbf N_{ik} D_{ik},$$ 
where $D_{ik}$ are the double increments computed from $(X^0_t)$ and $\Phi_\theta(\delta,\Delta)$ is given in \eqref{eq:PhiTheta}. In the following, we verify the claim under Assumption (B). The result for the case of a fixed $T$ can be shown analogously by using the pathwise properties of $(N_t)$ derived in Lemma \ref{lem:Nik_pnorm}.
We treat the cases $M\sqrt \Delta  = \O(1)$ and $M\sqrt \Delta  \to \infty$ separately.

\emph{Case $M\sqrt \Delta = \O(1)$:} 
Using Lemma \ref{lem:Nik_pnorm} with $\alpha=0$ and $\beta=1$ yields  $\E(\mathbf N_{ik}^2)\lesssim \Delta^{2}$ and, hence,
\begin{align*}
\sqrt{MN}\E(|R_1|)\lesssim \frac{\sqrt{MN}}{MN \sqrt \Delta  } \sum_{i,k}\E(\mathbf N_{ik}^2)\lesssim \sqrt{MN} \Delta^{3/2}  \to 0.
\end{align*}
For the cross terms, we take $\beta=1$ and $\alpha=\frac{a}{2}<\frac{1}{2}$ in Lemma \ref{lem:Nik_pnorm} to bound
$$\E(|D_{ik} \mathbf N_{ik}|)\leq \E(D_{ik}^2 )^{1/2}\E(\mathbf N_{ik}^2 )^{1/2}\lesssim \Delta^{1/4}\Delta \delta ^{a/2},$$ implying 
$\sqrt{MN}\E(|R_1|)
\to 0.$

\emph{Case $M\sqrt \Delta \to \infty$: } 
With $\beta=1/2$ and $\alpha=\frac{a+3}{4}<1$ in Lemma \ref{lem:Nik_pnorm}, we get $\E(\mathbf N_{ik}^2)\lesssim \Delta^{3/2}\delta^{2\alpha}$ and, hence,
$
\sqrt{MN} \E(|R_1|)
\to 0.
$
To treat the cross terms, we use formula \eqref{eq:sum_by_parts} with $a_k:=\mathbf N_{ik}$ and $b_k:=H_{ik}:=X^0_{t_{i+1}}(y_k)-X^0_{t_{i}}(y_k)$  to deduce
\begin{align}
\sum_{i=1}^{N} \sum_{k=0}^{M-1}D_{ik} \mathbf N_{ik} = &-\sum_{i=1}^{N} \sum_{k=0}^{M-2}(\mathbf N_{i(k+1)}-\mathbf N_{ik})H_{i(k+1)} 
+ \sum_{i=1}^{N} \mathbf N_{i(M-1)}H_{iM}-\sum_{i=1}^{N} \mathbf N_{i0}H_{i0}. \label{eq:sum_py_parts_CLT}
\end{align}
Since $\E(H_{ik}^2)\lesssim \sqrt \Delta$,  Lemma \ref{lem:Nik_pnorm} gives for any $\gamma<2$ and $\eps<1/4$ that
\begin{align*}
&\E \Big(\Big|\frac{1}{MN\delta} \sum_{i=1}^{N} \sum_{k=0}^{M-2}(\mathbf N_{i(k+1)}-\mathbf N_{ik})H_{i(k+1)} \Big|\Big)\\
&\leq 
\frac{1}{MN \delta} \sum_{i=1}^{N} \sum_{k=0}^{M-2}\E((\mathbf N_{i(k+1)}-\mathbf N_{ik})^2)^{1/2}\E(H_{i(k+1)}^2)^{1/2}
\lesssim 
\frac{\delta^\gamma \Delta^\eps \sqrt \Delta}{\delta} 
=\delta^{\gamma-1} \Delta^{\eps+1/4}.
\end{align*}
Further, by picking $\eps$ and $\gamma$ in such a way that $2\gamma-4+4\eps>a$, we get $\sqrt{MN}\delta^{\gamma-1}\Delta^{\eps + 1/4}
\to 0.$ 
For the remaining two terms in \eqref{eq:sum_py_parts_CLT}, take $\alpha=\frac{a+1}{2}<1$ and $\beta = 1/2$ in Lemma \ref{lem:Nik_pnorm}. Then,
\begin{align*}
\E \Big( \Big|\frac{1}{MN \delta}\sum_{i=1}^{N} \mathbf N_{i(M-1)}H_{iM} \Big|\Big) \lesssim  \frac{ \Delta^{1/4}\delta^\alpha \Delta^{3/4}}{M\delta} \eqsim\Delta \delta^{\gamma}
\end{align*}
and, since $\sqrt {MN} \Delta \delta^\alpha
\to 0$, we obtain $\sqrt{MN}\E(|R_2|) \to 0$, which finishes the proof.
\end{proof}

\subsection{Further proofs and auxiliary results}
\label{subsec:aux_nonlinear}
\begin{proof}[Proof of Proposition \ref{prop:f_coerc}]
First, we sketch the existence proof and show that Assumption (B) is satisfied for $\xi=0$. To that aim, we follow the line of arguments from \cite[Theorem 7.7]{DaPrato14}, see also \cite[Propositon 6.1]{Goldys06}.
As before, write $m\equiv f(0)$ as well as $f_0(x)=f(x)-m$ and decompose $X_t= w(t)+v(t) $ with $w(t):= X_t^0 + \int_0^t S(r) m\,dr$ and $v(t): = S(t)\xi + \int_0^t S(t-s) f_0(X_s)\,dt$. It follows from Lemma \ref{lem:linear} and $\Vert S(r)m\Vert_\infty \lesssim \e^{-\lambda_0 r} \Vert m \Vert_\infty $ (cf. Lemma \ref{lem:intermediate_norms}) that $w \in C(\R_+,E)$ holds almost surely and 
\begin{equation} \label{eq:wt_pnorm}
\sup_{t\geq 0}\E(\Vert w_t \Vert_\infty^p)<\infty.
\end{equation}
Further, since $F_0(u):=f_0\circ u$ is a locally Lipschitz continuous function from $E$ into itself, there exists a solution to equation \eqref{eq:SPDE1} up to a terminal time $t_{\mathrm{max}}= t_{\mathrm{max}}(\omega)>0$. Thus, global existence follows from an a priori estimate on $\Vert v(\cdot)\Vert_\infty$. We consider the approximation $v_n := nR(n,{A_\vt})S(t)\xi+\int_0^t nR(n,{A_\vt}) S(t-s)f_0(v(s)+w(s))\,ds$ where $R(n,{A_\vt}):=(nI-{A_\vt})^{-1}$ is the resolvent operator of ${A_\vt}$. Then, $v_n$ is differentiable in time, even when $v$ is not. Now, for any $x \in E$ and $x^* \in  \partial \Vert x\Vert $, it follows like in \cite[Example 7.8]{DaPrato14}  that  $\langle {A_\vt} x,x^* \rangle \leq 0$ where $\partial \Vert x\Vert$ is the subdifferential of the norm. Recall that, for a function $u\in E$,  the functional $h_u\colon E\ni v\mapsto \sigma_u v(\xi_u)$ with $\xi_u \in \argmax_r |u(r)|$ and $\sigma_u:= \mathrm{sgn}(u(\xi_u))$ is an element of $\partial \Vert u \Vert$.  Thus, setting $\delta_n(t) := v_n'(t)-{A_\vt}v_n - f(v_n(t)+w(t))$, we can estimate
\begin{align*}
\frac{d^-}{dt} \Vert v_n(t) \Vert_\infty \leq  \big\langle \frac{d}{dt}v_n(t), h_{v_n(t)} \big\rangle &=   \langle {A_\vt} v_n(t), h_{v_n(t)} \rangle  + \langle f(v_n(t)+w(t)) ,h_{v_n(t)}\rangle +\langle \delta_n(t),h_{v(t)}\rangle\\
&\leq \langle f(v_n(t)+w(t)) ,h_{v_n(t)}\rangle +\Vert  \delta_n(t)\Vert_\infty\\
&\leq -a\Vert  v_n(t) \Vert_\infty+b\Vert w(t)\Vert_\infty^\beta  + c +\Vert \delta_n(t) \Vert_\infty.
\end{align*} 
Using Gronwall's inequality and the fact that $v_n(t)\to v(t)$ and $\delta_n(t) \to 0$ uniformly on compact time intervals yields
$ \Vert v(t)\Vert_\infty \leq \e^{-a t}\Vert \xi\Vert_\infty + \int_0^t\e^{-a(t-s)}(b\Vert w(s)\Vert_\infty^\beta  + c)\,ds.$
By Jensen's inequality, we have
\begin{align*}
\Vert v(t) \Vert_\infty^p \lesssim \e^{-apt}\Vert \xi \Vert_\infty^p + \int_0^t\e^{-a(t-s)}(b\Vert w(t)\Vert_\infty^\beta  + c)^p\,ds \cdot \Big( \int_0^t \e^{-as}\,ds \Big)^{p-1}
\end{align*}
and Fubini's theorem as well as \eqref{eq:wt_pnorm} show that there exists $K>0$ such that for non-random initial conditions $\xi = x \in E$, we have 
\begin{equation} \label{eq:nonrandomIC_bound}
\E(\Vert X_t\Vert^p_\infty ) \lesssim \e^{-apt} \Vert x \Vert_\infty^p + K.
\end{equation}
In particular, Assumption (B) with $\xi =0$ is satisfied. 

Further, based on their derivation of lower bounds for the transition densities associated with the Markov semigroup $(P_t)$, \cite[Theorem 6.3]{Goldys06} show the existence of an invariant measure $\pi$ on $E$ and  of constants $C,\gamma >0$ such that 
$\Vert P_t^*\nu-\pi\Vert_{\mathrm{TV}}\leq C \Big(\int_E \Vert u\Vert_\infty \,\nu(du)+1\Big)\e^{-\gamma t}$ with $P_t^*\nu := \int_E P_t(u,\cdot)\,\nu(du)$ holds for all probability measures $\nu$ on $E$. Thus, we have $\Vert P_t(x,\cdot)-\pi\Vert_{\mathrm{TV}}\leq C ( \Vert x\Vert_\infty +1)\e^{-\gamma t}$ and $P_t(x,\cdot)$ converges weakly to $\pi(\cdot)$ as $t \to \infty$ for all $x\in E$. By Skorokhod's representation theorem, there exists a probability space on which there are $E$-valued random variables $Z,Z_1,Z_2,\ldots$ with $Z_i \sim P_i(x, \cdot),\, Z \sim \pi$ and $Z_i \to Z$ almost surely. Denoting the expectation on the second probability space by $\tilde{\E}$, Fatou's Lemma and \eqref{eq:nonrandomIC_bound} yield
$$\int_E \Vert u \Vert_\infty^p\,\pi(du)=\tilde{\E}(\Vert Z\Vert_\infty^p ) \leq \liminf_{i\to \infty} \tilde{\E}(\Vert Z_i\Vert_\infty^p ) = \liminf_{i\to \infty} \int_{E} \Vert u\Vert_\infty^p\,P_i(x,du)<\infty.$$
Thus, if $X_0=\xi\sim \pi$, then we have $\E(\Vert X_t\Vert_\infty^p)=\E(\Vert X_0\Vert_\infty^p)<\infty$ and 
$\int_E \Vert P_t(u,\cdot)-\pi\Vert_{\mathrm{TV}}\, \pi(du) \leq C \Big( \int_E \Vert u \Vert_\infty\,\pi(du) +1\Big)\e^{-\gamma t} \lesssim \e^{-\gamma t},
$
as required for (M) as well as (B)  in case of $\xi\sim\pi$.
\end{proof}

\subsubsection{Technical Lemmas for the nonparametric estimator of $f$}
\begin{proof}[Proof of Lemma \ref{lem:L2norm_approx}]
In view of Dini's test, the H\"{o}lder condition implies convergence of the Fourier series of $H$ at the points $y_k$, i.e., $\bar H^n(y_k) := \sum_{l=1}^n h_le_l(y_k) \to H(y_k) $ as $n \to \infty$ for any $1\leq k\leq M-1$. 
Therefore, 
$
|\langle H ,e_k \rangle_M-\langle \bar H^n ,e_k \rangle_M| \leq \frac{1}{M}\sum_{l=1}^{M-1} |H(y_l)-\bar H^n(y_l)||e_k(y_l)|
$
tends to 0 as $n\to \infty$. Hence, the sequence $\langle \bar H^n,e_k\rangle_M =\sum_{l \in \mathcal I_k^+\cap[1,n]}h_l-\sum_{l \in \mathcal I_k^-\cap[1,n]}h_l $ converges to the limit $\langle H,e_k \rangle_M$, proving the first part of the lemma. In the same way, using $e_l(y_k)= \pm e_j(y_k)$ for $l\in \mathcal I_j^\pm$,  one can show that $H(y_k) = \sum_{l=1}^{M-1}H_l e_l(y_k)$. Consequently,
\begin{align*}
\frac{1}{M}\sum_{k=1}^{M-1} H^2(y_k) =\frac{1}{M}\sum_{k=1}^{M-1} \Big( \sum_{l=1}^{M-1}H_l e_l(y_k) \Big)^2 =  \sum_{l,l'=1}^{M-1}H_l H_{l'} \langle e_l,e_{l'} \rangle_M = \sum_{l=1}^{M-1} H_l^2 = \Vert H^M \Vert_{L^2}^2.\tag*{\qedhere}  
\end{align*} 
\end{proof}

The following lemma analyzes the regularity of $X_t$ in the spaces 
\begin{equation} \label{eq:Deps_def}
D_\eps := \mathcal{D}((-{A_\vt})^\eps) := \big\{u\in L^2((0,1)):\, \sum_{k\geq 1} \lambda_k^{2\eps}\langle u,e_k\rangle^2<\infty \big\}
\end{equation}
endowed with the norm $\Vert u\Vert_{D_\eps}:= \Vert (-{A_\vt})^\eps u\Vert_{L^2}$. For $\eps<1/4$, these spaces can be identified with $L^2$-Sobolov spaces on $(0,1)$, namely $D_\eps = W^{2\eps,2}$ and the norms are equivalent, see, e.g., \cite{Bonforte15}.
\begin{lem} \label{lem:Sobolev_stationary}
Under Assumption $(M)$,  we have $ \E(\Vert X_t\Vert_{D_\eps}^p) = \E(\Vert X_0\Vert_{D_\eps}^p)< \infty$ and $ \E(\Vert f(X_t)\Vert_{D_\eps}^p)= \E(\Vert f(X_0)\Vert_{D_\eps}^p)< \infty$ for all $\eps <1/4$ and $p\geq 1$. Moreover,  $\E(\Vert N_t\Vert_{D_\eps}^p)<\infty$ holds for all $\eps<1$.
\end{lem}
\begin{proof}
We use the usual decomposition $X_t = S(t)X_0 + X_t^0 +N_t$. By stationarity, we may choose $t=1$.
As before, $\E(\Vert  X_1^0\Vert_{D_\eps}^p)<\infty$ for $\eps<1/4$ can be shown  by a direct calculation. Further, $\E(\Vert S(1)X_0\Vert_{D_\eps}^p)<\infty$ follows from 
\begin{align*}
\Vert S(1)X_0\Vert_{D_\eps}^2 = \sum_{k\geq 1} \e^{-2\lambda_k} \lambda_k^{2\eps} \langle X_0,e_k \rangle^2 \leq \Vert X_0\Vert^2_{L^2} \sum_{k\geq 1} \e^{-2\lambda_k} \lambda_k^{2\eps}  \lesssim \Vert X_0\Vert^2_\infty.
\end{align*}
To treat $N_1 = \int_0^1 S(1-s)f(X_s)\,ds$, we note that
$
\Vert (-A_\vt)^\eps S(h)u\Vert_{L^2}^2 = \sum_{k\geq 1} \lambda_k^{2\eps} \e^{-2\lambda_k h} \langle u,e_k \rangle^2 \leq \sup_{\lambda \geq \lambda_1} \lambda^{2\eps} \e^{-2\lambda h} \Vert u \Vert^2_{L^2}. 
$ The function $\lambda\mapsto \lambda ^{2\eps}\e^{-2\lambda h}$ attains its maximum over $\R_+$ in $\lambda^* := \eps/ h$ and is monotonically decreasing on $[\lambda^*,\infty)$. Thus, we have  $\sup_{\lambda \geq \lambda_1} \lambda^{2\eps} \e^{-2\lambda h} \leq g^2(h)$
with $g(h):= (\frac{\eps}{\e h})^{\eps}$ for $h\leq \eps/\lambda_1 $ and $g(h):=\lambda_1^{\eps}\e^{-\lambda_1 h}$ for $h> \eps/\lambda_1 $.
Since $g\in L^1(\R_+)$ for $\eps<1$, we can use Jensen's inequality to show
\begin{align*}
\Vert N_1 \Vert _{D_\eps}^p &\leq \Big(\int_0^1  g(1-s)\Vert f(X_s)\Vert_{L^2} \,ds \Big)^p 
\lesssim \int_0^1  g(1-s)\Vert f(X_s)\Vert_{L^2}^p \,ds.
\end{align*} 
Therefore, $\E(\Vert N_1 \Vert _{D_\eps}^p)\lesssim \E(\Vert f(X_0)\Vert_{L^2}^p) \lesssim \E(\Vert f(X_0)\Vert_\infty^p) <\infty $ by Assumption (M) which shows the claims for $X_t$  and $N_t$. In  order to transfer the result to $f(X_t)$,  we estimate
\begin{align*}
\Vert f(X_t) \Vert_{D_\eps}^2\lesssim \Vert f(X_t) \Vert_{W^{2\eps,2}}^2 &= \Vert f(X_t) \Vert_{L^2}^2+\int_0^1 \int_0^1 \frac{(f(X_t(x))-f(X_t((y)))^2}{|x-y|^{1+4\eps}} \,dx dy \\
&\leq\Vert f(X_t) \Vert_{L^2}^2 + \Vert f'(X_t)\Vert_\infty ^2 \Vert X_t \Vert_{D_\eps}^2 \lesssim \Vert f(X_t) \Vert_{L^2}^2 + \Vert f'(X_t)\Vert_\infty ^4+ \Vert X_t \Vert_{D_\eps}^4,
\end{align*} 
from where the claim follows by Assumption (M) and (F)  in view of the first part of this proof. 
\end{proof}

The following lemma is useful for bounding the expression $\Vert \hat S(0)f(X_t)-f(X_t) \Vert_{L^2}^2$ appearing in the remainder term $ R_i$ from the regression model \eqref{eq:Reg_model_discrete}. Of particular interest is the situation where $\alpha$ is close to $1/4$ and, hence, the exponent $\frac{8\alpha^2}{4\alpha+1}$ can be chosen close to $1/4$.
\begin{lem} \label{lem:L2norm_approx2}
Let $H \in C^{2\alpha}([0,1]) \cap D_\alpha$ for some $\alpha \in (0,\frac{1}{2})$. Further, let $H^M := \sum_{k=1}^{M-1} H_k e_k$ where $H_k:=\langle H, e_k\rangle_M=\frac{1}{M}\sum_{l=1}^{M-1}H(y_l)e_k(y_l)$. Then, there exists a constant $C>0$ such that
$\Vert H-H^M \Vert_{L^2}^2\leq CK^2\delta^{\frac{8\alpha^2}{4\alpha+1}}$
where $K:=\max(\Vert H\Vert_{\infty},\Vert H\Vert_{C^{2\alpha}},\Vert H\Vert_{D_\alpha})$.
\end{lem}
\begin{proof}
First of all, by regarding $H_k$ as a Riemann sum, we can bound
\begin{align}
|H_k -h_k| &= \Big|\frac{1}{M}\sum_{l=1}^{M-1} H(y_l)e_k(y_l) -\int_0^1 H(y)e_k(y)\,dy \Big|\leq \sum_{l=0}^M \int_{y_l}^{y_{l+1}} |H(y_l)e_k(y_l)-H(y)e_l(y)|\,dy \nonumber\\
&\lesssim (\Vert e_k\Vert_\infty \Vert H\Vert_{C^{2\alpha}}  + \Vert H\Vert_\infty \Vert e_k\Vert_{C^{2\alpha}})\delta^{2\alpha}\lesssim ( \Vert H\Vert_{C^{2\alpha}}  + \Vert H\Vert_\infty k^{2\alpha})\delta^{2\alpha} \lesssim K \lambda_k^\alpha\delta^{2\alpha}. \label{eq:L2norm_approx2_eq1}
\end{align}
Similarly,  since {$\frac{1}{M} \sum_{k=1}^{M-1}H^2(y_k) = \Vert H^M\Vert_{L^2}^2 = \sum_{k=1}^{M-1} H_k^2$} holds by Lemma \ref{lem:L2norm_approx}, we have
\begin{align}
\Big|\Vert H^M\Vert_{L^2}^2 -\Vert H \Vert_{L^2}^2 \Big|&=
\Big|\frac{1}{M} \sum_{k=1}^{M-1}H^2(y_k) -\Vert H \Vert_{L^2}^2 \Big| \leq  \sum_{k=0}^{M-1} \int_{y_k}^{y_{k+1}} |H^2(y_k) - H^2(y)| \,dy \nonumber\\& \leq \Vert H^2\Vert_{C^{2\alpha}}\delta^{2\alpha} \leq 2 \Vert H\Vert_\infty \Vert H\Vert_{C^{2\alpha}}\delta^{2\alpha}\lesssim K^2 \delta^{2\alpha}. \label{eq:L2norm_approx2_eq2}
\end{align}
Also, note that for $h_k:=\langle H,e_k\rangle_{L^2}$ and any $R\in \N$, we have 
\begin{align}
\sum_{l\geq R} h_l^2 \leq \lambda_R^{-2\alpha}\sum_{l\geq R} \lambda_l^{2\alpha}h_l^2 \leq  \Vert H \Vert^2_{D_\alpha} \lambda_R^{-2\alpha} \lesssim K^2/R^{4\alpha}. \label{eq:L2norm_approx2_eq3}
\end{align}
The three inequalities just derived are now used to bound
$
\Vert H-H^M\Vert_{L^2}^2
 \leq |\Vert H^M\Vert_{L^2}^2-\Vert H\Vert_{L^2}^2|+2|\langle H-H^M,H\rangle_{L^2}|.
$
Due to \eqref{eq:L2norm_approx2_eq2}, the first term can be bounded by $K^2\delta^{2\alpha}\lesssim K^2\delta^{\frac{8\alpha^2}{4\alpha+1}}$ up to a constant.  
For the second term, using Parseval's identity, we get 
\begin{align*}
|\langle H-H^M,H\rangle_{L^2}|&=\Big| \sum_{l=1}^{M-1}(h_l-H_l)h_l + \sum_{l=M}^{\infty}h_l^2\Big|
\leq\Big| \sum_{l=1}^{M-1}(h_l-H_l)h_l\Big| + \sum_{l=M}^{\infty}h_l^2=: T_1 +T_2. 
\end{align*}
It follows directly from \eqref{eq:L2norm_approx2_eq3} that $T_2 \lesssim K^2/M^4\lesssim K^2\delta^{\frac{8\alpha^2}{4\alpha+1}}$.
To estimate $T_{1}$, we decompose
$T_1   \leq | \sum_{l=1}^{M_0-1}(h_l-H_l)h_l |+ | \sum_{l=M_0}^{M-1}(h_l-H_l)h_l |=: T_{11}+T_{12}$
for some intermediate value $M_0\in \{1,\ldots M-1\}$.
The Cauchy-Schwarz inequality, \eqref{eq:L2norm_approx2_eq1} and  \eqref{eq:L2norm_approx2_eq3} imply
\begin{align*}
T_{11}^2 &\leq \Big(\sum_{l=1}^{M_0-1}\lambda_l^{-2\alpha}(h_l-H_l)^2 \Big)\Big(  \sum_{l=1}^{M_0-1}\lambda_l^{2\alpha}h_l^2 \Big) \lesssim K^2 M_0 \delta^{4\alpha} \Vert H \Vert^2_{D_\alpha}  \lesssim K^4 M_0 \delta^{4\alpha},\\
T_{12}^2 &\leq  \sum_{l=M_0}^{M-1}(h_l-H_l)^2 \sum_{l=M_0}^{M-1}h_l^2 \lesssim (\Vert H\Vert_{L^2}^2+\Vert H^M\Vert_{L^2}^2)  \sum_{l=M_0}^{\infty}h_l^2 \lesssim K^4/ M_0^{4\alpha}.
\end{align*}
Balancing the bounds for $T_{11}$ and $T_{12}$ shows that it is optimal to take $M_0 \eqsim \delta^{-\frac{4\alpha}{4\alpha+1}} $ and, with this choice, we obtain the overall bound
$T_1 \lesssim K^2 \delta^{\frac{8\alpha^2}{4\alpha+1}}$ which finishes the proof.
\end{proof}


\begin{lem} \label{lem:auxAdapt}
Let $C\geq 1$ be a constant satisfying property \eqref{eq:(E)implication_cont}.
If $\pen(m')\ge400\log\big(12\sqrt{C}\big)\frac{\sigma^{2}D_{m'}}{T}$,
then  $\Gamma(m',m)$ from \eqref{eq:gamma} satisfies
\[
\E\Big(\Big(\Gamma(m',m)^{2}-\frac{1}{8}\pen(m')\Big)_{+}\1_{\Xi_{N,M,\bar{m}}}\Big)\le\frac{2\sigma^{2}}{T}\e^{-D_{m'}}.
\]
\end{lem}

\begin{proof}
\emph{Step 1.} For any $\tau,v>0$ and $g \in \bigcup_{m\in \N} \mathcal V_m$, we prove that

\begin{equation}
\P\Big(\frac{1}{N}\sum_{i=0}^{N-1}\langle\widehat{g(X_{t_{i}})},\varepsilon_{i}\rangle_{L^{2}}\ge\tau,\|g\|_{N,M}^{2}\le v^{2}\Big)\le \e^{-(T\tau^{2})/(2\sigma^{2}v^{2})}.\label{eq:adaptConc}
\end{equation}
Since 
$
\widehat{g(X_{t_{i}})}=\hat{S}(0)g(X_{t_{i}})=\sum_{\ell=1}^{M-1}\langle g(X_{t_{i}}),e_{\ell}\rangle_{M}e_{\ell}$ and $\varepsilon_{i}=\frac{\sigma}{\Delta}\int_{t_{i}}^{t_{i+1}}S(t_{i+1}-s)dW_{s},
$
Parseval's identity yields
\[
\frac{1}{N}\sum_{i=0}^{N-1}\langle\widehat{g(X_{i})},\varepsilon_{i}\rangle_{L^{2}}=\frac{1}{N}\sum_{i=0}^{N-1}\sum_{\ell=1}^{M-1}\langle g(X_{t_{i}}),e_{\ell}\rangle_{M}\frac{\sigma}{\Delta}\int_{t_{i}}^{t_{i+1}}\e^{-\lambda_{\ell}(t_{i+1}-s)}d\beta_{\ell}(s)=\frac{\sigma}{T}Y^{M}(T)
\]
with the Itô process $Y^{M}(s):=\sum_{\ell=1}^{M-1}Y_{\ell}(s),s\ge0,$
and, for $\ell=1,\dots,M-1$,
\[
Y_{\ell}(s):=\int_{0}^{s}H_{\ell}(h)d\beta_{\ell}(h),\qquad H_{\ell}(h):=\sum_{i=0}^{N-1}\langle g(X_{t_{i}}),e_{\ell}\rangle_{M}\e^{-\lambda_{\ell}(t_{i+1}-h)}\1_{[t_{i},t_{i+1})}(h).
\]
For each $\ell$, the processes $Y_{\ell}$ as well as $Y^M$ are martingales with quadratic
variation satisfying
\begin{align*}
\langle Y^{M}\rangle_{T}=\sum_{\ell=1}^{M-1}\langle Y_{\ell}\rangle_{T} & =\sum_{\ell=1}^{M-1}\sum_{i=0}^{N-1}\langle g(X_{t_{i}}),e_{\ell}\rangle_{M}^{2}\int_{t_{i}}^{t_{i+1}}\e^{-2\lambda_{\ell}(t_{i+1}-s)}ds\\
 & =\sum_{\ell=1}^{M-1}\sum_{i=0}^{N-1}\langle g(X_{t_{i}}),e_{\ell}\rangle_{M}^{2}\frac{1-\e^{-2\lambda_{\ell}\Delta}}{2\lambda_{\ell}}\le N\Delta\|g\|_{N,M}^{2}
\end{align*}
where we used Lemma \ref{lem:L2norm_approx} in the last inequality. 
For any $\lambda>0$, the process $(\exp(\lambda Y_{s}^{M}-\frac{\lambda^{2}}{2}\langle Y^{M}\rangle_{s}),s\ge0)$ inherits the martingale property from $Y^M$. Hence,
\begin{align*}
 & \P\Big(\frac{1}{N}\sum_{i=0}^{N-1}\langle\widehat{g(X_{t_{i}})},\varepsilon_{i}\rangle_{L^{2}}\ge\tau,\|g\|_{M,N}^{2}\le v^{2}\Big)\le\P\Big(\frac{\sigma}{T}Y^{M}(T)\ge\tau,\langle Y^{M}\rangle_{T}\le Tv^{2}\Big)\\
 & \qquad\le\P\Big(\exp\Big(\lambda Y^{M}(T)-\frac{\lambda^{2}}{2}\langle Y^{M}\rangle_{T}\Big)\ge\exp\Big(\lambda\frac{T\tau}{\sigma}-\frac{\lambda^{2}}{2}Tv^{2}\Big)\Big)
 \le\exp\Big(-\lambda\frac{T\tau}{\sigma}+\frac{\lambda^{2}}{2}Tv^{2}\Big)\Big).
\end{align*}
Choosing the minimizer $\lambda=\frac{\tau}{\sigma v^{2}}$ yields
(\ref{eq:adaptConc}). 

\emph{Step 2.} We use a chaining argument to deduce from Step~1 a
bound for (\ref{eq:oracleInequAdapt}): Due to the nesting assumption, we have on $\Xi_{N,M,\bar{m}}$ that
\[
\sup_{g\in \mathcal V_{m',m}:\|g\|_{N,M}=1}\frac{1}{N}\sum_{i=0}^{N-1}\langle\widehat{g(X_{t_{i}})},\varepsilon_{i}\rangle_{L^{2}}\lesssim\sup_{g\in \mathcal V_{m',m}:\|g\|_{L^{2}}=1}\frac{1}{N}\sum_{i=0}^{N-1}\langle\widehat{g(X_{t_{i}})},\varepsilon_{i}\rangle_{L^{2}}.
\]
Since the dimension of $\mathcal V_{m',m}$ is bounded by $D_{m'}+1$,
we can cover the $L^{2}$-unit ball in $\mathcal V_{m',m}$ with
a $\eps$-net $\mathcal{G}_{\eps}$ of the maximal size $(3/\eps)^{D_{m'}+1}$ \cite[Lemma 4.14]{Massart07}.
Considering a sequence $\mathcal{G}_{\eps_{k}}$ of $\eps_{k}$-nets
with $\eps_{k}=\eps2^{-k}$, $k\ge1$, and some $\eps>0$,
we denote for any $g\in \mathcal V_{m',m}$ by $\pi_{k}(g)$ the
closest element in $\mathcal{G}_{\eps_{k}}$ and set $\pi_{0}(g)=0$.
We obtain from $g=\sum_{k\ge1}(\pi_k(g)-\pi_{k-1}(g))$ and Lemma~\ref{lem:L2norm_approx} the decomposition
\[
\frac{1}{N}\sum_{i=1}^{N-1}\langle\widehat{g(X_{t_{i}})},\varepsilon_{i}\rangle_{L^{2}}=\frac{1}{N}\sum_{i=0}^{N-1}\sum_{k\ge1}\langle\hat{S}_{0}(\pi_{k}(g)-\pi_{k-1}(g))(X_{t_{i}}),\varepsilon_{i}\rangle_{L^{2}}.
\]
Under Assumption~(E) and on the event $\Xi_{N,M,\bar{m}}$, we have
$\|\pi_{k}(g)-\pi_{k-1}(g)\|_{N,M}\le \sqrt{3 C/2}\|\pi_{k}(g)-\pi_{k-1}(g)\|_{L^{2}}\le \sqrt{3 C/2}\eps(2^{-k}+2^{-k+1})\le4\sqrt C\eps2^{-k}$
for $C$ from \eqref{eq:(E)implication_cont}. Together with Step~2, we obtain for $\eps=C^{-1/2}/4$ and any sequence $(\tau_k)$ that
\begin{align*}
 & \P\Big(\Big\{\Gamma(m',m)\ge\sum_{k\ge1}2^{-k}\tau_{k}\Big\}\cap\Xi_{N,M,\bar{m}}\Big)\\
 & \le\sum_{k\ge1}\sum_{g_{k}\in\mathcal{G}_{\eps_{k}},g_{k-1}\in\mathcal{G}_{\eps_{k-1}}}\P\Big(\frac{1}{N}\sum_{i=0}^{N-1}\langle\hat{S}_{0}(g_{k}-g_{k-1})(X_{t_{i}}),\varepsilon_{i}\rangle_{L^{2}}\ge2^{-k}\tau_{k},\|g_{k}-g_{k-1}\|_{N,M}\le4\sqrt C\eps2^{-k}\Big)\\
 & \le\sum_{k\ge1}\sum_{g_{k}\in\mathcal{G}_{\eps_{k}},g_{k-1}\in\mathcal{G}_{\eps_{k-1}}}\e^{-T\tau_{k}^{2}/(2\sigma^{2})}
 =\sum_{k\ge1}\exp\big(-T\tau_{k}^{2}/(2\sigma^{2})+\log|\mathcal{G}_{\eps_{k}}|+\log|\mathcal{G}_{\eps_{k-1}}|\big).
\end{align*}
In view of $\log|\mathcal{G}_{\eps_{k}}|\le(D_{m'}+1)\log(3/\eps_{k})=(D_{m'}+1)(\log(3/\eps)+k\log2)\le2\log(3/\eps)(D_{m'}+1)k$
we choose $\tau_{k}^{2}=\frac{2\sigma^{2}}{T}(D_{m'}+\tau+4\log(3/\eps)(D_{m'}+2)k)$ for some $\tau >0$.
Then the above probability is bounded by 
$
\e^{-\tau-D_{m'}}\sum_{k\ge1}\e^{-4\log(3/\eps)k}\le \e^{-\tau-D_{m'}}.$

Owing to $(\sum_{k\ge1}2^{-k}\tau_{k})^{2}\le\sum_{k\ge1}2^{-k}\tau_{k}^{2}=\frac{2\sigma^2}{T}(D_{m'}+\tau+4\log(3/\eps)(D_{m'}+2)\sum_{k\ge1}k2^{-k})\le\frac{2\sigma^2\tau}{T}+\frac{50\sigma^2}{T}\log(3/\eps))D_{m'}$. 
We conclude that
\[
\P\Big(\Big\{\Gamma(m',m)^{2}-50\log(3/\eps)\frac{\sigma^{2}D_{m'}}{T}\ge\frac{2\sigma^{2}\tau}{T}\Big\}\cap\Xi_{N,M,\bar{m}}\Big)\le \e^{-\tau-D_{m'}}.
\]
If $\pen(m')\ge400\log(3/\eps)\frac{\sigma^{2}D_{m'}}{T}$, we obtain
\begin{align*}
\E\Big(\Big(\Gamma(m',m)^{2}-\frac{1}{8}\pen(m')\Big)_{+}\1_{\Xi_{N,M,\bar{m}}}\Big) & \le\int_{0}^{\infty}\P\Big(\Big\{\Gamma(m',m)^{2}-50\log(3/\eps)\frac{\sigma^{2}D_{m'}}{T}\ge r\Big\}\cap\Xi_{N,M,\bar{m}}\Big)dr\\
 & \le \e^{-D_{m'}}\int_{0}^{\infty}\e^{-Tr/(2\sigma^{2})}dr=\frac{2\sigma^{2}}{T}\e^{-D_{m'}}. \qedhere
\end{align*}
\end{proof}

\setlength{\bibsep}{0.2\baselineskip}
\bibliography{refs}
\bibliographystyle{apalike}
\end{document}